\documentclass[11pt,letterpaper]{amsart}

\usepackage[T1]{fontenc}
\usepackage{lmodern}
\usepackage{hyperref}
\usepackage{etex}
\usepackage[shortlabels]{enumitem}
\usepackage{amsmath}
\usepackage{amsxtra}
\usepackage{amscd}
\usepackage{amsthm}
\usepackage{amsfonts}
\usepackage{amssymb}
\usepackage{mathabx}
\usepackage{eucal}
\usepackage[all]{xy}
\usepackage{graphicx}
\usepackage{tikz-cd}
\usepackage{mathrsfs}
\usepackage{subfiles}
\usepackage{mathpazo}
\usepackage[colorinlistoftodos, textsize=tiny]{todonotes}
\usepackage{morefloats}
\usepackage{pdfpages}
\usepackage{thm-restate}
\usepackage[utf8]{inputenc}
\usepackage{epigraph}
\usepackage{csquotes}
\usepackage[margin=1in]{geometry}
\usepackage{adjustbox}
\usepackage{scalerel}
\usepackage{stackengine}
\usepackage{stmaryrd}
\usepackage{amssymb}
\usepackage{comment}

\stackMath
\newcommand\reallywidehat[1]{%
\savestack{\tmpbox}{\stretchto{%
  \scaleto{%
    \scalerel*[\widthof{\ensuremath{#1}}]{\kern-.6pt\bigwedge\kern-.6pt}%
    {\rule[-\textheight/2]{1ex}{\textheight}}%WIDTH-LIMITED BIG WEDGE
  }{\textheight}% 
}{0.5ex}}%
\stackon[1pt]{#1}{\tmpbox}%
}
\graphicspath{ {images/} }

\RequirePackage{color}
\definecolor{myred}{rgb}{0.75,0,0}
\definecolor{mygreen}{rgb}{0,0.5,0}
\definecolor{myblue}{rgb}{0,0,0.65}

\usepackage{color}

\usepackage{hyperref}
\hypersetup{citecolor=blue}
\usepackage{tikz}
\usetikzlibrary{matrix,arrows,decorations.pathmorphing}

\theoremstyle{plain}
\newtheorem{theorem}[subsubsection]{Theorem}

\newtheorem{thmx}{Theorem}
\newtheorem{proposition}[subsubsection]{Proposition}
\newtheorem{lemma}[subsubsection]{Lemma}
\newtheorem{corollary}[subsubsection]{Corollary}

\theoremstyle{definition}
\newtheorem{definition}[subsubsection]{Definition}
\newtheorem{remark}[subsubsection]{Remark}
\newtheorem{construction}[subsubsection]{Construction}
\newtheorem{example}[subsubsection]{Example}

\newtheorem{question}[subsubsection]{Question}
\newtheorem{conjecture}[subsubsection]{Conjecture}

\theoremstyle{remark}

\numberwithin{equation}{subsubsection}
\newcommand\nc{\newcommand}
\nc\on{\operatorname}
\nc\renc{\renewcommand}

\nc\mf\mathfrak
\nc\mc\mathcal
\nc\mb\mathbb
\nc\msf\mathsf
\nc\mscr\mathscr

\newcommand\dR{\on{dR}}
\newcommand\Dol{\on{Dol}}
\newcommand\FDol{\on{FDol}}
\newcommand\conj{\on{conj}}
\newcommand\abs{\on{abs}}
\newcommand\Hod{\on{Hod}}
\newcommand\tf{\on{tf}}
\newcommand\At{\on{At}}
\newcommand\KS{\on{KS}}
  \newcommand\eq{\on{eq}}

\DeclareMathOperator\spec{\text{Spec}}

\newcommand*{\shom}{\mathscr{H}\kern -.5pt om}
\newcommand*{\stor}{\mathscr{T}\kern -.5pt or}
\newcommand*{\sext}{\mathscr{E}\kern -.5pt xt}

\newcommand{\piconj}{\pi_{\conj}}
\newcommand{\PD}{\mathrm{PD}}
\newcommand{\sxspd}{\widehat{(S\times S)}_{\PD}}

\makeatletter
\providecommand\@dotsep{5}
\renewcommand{\listoftodos}[1][\@todonotes@todolistname]{%
\@starttoc{tdo}{#1}}
\makeatother

\makeatletter
\newcommand{\customlabel}[2]{\protected@write \@auxout {}{\string \newlabel {#1}{{#2}{\thepage}{#2}{#1}{}} }\hypertarget{#1}{#2}}

\DeclareMathOperator\id{id}

\renewcommand\hom{\mathrm{Hom}}

\DeclareMathOperator\bun{Bun}

\DeclareMathOperator\End{End}

\DeclareMathOperator\tr{tr}

\DeclareMathOperator\aut{Aut}
\DeclareMathOperator\gl{GL}

\renewcommand\sl{\mathrm{SL}}

\DeclareMathOperator\su{SU}

\DeclareMathOperator\ad{ad}

\DeclareMathOperator\et{\acute et}

\DeclareMathOperator\an{an}

\DeclareMathOperator\Sym{Sym}

\renewcommand\top{\mathrm{top}}
\DeclareFontFamily{U}{wncy}{}
\DeclareFontShape{U}{wncy}{m}{n}{<->wncyr10}{}
\DeclareSymbolFont{mcy}{U}{wncy}{m}{n}
\DeclareMathSymbol{\Sha}{\mathord}{mcy}{"58}

\def\listtodoname{List of Todos}
\def\listoftodos{\@starttoc{tdo}\listtodoname}

\title{$p$-curvature and non-abelian cohomology}

\author{Yeuk Hay Joshua Lam and Daniel Litt}

\usepackage{microtype}

\begin{document}

\begin{abstract} 
Let $X\to S$ be a smooth projective morphism. Katz proved the Grothendieck-Katz $p$-curvature conjecture for the Gauss-Manin connection on the $i$-th cohomology of $X/S$: if its $p$-curvature vanishes mod $p$ for infinitely many $p$, then the action of $\pi_1(S,s)$ on $H^i(X_s, \mathbb{Z})$ factors through a finite group. We prove a non-abelian analogue of this statement: if the $p$-curvature of the isomonodromy foliation on the moduli of flat bundles of rank $r$ on $X/S$ vanishes mod $p$ for infinitely many $p$, then the action of $\pi_1(S,s)$ on the rank $r$ integral characters of $\pi_1(X_s)$ factors through a finite group. We deduce many new cases of the Bost/Ekedahl--Shepherd-Barron--Taylor conjecture.

The proofs rely on a non-abelian version of Katz's formula, and a non-abelian version of the Hodge index theorem.
\end{abstract}

\maketitle

\setcounter{tocdepth}{1}
\tableofcontents

\section{Introduction}\label{sec:intro}
Let $R\subset \mathbb{C}$ be a finitely-generated $\mathbb{Z}$-algebra and let $f: X\to S$ be a smooth projective morphism of smooth $R$-schemes. There are close relationships between the topology of $f_{\mathbb{C}}^{\an}$ and its arithmetic (for example, via the Weil conjectures). Typically these relationships are mediated by the various realizations of the cohomology of the fibers of $f$, e.g.~singular cohomology $R^if^{\an}_*\underline{\mathbb{Z}}$, $\ell$-adic cohomology $R^if^{\et}_*\underline{\mathbb{Q}_\ell}$, or algebraic de Rham cohomology $R^if_*\Omega^\bullet_{X/S, \dR}$, and the compatibilities between them; here $\Omega^\bullet_{X/S, \dR}$ denotes the relative de Rham complex of $f$. Notable among these relationships---and the starting point for this paper---is Katz's resolution of the $p$-curvature conjecture for Gauss-Manin connections \cite{katz-p-curvature}, which asserts that the monodromy of $R^if^{\an}_*\underline{\mathbb{Z}}$ is controlled by the arithmetic of $R^if_*\Omega^\bullet_{X/S, \dR}$.

Our goal in this paper is to study relations  between the topology and arithmetic of $f$ mediated not by cohomology, but instead by \emph{non-abelian cohomology}, i.e.~the moduli of local systems on the fibers of $f$, in their various (Betti, de Rham, and Dolbeault) incarnations. As applications,  we prove a precise non-abelian analogue of Katz's theorem, as well as many new cases of  a  conjecture attributed to Bost and Ekedahl--Shepherd-Barron--Taylor on $p$-curvatures of foliations. Doing so requires the development of a number of structures on non-abelian cohomology, most notably the non-abelian counterparts to Katz's formula and  the Hodge index theorem.
\subsection{Main results (applications)}
Katz's theorem  is as follows:
\begin{theorem}[Katz {\cite{katz-p-curvature}}]\label{thm:katz-p-curvature-Gauss-Manin}
	Let $R\subset \mathbb{C}$ be a finitely-generated $\mathbb{Z}$-algebra, and let $f: X\to S$ be a smooth projective morphism of smooth $R$-schemes. Fix an integer $i\geq 0$ and denote by $(\mscr{E}, \nabla):= (R^if_*\Omega_{X/S, \dR}^\bullet, \nabla_{GM})$ the $i$-th relative de Rham cohomology of $f$, equipped with its Gauss-Manin connection. Fix a point $s\in S(\mathbb{C})$. Suppose that  the $p$-curvature of $(\mscr{E}, \nabla) \bmod p$ vanishes for infinitely many primes $p$. Then the monodromy action of $\pi_1(S(\mathbb{C})^{\an}, s)$ on $H^i(X_s(\mathbb{C})^{\an}, \mathbb{Z})$ factors through a finite group.
\end{theorem}
Here the statement that the $p$-curvature of $(R^if_*\Omega_{X/S, \dR}^\bullet, \nabla_{GM}) \bmod p$ vanishes means that if $x_1, \cdots, x_{\dim S}$ are a system of local coordinates on $S$, then $$\nabla_{GM}\left(\frac{\partial}{\partial x_i}\right)^p\equiv 0 \bmod p$$ for all $i$; here we are implicitly assuming that $p$ is sufficiently large so that we can reduce $R$, and therefore $f$,  modulo $p$. More generally, for any flat connection $(\mscr{E}, \nabla)$ on $X/S$ in characteristic $p$, one can define an invariant 
\[
\psi_p: F^*_{\abs}T_{X/S}\rightarrow \End(\mscr{E}),
\]
known as the $p$-curvature; one remarkable property of $\psi_p$ is that it is $\mscr{O}_X$-linear, whereas $\nabla: T_{X/S}\rightarrow \End(\mscr{E})$, an a priori similar looking map, is certainly not. We say that the  $p$-curvature of $(\mscr{E}, \nabla)$ vanishes if $\psi_p$ is the zero map. See the beginning of \autoref{subsec:M_conj} for the definition of $\psi$ and further generalities on it. 

We now describe our non-abelian analogue of \autoref{thm:katz-p-curvature-Gauss-Manin}. Given a smooth projective morphism of smooth $R$-schemes $f: X\to S$, we first discuss the analogues of the two main objects, namely  the relative de Rham cohomology $(\mscr{E}, \nabla)$ equipped with its Gauss-Manin connection, and Betti cohomology $H^i(X_s(\mb C)^{\an}, \mb Z)$ equipped with its $\pi_1(S(\mb C)^{\an})$-action.  The non-abelian de Rham cohomology is the moduli stack $\mathscr{M}_{\dR}(X/S, r)$ of rank $r$ flat vector bundles on $X/S$, with trivial determinant; it is equipped with a \emph{non-abelian Gauss-Manin connection} constructed by Simpson \cite[\S8]{simpson-moduli-representations-2}, or, equivalently, the  \emph{isomonodromy foliation}. This is a foliation on $\mscr{M}_{\dR}(X/S, r)$ whose complex-analytic leaves are the families of flat bundles whose monodromy representations are locally constant (up to conjugation). To sum up, the stack $\mathscr{M}_{\dR}(X/S, r)$ with its isomonodromy foliation is the non-abelian analogue of de Rham cohomology with its Gauss-Manin connection. 

Now fix $s\in S(\mathbb{C})$. The analogue of Betti cohomology is then $$M_B(X_s, r):=\hom(\pi_1(X_s(\mathbb{C})^{\an}), \sl_r)\sslash \sl_r,$$ the character variety of $\pi_1(X_s)$, parametrizing conjugacy classes of semisimple representations of $\pi_1(X_s)$ into $\sl_r$. There is a natural action of $\pi_1(S(\mathbb{C})^{\an}, s)$ on $M_B(X_s, r)$ (through its action by outer automorphisms on $\pi_1(X_s)$), which is the analogue of the $\pi_1(S(\mathbb{C})^{\an}, s)$-action on usual Betti cohomology. %$M_B(X_s, r)$ is the analogue here of singular cohomology, and the $\pi_1(S(\mathbb{C})^{\an},s)$-action is analogous to the monodromy representation.

With these analogies in mind, our first main theorem is a precise analogue of Katz's result \autoref{thm:katz-p-curvature-Gauss-Manin}:

\begin{thmx} \label{thm:main-theorem-integral-points}
	Let $R\subset \mathbb{C}$ be a finitely-generated $\mathbb{Z}$-algebra, and $f: X\to S$ a smooth projective morphism of smooth $R$-schemes. Fix a point $s\in S(\mathbb{C})$ and a positive integer $r\in \mathbb{Z}_{>0}$. If the $p$-curvature of the isomonodromy foliation on $\mathscr{M}_{\dR}(X/S, r)$ mod $p$ vanishes for infinitely many primes $p$, then the action of $\pi_1(S(\mathbb{C})^{\an},s)$ on $M_B(X_s, r)(\overline{\mathbb{Z}})$ factors through a finite group.
\end{thmx}
Here we say that an algebraic foliation $\mathscr{F}$ has \emph{vanishing $p$-curvature mod $p$} if, for every vector field $\vec v$ (in characteristic $p$) tangent to $\mathscr{F}\bmod p$, we have that $\vec v^p$ is tangent to $\mathscr{F}\bmod p$ as well, i.e.~$\mathscr{F}$ is $p$-closed. We denote by $\overline{\mathbb{Z}}$  the ring of algebraic integers.

\autoref{thm:katz-p-curvature-Gauss-Manin} has the immediate consequence that the Grothendieck-Katz $p$-curvature conjecture is true for Gauss-Manin connections. We briefly recall this conjecture and its analogue for foliations, and then explain our results towards a non-abelian $p$-curvature conjecture for isomonodromy foliations.

\begin{conjecture}[Grothendieck-Katz]\label{conj:gk}
	Let $R\subset \mathbb{C}$ be a finitely-generated $\mathbb{Z}$-algebra, and $S$ a smooth $R$-scheme. Let $(\mathscr{E}, \nabla)$ be a flat vector bundle on $S/R$. There exists a finite \'etale cover $S'$ of $S_{\mathbb{C}}$ trivializing $(\mathscr{E}, \nabla)$  if and only if $(\mathscr{E}, \nabla)$ mod $p$  has vanishing $p$-curvature for almost all primes $p$.
\end{conjecture}
Katz's \autoref{thm:katz-p-curvature-Gauss-Manin} implies that, if $(R^if_*\Omega_{X/S,\dR}^\bullet, \nabla_{GM})$ has vanishing $p$-curvatures mod $p$ for infinitely many $p$, then after passing to a finite cover, one may arrange that it has trivial monodromy. Then the conjecture follows for $(R^if_*\Omega_{X/S,\dR}^\bullet, \nabla_{GM})$ from the theory of regular singularities.

We now explain a version of the $p$-curvature conjecture for foliations, due to Ekedahl--Shepherd-Barron--Taylor and Bost \cite[p. 165]{bost-foliations}, which applies in our non-abelian setting.

\begin{conjecture}[Ekedahl--Shepherd-Barron--Taylor {\cite[Conjecture F]{esbt}}, Bost \cite{bost-foliations}]\label{conj:esbt}
	Let $R\subset \mathbb{C}$ be a finitely-generated $\mathbb{Z}$-algebra, and $M$ a smooth $R$-scheme. Let $\mathscr{F}\subset T_{M/R}$ be a  foliation, i.e.~a sub-bundle of Lie algebras. Then the complex-analytic leaves of $\mathscr{F}$ are analytifications of algebraic subvarieties of $M_{\mathbb{C}}$ if and only if $\mathscr{F}\bmod p$ is closed under $p$-th powers mod $p$ for almost all primes $p$, i.e.~it has vanishing $p$-curvature.
\end{conjecture} 
It is a pleasant exercise to check that \autoref{conj:esbt} implies \autoref{conj:gk} by taking $M$ to be the total space of the vector bundle $(\mathscr{E},\nabla)$, with foliation induced by $\nabla$. 

Our main result is that \autoref{conj:esbt} (or rather a mild stacky generalization thereof) is true for $\mathscr{M}_{\dR}(X/S)$, under hypotheses that are arguably mild.
\begin{thmx}\label{thm:main-theorem-esbt}
	With notation and assumptions as in \autoref{thm:main-theorem-integral-points}, suppose in addition that every irreducible component of $M_B(X_s, r)$ has a $\overline{\mathbb{Z}}$-point. Then the action of $\pi_1(S(\mathbb{C})^{\an},s)$ on  $M_B(X_s, r)(\mathbb{C})$ factors through a finite group. Consequently all leaves of the isomonodromy foliation are algebraic.
\end{thmx}
\begin{remark}
While we do not know how to verify the hypothesis that each irreducible component of $M_B(X_s, r)$ has a $\overline{\mathbb{Z}}$-point in general, it is true in important special cases. For example, we verify the Ekedahl--Shepherd-Barron--Taylor conjecture for the isomonodromy foliation on $\mathscr{M}_{\dR}(X/S, r)$ if:
\begin{enumerate}
	\item $r=2$, or
	\item $M_B(X_s, r)_{\mathbb{C}}$ is  irreducible, or
	\item $X/S$ is a relative curve.
\end{enumerate}
See \autoref{cor:main-theorem-esbt-cases}. 
\end{remark}

Note that the hypothesis that each irreducible component of $M_B(X_s, r)$ contains a $\overline{\mb Z}$-point  is a purely  topological one---that is, it depends only on $\pi_1(X_s)$. We conjecture that it is in fact always satisfied, which is a slight  extension of Simpson's integrality conjecture; see \cite[Conjecture 1.1.1]{coccia2025density}, \autoref{conj:integral-points}, and \autoref{rem:coccia}.
%\begin{theorem}\label{thm:main-theorem-esbt-cases}
%	Let $f: X\to S$ be a smooth projective morphism, and fix a point $s\in S(\mathbb{C})$. Suppose that the $p$-curvature of the isomonodromy foliation on $M_{\dR}(X/S, r)$ vanishes for almost all $p$, and either 
%	\begin{enumerate}
%		\item $r=2$, or
%		\item $X/S$ is a relative curve, or
%		\item $M_B(X_s,r)_{\mathbb{C}}$ is irreducible.
%	\end{enumerate}
%	Then the action of $\pi_1(S,s)$ on $M_B(X_s, r)(\mathbb{C})$ factors through a finite group.
%\end{theorem}

\subsection{Main structural results on non-abelian cohomology}
Katz's proof of \autoref{thm:katz-p-curvature-Gauss-Manin} relies on various structures on the cohomology of algebraic varieties and the interrelations between them. The main purpose of this work is to develop and exploit analogous structures in the non-abelian setting. In particular, we give non-abelian versions of both Katz's formula and the Hodge index theorem.
\subsubsection{Katz's proof of the $p$-curvature conjecture for Gauss-Manin connections---the  abelian setting}
We briefly recall the various structures that go into Katz's proof of \autoref{thm:katz-p-curvature-Gauss-Manin} and how they are used; we will explain the analogous results in our setting in \autoref{subsubsec:non-abelian-analogue}. We retain notation as in \autoref{thm:katz-p-curvature-Gauss-Manin}.

\emph{Step 1. (Arithmetic)} Katz first shows that if the $p$-curvature of $\nabla_{GM}$ mod $p$ vanishes for infinitely many $p$, then, in characteristic zero, $\nabla_{GM}$ preserves the Hodge filtration on $R^if_*\Omega_{X/S, \dR}^\bullet$. The primary input is a comparison, known as \emph{Katz's formula} \cite[Theorem 3.2]{katz-p-curvature}, between 
\begin{enumerate}
\item 	the associated graded of the $p$-curvature of $\nabla_{GM}$ mod $p$ with respect to the \emph{conjugate filtration}, and
\item the Frobenius pullback of the associated graded of $\nabla_{GM}$ itself with respect to the Hodge filtration, that is, the Kodaira-Spencer map.
\end{enumerate}
  This comparison itself has been the subject of much development, and is notably given a beautiful explanation in \cite{ogus-vologodsky}.

\emph{Step 2. (Geometry)} Katz next observes that if $\nabla_{GM}$ preserves the Hodge filtration on $(R^if_*\Omega_{X/S, \dR}^\bullet)_{\mathbb{C}}$ then the monodromy representation $$\rho: \pi_1(S(\mathbb{C})^{\an}, s)\to \gl(H^i(X_s, \mathbb{C}))$$ is unitary, i.e.~it preserves a definite Hermitian form. This is a consequence of the Hodge index theorem, i.e.~the fact that $(R^if_*\Omega_{X/S, \dR}^\bullet)_{\mathbb{C}}$ underlies a polarizable variation of Hodge structure.

\emph{Step 3. (Topology)} Finally, Katz makes use of the integral structure on singular cohomology, i.e. the fact that $\rho$ factors through a representation into $\gl(H^i(X_s, \mathbb{Z}))$. That is, $\rho$ factors as 
$$\xymatrix{
\pi_1(S(\mathbb{C})^{\an}, s)\ar[rr] \ar[rd] & & \gl(H^i(X_s, \mathbb{C}))\\
& U(r)\cap \gl(H^i(X_s, \mathbb{Z})/\text{tors}).\ar[ur] &
}$$
But this intermediate group is both discrete and compact, hence finite, completing the proof.
\subsubsection{A non-abelian analogue---Leitfaden}\label{subsubsec:non-abelian-analogue}
The main idea of our proofs of \autoref{thm:main-theorem-integral-points} and \autoref{thm:main-theorem-esbt} is in some sense to simply copy Katz's proof glossed above. The primary challenge is to make sense of the non-abelian analogues of the objects Katz encounters, e.g.~the Hodge and conjugate filtrations, the Kodaira-Spencer map, the $p$-curvature, the Hodge index theorem, and the various compatibilities between them. We now explain these analogues and how they are used.

\emph{Step 1. (Arithmetic---a non-abelian Katz formula)}
  Recall that Katz compared (via \emph{Katz's formula} \cite[Theorem 3.2]{katz-p-curvature}) the Frobenius pullback of $\on{gr}_{F_{\Hod}} \nabla_{GM}$ with $\on{gr}_{\conj} \psi_p(\nabla_{GM})$, the associated graded of the $p$-curvature $\psi_p(\nabla_{GM})$ with respect to the conjugate filtration. To make sense of this in the non-abelian setting, we need 
  \begin{enumerate} 
  	\item to find analogues of the Hodge, resp.~ conjugate filtrations,
	\item to understand how to take the associated graded of the isomonodromy foliation, resp.~its $p$-curvature with respect to these, and
	\item to compare these associated gradeds up to Frobenius twist. 
  \end{enumerate}
  The analogue of a \emph{filtration} on $\mathscr{M}_{\dR}(X/S)$ is well-known (see e.g.~\cite{simpson1996hodge}); it is a $\mathbb{G}_m$-equivariant stack over $\mathbb{A}^1_S$ equipped with an isomorphism between its fiber over $1$ and $\mathscr{M}_{\dR}(X/S)$, analogous to the Rees construction. The fiber over $0$ is referred to as the \emph{associated graded}. The Hodge filtration has been understood (since work of Simpson and Deligne) to be analogous to the moduli stack of $\lambda$-connections $\mathscr{M}_{\Hod}(X/S)/\mathbb{A}^1_S$, recalled in \autoref{defn:M_Hod}.  In this case the associated graded---the $0$-fiber---is the moduli stack $\mathscr{M}_{\Dol}(X/S)$ of Higgs bundles on $X/S$
  
  For a morphism $f: X\to S$ in characteristic $p>0$ the analogue of the conjugate filtration is a stack we denote by $\mathscr{M}_{\conj}(X/S)/\mathbb{A}^1_S$, introduced in \autoref{defn: mconj}; it is a slight variant of a stack previously studied by Menzies \cite{menzies2019p} in a closely related context. The associated graded is the moduli stack of Higgs bundles on $X^{(1)}/S$; here $X^{(1)}$ denotes the pullback of $X$ along absolute Frobenius on $S$.
  
  The analogue of the Kodaira-Spencer map $\on{gr}_{F_{\Hod}} \nabla_{GM}$ is an object on $\mathscr{M}_{\Dol}(X/S)$ referred to as a \emph{lifting of tangent vectors}, previously studied by a number of authors, e.g.~\cite{chen2012associated, collier2025higgs, sheng2025nonlinearharmonic, sheng2025nonlinear, fu2025nonabelian}, which one may loosely speaking understand as the obstruction to extending the isomonodromy foliation on $\mathscr{M}_{\dR}(X/S)$ over the central fiber $\mathscr{M}_{\Dol}(X/S)$ of $\mathscr{M}_{\Hod}(X/S)$, just as the Kodaira-Spencer map is the obstruction to $\nabla_{GM}$ preserving the Hodge filtration. This is explained in \autoref{rem:theta-is-KS}. We denote this lifting of tangent vectors by $\Theta_{X/S}$.
  
  Just as the conjugate filtration on de Rham cohomology is preserved by the Gauss-Manin connection, the isomonodromy foliation extends over all of $\mathscr{M}_{\conj}(X/S)/\mathbb{A}^1_S$. Its $p$-curvature $\Psi_{\conj}$ vanishes at the zero fiber, analogous to the fact that the $p$-curvature of the Gauss-Manin connection vanishes on the associated graded of the conjugate filtration. Our version of Katz's formula is:
  
\begin{thmx}[Non-abelian Katz formula]\label{thm:intro-katz-formula}
	There is an identification between $\frac{1}{\lambda}\Psi_{\conj}|_{\mathscr{M}_{\conj}(X/S)_0}$ and the Frobenius pullback of  $\Theta_{X/S}$, where $\lambda$ is the coordinate on $\mathbb{A}^1_S$.
\end{thmx}
 See \autoref{subsec:non-abelian-katz-formula} and in particular \autoref{lemma:katz-formula} for a precise statement. We think it is plausible that this formula is implicit in some modern developments in $p$-adic Hodge theory and its characteristic $p$ avatars, though some work would be required to extract it.
    
\begin{remark}
 In Katz's proof, one deduces  immediately from Katz's formula that the vanishing of infinitely many $p$-curvatures implies the vanishing of the Kodaira--Spencer map, i.e. the Higgs field. This is a consequence of Hodge-de Rham degeneration, and subsequently  the degeneration of the conjugate spectral sequence.
 
 %It is unfortunately not a formal consequence of this formula that, if $\Psi_{\conj}$ vanishes mod $p$ for infinitely many $p$, the same is true for $\Theta_{X/S}$ in characteristic zero, 
 
 In our setup, we would similarly like to conclude that  $\Theta_{X/S}$ vanishes in characteristic zero, given that  $\Psi_{\conj}$ vanishes mod $p$ for infinitely many $p$. However, things are rather more complicated in our non-abelian situation: we do not know what the non-abelian analogue of the degeneration of the conjugate spectral sequence is, and more generally the geometry of $\mathscr{M}_{\conj}(X/S)$ is poorly understood. We get around this and prove the statement, using Ogus-Vologodsky's non-abelian Hodge theory in positive characteristic \cite{ogus-vologodsky} to construct canonical sections to $\mathscr{M}_{\conj}(X/S)/\mathbb{A}^1_S$, which we use to prove a kind of isosingularity principle for $\mathscr{M}_{\conj}(X/S)$, in \autoref{subsec:NAHTcharp}. This argument has some precursor in Menzies' thesis \cite{menzies2019p}, who proved a related result under the additional assumption that $X/S$ admits a Frobenius lift.
 \end{remark}
  
  Ultimately we conclude in \autoref{thm:p-curvature-vanishing-implies-theta-vanishing} that under the hypotheses of \autoref{thm:main-theorem-integral-points}, the lifting of tangent vectors $\Theta_{X/S}$ vanishes identically: loosely speaking, the isomonodromy foliation preserves the Hodge filtration. 
  
\emph{Step 2. (Geometry---a non-abelian Hodge index theorem)} 
Having proven that $\Theta_{X/S}$ vanishes, we wish to show that the action of $\pi_1(S(\mathbb{C})^{\an}, s)$ on $M_B(X_s, r)$ is \emph{unitary}; in the abelian setting this was a consequence of the Hodge index theorem. In \autoref{sec:NAHI},  we prove \emph{two} analogues of the Hodge index theorem.

The first says:
\begin{thmx}[Non-abelian Hodge index theorem I]\label{thm:intro-hodge-1}
	If $\Theta_{X/S}$ vanishes on $\mathscr{M}_{\Dol}(X/S, r)$, then the orbits of $\pi_1(S(\mathbb{C})^{\an},s)$ on $M_B(X_s, r)(\mathbb{C})$ have compact closure (in the analytic topology).
\end{thmx}
See \autoref{thm:compactness-thm} for a precise statement. One should think of this as analogous to the unitarity of a variation of Hodge structure with vanishing Higgs field: if $$\rho: \Gamma\to GL_r(\mathbb{C})$$ is unitary, then the $\Gamma$-orbit of any vector in $\mathbb{C}^r$ has compact closure.

The second analogue of the Hodge index theorem is:
\begin{thmx}[Non-abelian Hodge index theorem II]\label{thm:intro-hodge-2}
	Under the same assumptions as \autoref{thm:intro-hodge-1}, the action of $\pi_1(S(\mathbb{C})^{\an},s)$ on the reduced smooth locus of $M_B(X_s, r)(\mathbb{C})$ preserves a natural Riemannian metric.
\end{thmx}
See \autoref{thm:kahler-preserved} and \autoref{thm:triholomorphic} for more precise statements. In particular, at a fixed point  of the $\pi_1(S(\mb C)^{\an}, s)$-action on  $M_B(X_s, r)(\mathbb{C})$, the action on the tangent space is in fact unitary.

Both of these results rely on Simpson's complex non-abelian Hodge theory and ultimately have the same source, \autoref{lemma:energy-constant}: if $\Theta_{X/S}$ vanishes along a leaf of the isomonodromy foliation, then the \emph{energy functional} is constant along that leaf. This result appears closely related to \cite[Theorem E]{collier2025higgs}, which appeared while this paper was in preparation.

\emph{Step 3. (Topology)} We now observe that for a $\mathbb{Z}$-point of $M_B(X_s, r)$, its orbit under $\pi_1(S(\mathbb{C})^{\an},s)$ is discrete (as the $\mathbb{Z}$-points of an affine algebraic variety are discrete in its complex points) and compact, by \autoref{thm:compactness-thm}. The case of $\overline{\mathbb{Z}}$-points is handled analogously, by considering all Galois conjugates at once. This is done in \autoref{thm:main-theorem-integral-points-weak}.

The proof of \autoref{thm:main-theorem-integral-points} and \autoref{thm:main-theorem-esbt} now uses an analysis of the action on the tangent spaces of integral points combined with some Riemannian geometry, again relying on the integral structure and the unitarity from \autoref{thm:triholomorphic}. This last step is performed in \autoref{subsec:esbt}.

\begin{table}[t]
\centering
\small
\setlength{\tabcolsep}{6pt}
\renewcommand{\arraystretch}{1.5}
\begin{tabular}{|p{0.475\textwidth}|p{0.475\textwidth}|}
\hline
\textbf{Abelian (cohomological) setting} & \textbf{Non-abelian setting} \\
\hline
 $R^if^{\on{an}}_*\underline{\mathbb{Z}}$ on $S(\mathbb{C})$, with monodromy
$\pi_1(S(\mathbb{C})^{\an},s)\to GL(H^i(X_s,\mathbb{Z}))$ &
 $M_B(X_s,r)$, with its natural $\pi_1(S(\mathbb{C})^{\an},s)$-action (via the outer automorphism action on $\pi_1(X_s)$) \\
\hline
 $R^if_*\Omega^\bullet_{X/S,\dR}$ with Gauss--Manin connection $\nabla_{GM}$ &
$\mathscr{M}_{\dR}(X/S,r)$ with the isomonodromy foliation \\
\hline
Hodge filtration $F_{\Hod}$ on $R^if_*\Omega^\bullet_{X/S,\dR}$, with associated graded $\bigoplus_{p+q=i} R^pf_*\Omega^q_{X/S}$ &
$\mathscr{M}_{\Hod}(X/S,r)\to \mathbb{A}^1_S$ with $0$-fiber $\mathscr{M}_{\Dol}(X/S,r)$\\
\hline
Conjugate filtration on $R^if_*\Omega^\bullet_{X/S,\dR}$ in char.~$p>0$, with associated graded $\bigoplus_{p+q=i} R^qf_*\Omega^p_{X^{(1)}/S}$ &
 $\mathscr{M}_{\conj}(X/S,r)\to \mathbb{A}^1_S$ with $0$-fiber $\mathscr{M}_{\Dol}(X^{(1)}/S,r)$. \\
\hline
Kodaira--Spencer map 
$\on{gr}_{F_{\Hod}}\nabla_{GM}$, i.e. the Higgs field &
Lifting of tangent vectors $\Theta_{X/S}$  \\
\hline
$p$-curvature $\psi_p(\nabla_{GM})$ &
$p$-curvature of isomonodromy foliation $\Psi$  \\
\hline
\emph{Katz's formula} \cite[Thm.\ 3.2]{katz-p-curvature}:
 comparison between $\on{gr}_{\conj}\psi_p(\nabla_{GM})$ and $F_{\abs}^*(\on{gr}_{F_{\Hod}}\nabla_{GM})$. &
\autoref{thm:intro-katz-formula},
 comparing $\frac{1}{\lambda}\Psi_{\conj}|_{\mathscr{M}_{\conj}(X/S)_0}$ with the Frobenius pullback of $\Theta_{X/S}$. \\
\hline
Hodge index theorem  & \autoref{thm:intro-hodge-1} and \autoref{thm:intro-hodge-2} \\
\hline
\end{tabular}
\caption{Dictionary between Katz's proof of \autoref{thm:katz-p-curvature-Gauss-Manin} and our results}
\label{tab:intro-analogies}
\end{table}

\begin{remark}
We have opted to state our results in the setting of a smooth projective morphism $f: X\to S$. It is natural to ask about the quasi-projective setting as well. Indeed, suppose we are given a smooth projective morphism $f$ as above, and $X/S$ is equipped with a relative SNC divisor $D$. We expect that everything in this paper works more or less verbatim if one considers flat bundles on $X/S$ with regular singularities along $D$ and fixed semisimple residues, with eigenvalues in $\mathbb{Q}$ (i.e.~semisimple, finite order, local monodromy). We have opted not to try to write the paper in this generality both for simplicity of notation and because the analytic preliminaries in \autoref{sec:NAHI} are more involved in this case.

On the other hand, one pleasant aspect of this situation is that the hypotheses of \autoref{thm:main-theorem-esbt} on existence of integral points may be verified whenever the moduli of local systems with fixed local monodromy along $D$ (i.e.~the relative character variety) is irreducible, by the main result of \cite{de2024integrality}.
\end{remark}

\subsection{Questions}
There are a number of interesting questions raised by this work. The first is purely a question of complex algebraic geometry:
\begin{conjecture}\label{conj:theta-vanishing}
	Let $f: X\to S$ be a smooth projective morphism of smooth complex varieties, and fix $s\in S$. If $\Theta_{X/S}\equiv 0$ on $\mathscr{M}_{\Dol}(X/S, r)$, then the action of $\pi_1(S(\mathbb{C})^{\an},s)$ on $M_B(X_s, r)(\mathbb{C})$ factors through a finite group.
\end{conjecture}

Note that if $f$ is a relative curve, we in fact show in \autoref{thm:relative-curves} that the vanishing of $\Theta_{X/S}$ implies that $X/S$ is isotrivial.

The proof of \autoref{thm:main-theorem-esbt} (see \autoref{thm:theta-vanishing}) reduces \autoref{conj:theta-vanishing}  to an \emph{arithmetic} conjecture on the existence of enough integral points on $M_B(X_s, r)$; it is a weak form of \cite[Conjecture 1.1.1]{coccia2025density}.
\begin{conjecture}\label{conj:integral-points}
	Let $X$ be a smooth projective complex variety. Then each component of $M_B(X, r)$ has a $\overline{\mathbb{Z}}$-point.
\end{conjecture}
This conjecture is a strengthening of Simpson's integrality conjecture, which is the same statement for $0$-dimensional components. It is known for reduced $0$-dimensional components \cite{eg-selecta} and if $r=2$ \cite{coccia2025density}.

Note that \autoref{conj:theta-vanishing} appears closely related to \cite[Conjecture B]{hu2025kodaira}, which appeared while this paper was in preparation, and in particular that conjecture also follows from  \autoref{conj:integral-points}, by the methods of this paper. Indeed, by e.g. \autoref{thm:triholomorphic} or~\cite[Theorem E]{collier2025higgs}, the hypothesis of \cite[Conjecture B]{hu2025kodaira} is the same as that of our \autoref{conj:theta-vanishing}, and the conclusion follows from \autoref{thm:theta-vanishing} if one assumes \autoref{conj:integral-points}.

\autoref{conj:theta-vanishing} would follow from some compatibility between complex non-abelian Hodge theory and the theory of harmonic maps to buildings associated to $p$-adic groups, due to Gromov--Schoen \cite{gromov1992harmonic} and others. In particular, by a $p$-adic analogue of \autoref{lemma:energy-constant}, it would suffice to know the following:
\begin{question}\label{question: hopf-diff}
	Let $X$ be a smooth projective variety over $\mb{C}$, and $K$ a $p$-adic field. Is the set of Hopf differentials (see e.g.~\cite[p.~154]{schoen}) arising from harmonic maps from $X$ to the building for $\sl_n(K)$ contained in the span of the image of the quadratic Hitchin map for $\sl_n$-Higgs bundles on  $X$?
\end{question}

A related question was asked by He, Liu, and Mok in  \cite[Question 4.5]{he2024spectral}. In the formulation of \autoref{question: hopf-diff}, the quadratic Hitchin map refers to  the part of the Hitchin map, i.e. the map giving rise to the Hitchin fibration, landing in quadratic differentials on $X$--see \autoref{section: chen's-formula} for the definition in the case of curves, and the  case of a general smooth projective variety  is identical.
\subsection{Related works}
This work takes inspiration from a number of authors.  Most relevant are the theses of Papaioannou \cite{papaioannou2013algebraic} and Menzies \cite{menzies2019p}, in which they studied precisely the Bost/Ekedahl--Shepherd-Barron--Taylor conjecture for the isomonodromy foliation. The work of Shankar \cite{shankar-p-curvature} studied the $p$-curvature conjecture when one degenerates a curve to a nodal curve. These were all PhD theses written under Mark Kisin's supervision.  As far as we understand, those works are part of a broader strategy of Kisin's to probe the Grothendieck--Katz $p$-curvature conjecture via isomonodromy: that is, the idea is to study isomonodromic deformations of the flat bundle as one varies the complex structure of the base, under the assumptions of vanishing $p$-curvatures. The paper \cite{patel2021rank} also studies the $p$-curvature conjecture in the context of isomonodromy.
The above-mentioned works, as well as \cite{katz-p-curvature} of course, were the starting point for this paper.

More broadly, this work is part of a large literature on the $p$-curvature conjecture. Aside from Katz's work, significant work on the conjecture was done by the Chudnovsky brothers \cite{chudnovsky2006applications}, Andr\'e \cite{andre2004conjecture}, and Bost \cite{bost-foliations}. Bost's work \cite{bost-foliations} (as well as \cite{esbt}) in particular pioneered the study of the $p$-curvature conjecture in the context of foliations.

The philosophy that one should view moduli spaces such as $\mscr{M}_{\dR}$ and $\mscr{M}_{B}$ as non-abelian analogues of cohomology was pioneered by Simpson and Deligne, especially through Simpson's construction \cite{simpson-homotopy-de-rham} of the \emph{de Rham stack}.  The latter has been  a great source of inspiration in recent studies of cohomology theories of algebraic varieties, notably in the works of  Bhatt, Drinfeld, Lurie, Scholze, and that of  many others \cite{drinfeld2018stacky, bhatt-f-gauge, bhatt-lurie2022prismatization, scholze-six-functor}; this approach has been coined  \emph{transmutation} by Bhatt. As explained in \autoref{subsubsec:non-abelian-analogue}, we expect that part of \autoref{thm:intro-katz-formula} is implicit in such works. 

The complex-analytic aspects of this paper are related to a number of recent works, notably \cite{chen2012associated, collier2025higgs, sheng2025nonlinearharmonic, fu2025nonabelian, hu2025kodaira}, as explained in  \autoref{subsubsec:non-abelian-analogue}.

This paper is also in some sense a sequel to \cite{p-painleve}. There, we formulated a conjectural arithmetic characterization of algebraic leaves   for arbitrary foliations, or equivalently algebraic solutions to (possible non-linear) differential equations;  this conjecture would yield a proof of the $p$-curvature conjecture via isomonodromy, as envisioned by Kisin.

The main result of \cite{p-painleve} is the proof of this conjecture for the isomonodromic foliation at Gauss-Manin connections. We think of  \cite{p-painleve} and the current paper as part of a general program aimed at understanding the geometry and arithmetic of ``motivic subvarieties" of moduli of local systems---\cite{p-painleve} studies the $p$-curvature at ``motivic points",  while this paper focuses on the opposite extreme---the $p$-curvature on the  entire moduli space. Motivated by a question of Michael Groechenig, we explain how to deduce a consequence of \autoref{thm:main-theorem-esbt} from the results of \cite{p-painleve} in \autoref{subsec:motivic-points}, as well as  obstructions to going further with those methods.

\subsection{Outline} We give a brief summary of each of the sections. We begin in \autoref{sec:filtrations} by defining precisely the stacks which will be the main players of this work, namely the de Rham, Dolbeault, Hodge, and conjugate moduli stacks $\mscr{M}_{\dR}, \mscr{M}_{\Dol}, \mscr{M}_{\Hod}, \mscr{M}_{\conj}$; the latter two are viewed as filtrations on $\mscr{M}_{\dR}$. We also collect various properties of these stacks, such as Hitchin maps and $\mb{G}_m$-actions. In \autoref{sec:NAGM} we recall the construction of the non-abelian Gauss-Manin connection on $\mscr{M}_{\dR}$, and extend it to $\mscr{M}_{\conj}$. We introduce the lifting of tangent vectors $\Theta$ on $\mscr{M}_{\Dol}$, which is the non-abelian analogue of the Higgs field. We then study the $p$-curvatures of the non-abelian connections on $\mscr{M}_{\dR}$ and $\mscr{M}_{\conj}$, and prove the non-abelian Katz formula \autoref{lemma:katz-formula}, the main result of this section.   

In \autoref{sec: vanish-p-curv}, we construct \emph{canonical sections} on $\mscr{M}_{\conj}$ passing through nilpotent Higgs bundles (of small order of nilpotence); this should be viewed as the $\bmod p$ analogue of canonical sections on $\mscr{M}_{\Hod}$ constructed by Simpson using non-abelian Hodge theory. Using these, we deduce that the vanishing of the non-abelian $p$-curvature for infinitely many $p$ implies the vanishing of the lifting of tangent vectors. In \autoref{sec:NAHI} we prove our versions of the Hodge index formula by studying  the energy functional on the moduli space of Higgs bundles; one consequence is that, assuming the vanishing of $\Theta$, we prove that the $\pi_1(S(\mb{C})^{\an}, s)$-action on the moduli of Higgs bundles is \emph{tri-holomorphic}, i.e. the action is holomorphic in all three natural complex structures of this hyperk\"ahler manifold. This section is of a purely complex analytic nature. Finally in \autoref{sec:finiteness-na-monodromy} we combine the ingredients and prove our main theorems.

%\subsection{Notation}
%\begin{itemize}
%	\item $A$ is a base ring
%	\item $f: X\to S$ is our smooth projective morphism
%	\item Points to test with are $T$-points if no tangent bundles involved, $S'$-points or $R$-points otherwise.
%	\item $\mathscr{M}_{?}$ is a moduli stack, $M_{?}$ is a good moduli space, where $?=\\conj, \dR, \Dol, \Hod,$ etc. $\pi: \mscr{M}_{\dR}\rightarrow S$, $\pi_{\Hod}: \mscr{M}_{\Hod}\rightarrow S$, $\pi_{\conj}: 
%    \mscr{M}_{\conj}\rightarrow S\times \mb{A}^1$.
%	\item $M_{?}$ is the good moduli space
%\end{itemize}

\subsection{Acknowledgments} To be added after the referee process is complete.
%Brunebarbe, Esnault, Kisin, Klingler, Groechenig, Menzies, Pappaianou, Petrov, Vologodsky, ...
\section{The non-abelian Hodge and conjugate filtrations}\label{sec:filtrations}
In this section we introduce the non-abelian analogues of algebraic de Rham and Hodge cohomologies,   the Hodge and conjugate filtrations on the former, as well as  some related structures; this point of view was pioneered by Simpson. Throughout we let $S$ be a scheme and $f: X\to S$ a smooth projective morphism. 
\subsection{Flat bundles and the de Rham moduli stack}\label{subsec:M_dr}
The non-abelian analogue of de Rham cohomology is the \emph{de Rham moduli stack} of $X/S$, defined as follows.
\begin{definition}\label{defn: mdr}
	The \emph{de Rham moduli stack} $\mathscr{M}_{\dR}(X/S)$ is the $S$-stack whose $T$-points are the groupoid of vector bundles with flat connection and trivialized determinant on $X_T/T$, i.e.~triples $(\mathscr{E},\nabla, \xi)$ with $\mathscr{E}$ a vector bundle on $X_T$, 
    \begin{equation}\label{eqn: flat-conn} 
    \nabla: \mathscr{E}\to \mathscr{E}\otimes \Omega^1_{X_T/T}
    \end{equation}
    a flat connection,
	and $\xi: \det(\mathscr{E},\nabla)\overset{\sim}{\to} (\mathscr{O}_{X_T}, d)$ an isomorphism where the latter is the trivial rank one flat connection. Recall that, on the vector bundle $\mscr{E}_T$, a flat connection  relative to $T$ is a $\mathscr{O}_T$-linear map $\nabla$ as in \eqref{eqn: flat-conn} such that
	\begin{enumerate}
		\item $\nabla(fs)=f\nabla(s)+s\otimes df$ for $f$ a local section of $\mathscr{O}_{X_T}$ and $s$ a local section of $\mathscr{E}$, and
		\item $\nabla\circ\nabla: \mscr{E}\rightarrow \mscr{E}\otimes \Omega^2_{X_T/T}$ vanishes.
	\end{enumerate}
\end{definition} 
\subsection{The Dolbeault moduli stack}\label{subsec:M_Dol}
We now move onto the non-abelian analogue of Hodge cohomology. 
\begin{definition}
	The \emph{Dolbeault moduli stack} $\mathscr{M}_{\Dol}(X/S)$ is the $S$-stack whose $T$-points are the groupoid of Higgs bundles with trivialized determinant on $X_T/T$, i.e.~triples $(\mathscr{E}, \theta, \xi)$ with $\mathscr{E}$ a vector bundle on $X_T$, $$\theta: \mathscr{E}\to \mathscr{E}\otimes \Omega^1_{X_T/T}$$ an $\mathscr{O}_{X_T}$-linear map with $$\theta\circ \theta: \mathscr{E}\to \mathscr{E}\otimes \Omega^2_{X_T/T}$$ identically zero, and $\xi: \det(\mathscr{E},\theta)\overset{\sim}{\to} (\mathscr{O}_{X_T}, 0)$. We refer to $\theta$ as the \emph{Higgs field} on $\mathscr{E}$. There is a natural $\mb{G}_{m, S}$-action on $\mscr{M}_{\Dol}(X/S)$, obtained by scaling $\theta$.
\end{definition}

\subsubsection{The Hitchin morphism}\label{section: hitchin-map}
Let $\mscr{M}_{\Dol}(X/S, r)\subset \mscr{M}_{\Dol}(X/S)$ be the substack of rank $r$ Higgs bundles. The Hitchin morphism is the map
\begin{align*}
    h: \mscr{M}_{\Dol}(X/S, r)&\rightarrow \bigoplus_{i=2}^r f_*\Sym^i\Omega^1_{X/S} \\
    [(\mscr{E}, \theta)] &\mapsto (h_2, \cdots, h_r).
\end{align*}
where the $h_i$'s are the coefficients of the characteristic polynomial $\det(t-\theta)$; that is,
\[
\det(t-\theta)=t^r+h_2t^{r-2}+\cdots+h_r.
\]

\subsection{Filtrations on non-abelian cohomology}\label{subsec:filtrations}
Recall that there is a natural equivalence of categories between vector spaces with a finite, exhaustive filtration over a field $k$ and $\mathbb{G}_m$-equivariant vector bundles on $\mathbb{A}^1_k$. Here if $(V, F^\bullet)$ is a filtered vector space and $\mathscr{E}_{(V, F^\bullet)}$ the $\mathbb{G}_m$-equivariant vector bundle on $\mathbb{A}^1$ associated to it via the Rees construction, there are natural identifications $$\mathscr{E}_{(V, F^\bullet)}|_1\overset{\sim}{\to} V,\  \mathscr{E}_{(V, F^\bullet)}|_0\overset{\sim}{\to} \on{gr}_{F^\bullet}V.$$ (See e.g.~\cite[\S5]{simpson1996hodge}, as well as \cite[Prop. 2.2.6]{bhatt-f-gauge} for a more general statement on the filtered derived category.) Thus, given a $\mathbb{G}_m$-equivariant $\mathbb{A}^1$-scheme or stack $\mathscr{M}$ and an identification $\mathscr{M}_1\overset{\sim}{\to} X,$ we will refer to $\mathscr{M}$ as a \emph{filtration} on $X$ and $\mathscr{M}_0$ as the \emph{associated graded of the filtration}.

\subsubsection{The Hodge moduli stack}\label{subsec:M_Hod}
We first discuss the analogue of the Hodge filtration on de Rham cohomology. 
\begin{definition}\label{defn:M_Hod}
	The \emph{Hodge moduli stack} $\mathscr{M}_{\Hod}(X/S)$ is the $\mathbb{A}^1_S$-stack whose $T$-points (over a map $\lambda: T\to \mathbb{A}^1_S$) are pairs $(\mathscr{E}, \gamma)$ with trivialized determinant,
	 where $\mathscr{E}$ is a vector bundle on $X_T$ and $$\gamma: \mathscr{E}\to \mathscr{E}\otimes\Omega^1_{X_T/T}$$ is a flat $\lambda$-connection on $\mathscr{E}$ over $S$, i.e.~for local sections $s$ of $\mathscr{E}$ and $f$ of $\mathscr{O}_{X_T}$, 
	 we have \begin{equation}\label{eqn:lambda-liebniz} \gamma(fs)=f\gamma(s)+\lambda s\otimes df,\end{equation}
	  and that $$\gamma\circ \gamma: \mathscr{E}\to \mathscr{E}\otimes \Omega^2_{X_T/T}$$ is identically zero. Here we extend $\gamma$ to a sequence of maps $$\gamma: \mathscr{E}\otimes \Omega^p_{X_T/T}\to \mathscr{E}\otimes \Omega^{p+1}_{X_T/T}$$ as in \eqref{eqn:lambda-liebniz}; that is, we set \begin{equation}\label{eqn:lambda-extension} \gamma(s\otimes \omega)=\gamma(s)\wedge \omega+\lambda s\otimes d\omega.\end{equation}
	  \end{definition}

By scaling the $\lambda$-connections, we obtain a natural $\mb{G}_{m, S}$-action on $\mathscr{M}_{\Hod}(X/S)$: precisely, $\mu\in \mb{G}_{m, S}(T)$ acts on a $\lambda$-connection $(\mscr{E}, \gamma)/(X_T/T)$ by 
\[
\mu\cdot [(\mscr{E}, \gamma)]  = [(\mscr{E}, \mu \gamma)].
\]
This gives a $\mb{G}_{m, S}$-action on $\mscr{M}_{\Hod}(X/S)$ for which  the structure map 
\[
\mscr{M}_{\Hod}(X/S)\rightarrow \mb{A}^1_S
\]
is equivariant, where the latter is equipped with the obvious $\mb{G}_{m, S}$-action; this action evidently extends the natural action of $\mb{G}_{m, S}$ on $\mscr{M}_{\Dol}(X/S)$.

Note that there are canonical isomorphisms $$\mathscr{M}_{\Hod}(X/S)_1\overset{\sim}{\to} \mathscr{M}_{\dR}(X/S)$$ $$\mathscr{M}_{\Hod}(X/S)_0\overset{\sim}{\to} \mathscr{M}_{\Dol}(X/S),$$ as a $1$-connection is just a flat connection, and a $0$-connection is a Higgs field.

That is, in the language of \autoref{subsec:filtrations}, $\mathscr{M}_{\Hod}(X/S)$ is a filtration on $\mathscr{M}_{\dR}(X/S)$, with associated graded $\mathscr{M}_{\Dol}(X/S)$.
\subsubsection{The conjugate moduli stack}\label{subsec:M_conj}

We come to the analogue of the conjugate filtration, which is specific to characteristic $p$ (although see \cite[\S 2]{ogus-motives}, \cite{bloch-ogus-gersten} for some fascinating investigations on the conjugate filtration in characteristic zero). Recall that for a smooth projective variety $Y$ over a field $k$ of positive characteristic, $H^i_{\dR}(Y)$ carries a conjugate filtration $F^*_{\conj}$, whose associated graded is $\bigoplus_{p+q=i} H^q(Y^{(1)}, \Omega^p_{Y^{(1)}})$ under  mild assumptions; this filtration was introduced by Katz in \cite[\S 2.3]{katz-p-curvature}.

We now describe the non-abelian analogue, and assume $S$ is in characteristic $p>0$. In this case, there is a natural invariant of a flat connection $(\mathscr{E},\nabla)$ on $X/S$: its $p$-curvature, which is an $\mathscr{O}$-linear map $$\psi_p(\mathscr{E},\nabla): F_{\abs}^*T_{X/S}\to \on{End}(\mathscr{E}),$$ defined by $$\partial\mapsto \nabla(\partial)^p-\nabla(\partial^p)$$ for a vector field $\partial$. By adjunction, we may view $\psi_p=\psi_p(\mathscr{E}, \nabla)$ as an $\mathscr{O}$-linear map $$\psi_p(\mathscr{E}, \nabla): \mathscr{E}\to \mathscr{E}\otimes F_{\abs}^*\Omega^1_{X/S}.$$ A local computation, performed for example in \cite[Proof of Theorem 1.2.1]{ogus-p-curvature}, shows that the composition $$\mathscr{E}\overset{\psi_p}{\longrightarrow} \mathscr{E}\otimes F_{\abs}^*\Omega^1_{X/S}\overset{\psi_p}{\longrightarrow}\mathscr{E}\otimes F_{\abs}^*\Omega^2_{X/S}$$ vanishes. Moreover, by the same local computation $\psi_p$ is horizontal, i.e.~it respects the natural connections on $\mathscr{E}$, $\mathscr{E}\otimes F_{\abs}^*\Omega^1_{X/S}$ (where $F_{\abs}^*\Omega^1_{X/S}$ is given its canonical Frobenius-pullback connection).

We exploit this extra structure to make the following definition:

\begin{definition}\label{defn: mconj}
	The \emph{conjugate moduli stack} $\mathscr{M}_{\conj}(X/S)$ is the stack over $\mathbb{A}^1_S$ whose $T$-points (over a map $\lambda: T\to \mathbb{A}^1_S$) consist of the groupoid of triples $$(\mathscr{E}, \nabla, \theta, \xi)$$ where $(\mathscr{E},\nabla)$ is a flat bundle on $X_T/T$, 
    \begin{equation}\label{eqn: f-higgs-field}
    \theta: \mathscr{E}\to \mathscr{E}\otimes F_{\abs}^*\Omega^1_{X_T/T}
    \end{equation} is a horizontal map with trace zero and with
    \begin{equation}\label{eqn: f-higgs-integrable}
    \theta\circ\theta: \mathscr{E}\to \mathscr{E}\otimes F_{\abs}^*\Omega^2_{X_T/T}
    \end{equation} identically zero, and $$\psi_p(\mathscr{E},\nabla)=\lambda\cdot \theta,$$ and $$\xi: \det(\mathscr{E}, \nabla, \theta)\overset{\sim}{\to} (\mathscr{O}_{X_T}, d, 0)$$ is an isomorphism. Here the requirement that $\theta$ be horizontal means that it respects the connections on $\mathscr{E}, \mathscr{E}\otimes F_{\abs}^*\Omega^1_{X_T/T}$ (where $F_{\abs}^*\Omega^1_{X_T/T}$ is equipped with its canonical Frobenius-pullback connection).
\begin{remark}
    As far as we are aware, this definition first appeared in \cite[Definition 7.12]{menzies2019p}, although the horizontality of $\theta$, which is a crucial part of the definition for us, was missing there.
\end{remark}
\begin{remark}
    One can show that $\mscr{M}_{\conj}$ is an Artin stack, as well as construct a good moduli space for it, by considering an auxiliary  \emph{framed} moduli space, as in \cite[Theorem 4.10]{simpson-moduli-i}. We will not need these statements in this paper.
\end{remark}
    We equip $\mscr{M}_{\conj}(X/S)$ with the  $\mb{G}_{m, S}$ action, which on $T$-points is given by the following:  $\mu\in \mb{G}_{m, S}(T)$ acts  on a point  $[(\mscr{E}, \nabla, \theta)]\in \mscr{M}_{\conj}(X/S)(T)$ by
    \[
    \mu \cdot [(\mscr{E}, \nabla, \theta)] = [(\mscr{E}, \nabla, \mu^{-1}\theta)].
    \]
    It is straightforward to check that the structure map $\mscr{M}_{\conj}(X/S)\rightarrow \mb{A}^1_S$ is equivariant with respect to this $\mb{G}_{m, S}$-action and the obvious $\mb{G}_{m, S}$-action on $\mb{A}^1_S$. 
\end{definition}
\begin{remark}
    We will refer to a map $\theta$ as in \eqref{eqn: f-higgs-field} satisfying \eqref{eqn: f-higgs-integrable} as a \emph{F-Higgs field} on the bundle $\mscr{E}$.
\end{remark}
Note that $\mathscr{M}_{\conj}(X/S)_1$ is naturally identified with $\mathscr{M}_{\dR}(X/S)$ by forgetting $\theta$ (as in this case $\theta$ is forced to be  the $p$-curvature of $(\mathscr{E},\nabla)$). Thus $\mathscr{M}_{\conj}(X/S)$ is, in the terminology of \autoref{subsec:filtrations}, a filtration on $\mathscr{M}_{\dR}(X/S)$.

Set $\mathscr{M}_{\conj}(X/S)_0$ to be the fiber $\lambda^{-1}(0)$ of the natural map $\mathscr{M}_{\conj}(X/S)\overset{\lambda}{\longrightarrow} \mathbb{A}^1_S$ --- that is, the associated graded of this filtration. Note that there is a natural equivalence  \begin{equation} \label{eqn:non-abelian-f-zip} F_{X/S}^*: \mathscr{M}_{\Dol}(X^{(1)}/S)\overset{\sim}{\to} \mathscr{M}_{\conj}(X/S)_0, \end{equation}

where $X^{(1)}$ denotes the Frobenius twist of $X$ over $S$, and the map in \eqref{eqn:non-abelian-f-zip} is the pullback along  relative Frobenius $F_{X/S}: X\to X^{(1)}$. Indeed, if $(\mathscr{E}, \theta)$ is a Higgs bundle on $X^{(1)}$, then the triple $(F_{X/S}^*\mathscr{E}, \nabla_{\text{can}}, F_{X/S}^*\theta)$, where $\nabla_{\text{can}}$ is the canonical Frobenius-pullback connection on $F_{X/S}^*\mathscr{E}$, is a point of $\mathscr{M}_{\conj}(X/S)$ over $0$, as $\nabla_{\text{can}}$ has identically zero $p$-curvature. On the other hand, any point of $\mathscr{M}_{\conj}(X/S)_0$ arises via this construction, by Cartier theory: indeed, if $(\mscr{E}, \nabla, \theta)$ is a triple on $X/S$, as in \autoref{defn: mconj}, living over $\lambda=0$, then by definition $(\mscr{E}, \nabla)$ has vanishing $p$-curvature and hence descends to $X^{(1)}$; moreover, $\theta$ also descends, since it is by definition a flat map. 

\begin{remark}\label{rem:non-abelian-f-zip}
The isomorphism \eqref{eqn:non-abelian-f-zip} above is the non-abelian analogue of the \emph{$F$-zip structure} on de Rham cohomology, in the sense of \cite{moonen2004discrete}. Indeed, recall that if $R^if_*\Omega^\bullet_{X/S, \dR}$ is locally free for each $i$, and the Hodge-de Rham spectral sequence degenerates,  then each flat bundle $R^if_*\Omega^\bullet_{X/S, \dR}$ is equipped with the structure of an $F$-zip, i.e.~a decreasing filtration $F^\bullet$ (the Hodge filtration), an increasing filtration $C_\bullet$ (the conjugate filtration) and isomorphisms $$\on{gr}_j^{C_\bullet} R^if_*\Omega^\bullet_{X/S}\overset{\sim}{\to} (\on{gr}^j_{F^\bullet}R^if_*\Omega^\bullet_{X/S})^{(1)}.$$ That is, the associated graded of the conjugate filtration on the de Rham cohomology of $X/S$ is identified (up to Frobenius twist) with the associated graded of the Hodge filtration. Identifying $$\mathscr{M}_{\Dol}(X^{(1)}/S)\simeq \mathscr{M}_{\Dol}(X/S)^{(1)}\simeq \mathscr{M}_{\Hod}(X/S)_0^{(1)}$$ (where the first isomorphism arises via pullback along the natural map $X^{(1)}\to X$ over absolute Frobenius on $S$), we see that \eqref{eqn:non-abelian-f-zip} says precisely that there is a natural isomorphism $$\mathscr{M}_{\Hod}(X/S)^{(1)}_0\overset{\sim}{\to} \mathscr{M}_{\conj}(X/S)_0,$$ i.e.~the associated graded of the conjugate filtration on $\mathscr{M}_{\dR}(X/S)$ is identified, up to Frobenius twist, with the associated graded of the Hodge filtration, in the language of \autoref{subsec:filtrations}. 

We will see later, in \autoref{subsec:non-abelian-katz-formula}, that there is additional structure here---a non-abelian version of Katz's formula intertwining the Kodaira-Spencer map (the associated graded of the Gauss-Manin connection with respect to the Hodge filtration) with the associated graded of the $p$-curvature of the Gauss-Manin connection with respect to the conjugate filtration.
\end{remark}
\begin{remark}
The following remark is due to Sasha Petrov. There is a natural scaling action of the monoid scheme $\mathbb{A}^1_S$ (with the monoid structure given by multiplication) on the cotangent bundle $T^*(X^{(p)}/S)$ of $X^{(p)}/S$; let $$a: \mathbb{A}^1\times 	T^*(X^{(p)}/S)\to T^*(X^{(p)}/S)$$ be the action map. Let $\mathscr{D}_{X/S}$ be the sheaf of algebras on $T^*(X^{(p)}/S)$ associated to $F_{\abs *}D_{X/S}$, as in \cite{ogus-vologodsky}, so that a flat bundle on $X/S$ is the same as a $\mathscr{D}_{X/S}$-module which is finite flat over $X$. Then the conjugate moduli stack is the same as the moduli of $a^*\mathscr{D}_{X/S}$-modules which are finite flat over $X$. 
\end{remark}

\subsection{Semistability and deformation to a $\mathbb{G}_m$-fixed point}\label{section: semistable-deform}
For all the stacks above, we will use the decoration $\mathscr{M}^{\on{ss}}$ (respectively $\mscr{M}^{\on{s}}$) to denote the semistable (respectively stable) locus with respect to some fixed polarization on $X/S$; in this generality such moduli stacks and spaces were constructed and studied  by  Langer  \cite{langer2014semistable}.

Let $k$ be a field and fix $m\in \mscr{M}^{\on{ss}}_{\Hod}(X/S)(k)$ over some point $\tau$ of $\mathbb{A}^1_k$, corresponding to a $\tau$-connection $(\mathscr{E}, \gamma)$.

Using Langton theory, Simpson showed that the family of $\lambda\tau$-connections $(\mscr{E}, \lambda\gamma)$ over $\mathbb{G}_{m,k}$ extends over $0\in \mathbb{A}^1_k$. 
\begin{theorem}\label{thm: simpson-langton}
    $(\mscr{E}, \gamma)$ extends to a $\lambda \tau$-connection $(\mscr{E}_{\Hod}, \gamma_{\Hod})$ on $X_k\times_k \mb{A}^1_k$, such that  $\mscr{E}_{\Hod}|_{X_k\times \{0\}}$ is semistable with vanishing rational Chern classes. Moreover, $(\mscr{E}_{\Hod}, \gamma_{\Hod})|_{X_k\times \{0\}}$ is a $\mb{G}_m$-fixed point of $\mscr{M}_{\Dol}(X/S)$.
\end{theorem}
\begin{proof}
 See \cite[Corollary 10.2]{simpson1996hodge}.   
\end{proof}
\section{The non-abelian Gauss-Manin connection}\label{sec:NAGM}
In this section we discuss the geometry of the stacks described in \autoref{sec:filtrations}, and recall some of the structures they carry, in particular the isomonodromy foliations on $\mathscr{M}_{\dR}$ and $\mathscr{M}_{\conj}$, and the lifting of tangent vectors on $\mathscr{M}_{\Dol}$. We conclude with a proof of our non-abelian version of Katz's formula, \autoref{lemma:katz-formula}.
\subsection{Deformation theory}
In this section we will describe the deformation theory of some of the stacks described in \autoref{sec:filtrations}.
\subsubsection{$\lambda$-Atiyah complexes}
We first describe the deformation theory of triples $(X, \mathscr{E}, \gamma)$, where $X$ is a smooth $S$-scheme and $(\mathscr{E},\gamma)$ is a $\lambda$-connection on $X$. Note that this includes both the deformation theory of flat bundles and the deformation theory of Higgs bundles (when $\lambda=1,0$ respectively). 

The deformation theory of a vector bundle $\mathscr{E}$ with $\lambda$-connection $\gamma$ is controlled by the \emph{$\lambda$-de Rham complex} of the adjoint bundle.

\begin{definition}\label{defn:lambda-de-rham}
	Let $(\mathscr{E}, \gamma)$ be a $\lambda$-connection on $X/S$, where $X$ is a smooth $S$-scheme and $\lambda$ is a regular function on $X$. The \emph{$\lambda$-de Rham complex} of $(\mathscr{E}, \gamma)$ is the complex $$\mathscr{E}^\bullet_{\lambda-\dR}: \mathscr{E}\overset{\gamma}{\longrightarrow} \mathscr{E}\otimes \Omega^1_{X/S}\overset{\gamma}{\longrightarrow} \mathscr{E}\otimes \Omega^2_{X/S} \overset{\gamma}{\longrightarrow} \cdots $$
	Here the differentials $\gamma$ are defined as in \eqref{eqn:lambda-extension}. 
\end{definition}
If $(\mathscr{E}, \gamma)$ is a $\lambda$-connection, the adjoint bundle $\on{ad}(\mathscr{E})$ is naturally equipped with a $\lambda$-connection, given by $\on{ad}(\gamma)$, the commutator with $\gamma$. We will need to define a variant of $\ad(\mathscr{E})^\bullet_{\lambda-\dR}$ where the first term, $\on{ad}(\mathscr{E})$, is replaced by the Atiyah bundle of $\mathscr{E}$, whose definition we now recall.
\begin{definition}
	Let $\pi: X\to S$ be a morphism and $\mathscr{E}$ be a vector bundle on $X$. Recall that a first-order differential operator on $\mathscr{E}$ is a $\pi^{-1}\mathscr{O}_S$-linear map $\tau: \mathscr{E}\to \mathscr{E}$ with symbol $$\tau_f: \mathscr{E}\to \mathscr{E}$$ $$\tau_f: s\mapsto \tau(fs)-f\tau(s)$$ an $\mathscr{O}_X$-linear map for all local sections $f\in \mathscr{O}_X(U)$. The \emph{Atiyah bundle} of $\mathscr{E}$, denoted $\on{At}_{X/S}(\mathscr{E})$ is the subsheaf of the sheaf $\on{Diff}^1_{X/S}(\mathscr{E})$ of first-order differential operators, consisting of first-order differential operators $\tau$ whose symbol is a scalar---that is, for all $f\in \mathscr{O}_X(U)$, $\tau_f(s)=\delta_\tau(f)s$ for some $\delta_\tau(f)\in \mathscr{O}_X(U)$. 
	
	If $X/S$ is clear from context, we will omit it from the notation and write only $\on{At}(\mathscr{E})$.
	\end{definition}

	The map $f\mapsto \delta_\tau(f)$ is a derivation, and hence gives rise to a map $$\delta: \on{At}(\mathscr{E})\to T_{X/S}$$ $$\tau\mapsto \delta_\tau.$$
	
	The Atiyah bundle sits in a short exact sequence \begin{equation} \label{eqn:atiyah-exact-sequence} 0\to \on{End}(\mathscr{E})\to \on{At}(\mathscr{E})\overset{\delta}{\longrightarrow} T_{X/S}\to 0\end{equation} where the inclusion $\on{End}(\mathscr{E})\hookrightarrow\on{At}(\mathscr{E})$ is given by viewing $\mathscr{O}$-linear endomorphisms as (zeroth order) differential operators. 	Note that taking commutators of differential operators endows $\on{At}(\mathscr{E})$ with the structure of a sheaf of Lie algebras, compatible with that on $T_{X/S}$, $\on{End}(\mathscr{E})$.
	
	The following is well-known (see e.g.~\cite[\S4]{martens2024determinant}):
	
	\begin{proposition}\label{prop:splitting-atiyah-sequence}
		There is a natural bijection between flat $\lambda$-connections $\gamma$ on $\mathscr{E}$ and maps $$q^\gamma: T_{X/S}\to \on{At}(\mathscr{E})$$ that respect the Lie algebra structure and such that $\delta\circ q^\gamma=\lambda\cdot \on{id}_{T_{X/S}}$, given by sending $\gamma$ to the map $$q^{\gamma}: T_{X/S}\to \on{At}(\mathscr{E})$$ $$\partial\mapsto \gamma(\partial).$$
	\end{proposition}
	
	Given an isomorphism $\det(\mathscr{E}, \gamma)\overset{\sim}{\to}(\mathscr{O}_X, \lambda d)$, 
	we obtain a natural map $$\on{At}(\mathscr{E})\to \on{At}(\mathscr{O}_X)$$ induced by the map sending a differential operator on $\mathscr{E}$ 
	to its induced action on $\det(\mathscr{E})$. The canonical connection $d$ on $\mathscr{O}_X$ gives rise to an isomorphism $$\on{At}(\mathscr{O}_X)\simeq \mathscr{O}_X\oplus T_{X/S}.$$ 
	We define $$\on{At}^0(\mathscr{E}):=\ker(\on{At}(\mathscr{E})\to \on{At}(\det(\mathscr{E}))\simeq \on{At}(\mathscr{O}_X)\simeq \mathscr{O}_X\oplus T_{X/S}\to \mathscr{O}_X).$$
	
	\begin{definition}[$\lambda$-Atiyah complexes]
		Let $(\mathscr{E}, \gamma)$ be a $\lambda$-connection on $X/S$, we define the $\lambda$-Atiyah complex of $\mathscr{E}$, 
		$$\on{At}_{X/S}(\mathscr{E})^\bullet_{\lambda}: \on{At}_{X/S}(\mathscr{E})\to \on{End}(\mathscr{E})\otimes \Omega^1_{X/S}\to  \on{End}(\mathscr{E})\otimes \Omega^2_{X/S}\to \cdots.$$
		Here away from degree zero the differentials are defined as in \autoref{defn:lambda-de-rham}, and in degree $0$ the differential is given by the commutator with $\gamma$. There is also a variant $\on{At}^0_{X/S}(\mathscr{E})^\bullet_{\lambda}$, where we replace $\on{At}_{X/S}(\mathscr{E})$ with $\on{At}^0_{X/S}(\mathscr{E})$, and $\on{End}(\mathscr{E})$ with $$\on{End}^0(\mathscr{E}):=\ker(\on{End}(\mathscr{E})\overset{\tr}{\longrightarrow}\mathscr{O}_X).$$ See 
		\cite[Definition 7.3.1]{p-painleve} for more details.
	\end{definition}
	
\subsubsection{Tangent spaces to moduli stacks}
Let $A$ be an arbitrary ring, and $S$ a scheme over $A$. All fiber products will be taken over $\spec A$ unless otherwise specified; similarly, tangent sheaves and tangent complexes will be taken relative to $A$ unless otherwise specified, so that, for example, $T_S$ means $T_{S/A}$, and so on.
\begin{proposition}\label{prop:tangent-spaces}
	Let $X$ be a smooth $S$-scheme and $(\mathscr{E}, \gamma)$ a $\lambda$-connection on $X$ (resp.~with a fixed isomorphism $\det(\mathscr{E}, \gamma)\overset{\sim}{\to} (\mathscr{O}_X, \lambda d)$).
	\begin{enumerate}
    \item There is a natural bijection between  infinitesimal  automorphisms of $(\mathscr{E},\gamma)$  (resp. with trivial determinant) and $H^0(X, \on{End}(\mathscr{E})_{\lambda-\dR}^\bullet)$ (resp.~$H^0(X, \on{End}^0(\mathscr{E})_{\lambda-\dR}^\bullet)$). By infinitesimal automorphisms, we mean  automorphisms of the trivial deformation of $(\mscr{E}, \gamma)$ to $X\otimes_S S[\epsilon]/\epsilon^2$, which restrict to the identity modulo $\epsilon$. 
		\item There is a natural bijection between  deformations of $(\mathscr{E},\gamma)$ over $S[\epsilon]/\epsilon^2$ (resp.~with fixed determinant) and $H^1(X, \on{End}(\mathscr{E})_{\lambda-\dR}^\bullet)$ (resp.~$H^1(X, \on{End}^0(\mathscr{E})_{\lambda-\dR}^\bullet)$).
        \item There is a natural bijection between  infinitesimal  automorphisms of $(X, \mathscr{E},\gamma)$  (resp. with trivial determinant) and $H^0(X, \on{At}(\mathscr{E})_{\lambda}^\bullet)$ (resp.~$H^0(X, \on{At}^0(\mathscr{E})_{\lambda}^\bullet)$). 
		\item There is a natural bijection between deformations of $(X, \mathscr{E},\gamma)$ over $S[\epsilon]/\epsilon^2$ (resp.~with fixed determinant) and $H^1(X, \on{At}_{X/S}(\mathscr{E})^\bullet_\lambda)$ (resp.~$H^1(X, \on{At}^0_{X/S}(\mathscr{E})^\bullet_\lambda)$).
	\end{enumerate}
\end{proposition}
\begin{proof}
	Standard; see e.g.~\cite[Lemma 4.4]{katzarkov1999non} for the case $\lambda=1$; the proof in general is identical.
\end{proof}

We can rephrase \autoref{prop:tangent-spaces} in terms of tangent stacks and their cohomology sheaves. Let $\pi: \mscr{M}\rightarrow S$ be a  stack. In \cite[D\'efinition 17.13]{laumon-moret-bailly-champs}, Laumon--Moret-Bailly define the tangent complex  $\mscr{T}_{\mscr{M}/S}$ of $\pi$ (their notation for this stack is  $\mscr{T}(\mscr{M}/S)$); it is an object in the derived category of $\mscr{M}$, and we can take its cohomology sheaves $H^i(\mscr{T}_{\mscr{M}/S})$.

\begin{definition}\label{defn: tangent-sheaf}
    We write $T_{\mscr{M}/S}$ for $H^0(\mscr{T}_{\mscr{M}/S})$, and refer to it as the \emph{tangent sheaf} of $\mscr{M}$ over $S$. 
\end{definition}
Recall that we have moduli stacks $\pi_{\dR}: \mscr{M}_{\dR}(X/S) \rightarrow S$, $\pi_{\Hod}: \mscr{M}_{\Hod}(X/S) \rightarrow \mathbb{A}^1_S$. 
Let $f_{univ}: X\times_S \mscr{M}_{\dR}\rightarrow \mscr{M}_{\dR}$ denote the natural projection map, and  $(\mscr{E}_{univ}, \nabla_{univ})$  the universal flat bundle on $X\times_S \mscr{M}_{\dR}/\mscr{M}_{\dR}$. 

Similarly, write $f_{univ}^{\Hod}: X\times_S \mscr{M}_{\Hod}\rightarrow \mscr{M}_{\Hod}$ for  the natural projection map, and  $(\mscr{E}^{\Hod}_{univ}, \gamma_{univ})$  the universal $\lambda$-connection bundle on $X\times_S \mscr{M}_{\Hod}/\mscr{M}_{\Hod}$.

\begin{corollary}
\begin{enumerate}
    \item   There are canonical isomorphisms 
    \begin{align*}
        T_{\mscr{M}_{\dR}(X/S)/S} &\simeq R^1f_{univ *} \End^0(\mscr{E}_{univ})_{\dR}, \\
        T_{\mscr{M}_{\dR}(X/S)} &\simeq \pi^*T_S \times_{R^1f_{univ *}T_{X\times \mscr{M}_{\dR}/\mscr{M}_{\dR} }} R^1f_{univ *} (\At^0(\mscr{E}_{univ})_{1}^\bullet),
    \end{align*}
    \item and similarly
    \begin{align*}
    T_{\mscr{M}_{\Hod}/S\times \mb{A}^1_A} &\simeq R^1f^{\Hod}_{univ *} \End^0(\mscr{E}^{\Hod}_{univ})_{\lambda-\dR},   \\
     T_{\mscr{M}_{\Hod}/\mb{A}^1_A} &\simeq \pi_{\Hod}^*T_S \times_{R^1f_{univ*}^{\Hod}T_{X\times \mscr{M}_{\Hod}/\mscr{M}_{\Hod}}}  R^1f^{\Hod}_{univ *} (\At^0(\mscr{E}^{\Hod}_{univ})_{\lambda}^\bullet ).       
    \end{align*}

\end{enumerate}
\end{corollary}
\subsection{Recollections on isomonodromy}
\subsubsection{Isomonodromy in the analytic category}
Let $\pi: X\to S$ be a smooth projective morphism of complex-analytic spaces, with connected fibers, and with $S$ smooth and connected. Fix $s\in S$. Let $$M_B(X_s, r)=\text{Hom}(\pi_1(X_s), SL_r)\sslash SL_r$$ be the Betti moduli space of $X_s$.  Its complex points are in bijection with conjugacy classes of semisimple representations of $\pi_1(X_s)$ into $SL_r(\mathbb{C})$.

As $\pi$ is a fiber bundle (by Ehresmann's theorem), there is a natural outer action of $\pi_1(S,s)$ on $\pi_1(X_s)$, and hence $\pi_1(S,s)$ acts on $M_B(X_s, r)$. Setting $\widetilde{S}$ to be the universal cover of $S$, we write $$M_B(X/S, r):=(M_B(X_s, r)_{\mathbb{C}}^{\an}\times \widetilde{S})/\pi_1(S,s).$$ Projection onto the second factor gives a natural map $$p_B: M_B(X/S,r)\to S,$$ whose fibers are each isomorphic to $M_B(X_s, r)_{\mathbb{C}}^{\an}$. As $M_B(X/S, r)$ is, analytically-locally on $S$, a product, the smooth locus of $p_B$ is equipped with a foliation horizontal over $S$. And more generally, $M_B(X/S, r)$ is a crystal in complex-analytic spaces over $S$, see \cite[\S8]{simpson-moduli-representations-2}.

\subsubsection{Isomonodromy in the algebraic category}
We recall some relevant facts about the algebraic structure of isomonodromy.  
Let $k$ be a field, $S$ a smooth scheme over $k$, and $f: X\rightarrow S$ a smooth projective morphism.

Let $I$ denote the ideal sheaf of the diagonal in $S \times S$, and let  $(S(1), \bar{I}, \gamma)$ be the completed PD envelope of $(S\times S, I)$; in other words,  $S(1)$ is a formal scheme with a map $j: S(1)\rightarrow S\times S$, $\bar{I}$ an ideal sheaf on $S(1)$, and $\gamma$ a divided power structure on $\bar{I}$, which are required to satisfy a universal property as in \cite[\href{https://stacks.math.columbia.edu/tag/07H8}{Tag 07H8}]{stacks-project}. Recall that a divided power structure $\gamma$ on an ideal $\bar{I}$ consists of a collection of maps $\{\gamma_n: \bar{I}\rightarrow \bar{I}\} $ indexed by the positive integers $n\geq 1$, and satisfying several conditions--see \cite[\href{https://stacks.math.columbia.edu/tag/07GK}{Tag 07GK}]{stacks-project} for details.

 Using the map $j: S(1)\rightarrow S\times S$, we can define  the ideal sheaf $j^{-1}I\cdot \mscr{O}_{S(1)}$ on $S(1)$; we sometimes also denote this ideal sheaf by $I$. In the case when $k$ has  characteristic zero, the completed PD envelope  is simply the usual completion of $S\times S$ along the diagonal.   Let $\pi_i: S(1)\rightarrow S$, $i=1,2$, denote the two projection maps.
\begin{remark}
    The two maps $\pi_{1,2}$ make $S(1)$ into a groupoid over $S$, and it is possible to formulate many of the statements in this section in this language; we will not take this point of view here.
\end{remark}

Let $T_0$ be a scheme. Fix a map $t_0: T_0\to S$ and $T$ a nilpotent PD-thickening of $T_0$. Fix two maps $T\to S$ extending $t_0$, or equivalently a map  $t: T\to S(1)$ extending $\Delta\circ t_0$, where $\Delta$ is the diagonal. For each $i$, define $X_{T,i}$ and $X_{S(1),i}$ via the diagram, where each square is  a pullback square:
\[
    \begin{tikzcd}
X_{T, i} \arrow[r, ""] \arrow[d]
& X_{S(1), i} \arrow[r] \arrow[d] &X \arrow[d] \\
T \arrow[r, "t"]
& S(1) \arrow[r, "\pi_i"] &  S.
\end{tikzcd}
\]

For the reader's convenience and as a warm up to more involved constructions later, we recall Simpson's description of the crystal structure on $\mscr{M}_{\dR}(X/S)$; the latter is defined in \autoref{defn: mdr}.
\begin{construction}\label{construction: dR-crystal}
%Let $T_0$ denote the pullback 
%\[
% \begin{tikzcd}
%T_0\arrow[r, ""] \arrow[d, "\Delta_T"]
%& S \arrow[d, "\Delta"]  \\
%T \arrow[r, "t"]
%& S(1), \\
%\end{tikzcd}
%\]
%and $\Delta_T$ the left hand vertical arrow. 
%
Note that the pullbacks of $X_{T,1}$ and $X_{T,2}$ to $T_0$ are canonically isomorphic, and so we denote both by $X_{T,0}$.
    Suppose $(\mscr{E}, \nabla)$ is a flat bundle on $X_{T,1}/T$. Let $
    \{U_{\beta}\}_{\beta \in I}$ be a covering of $X$ by open affines, and $U_{T, \beta, 0}, \widehat{U}_{T, \beta, 1}, \widehat{U}_{T, \beta, 2}$ the pullbacks to $X_{T, 0}, X_{T,1}, X_{T,2}$ respectively. Then each of $\widehat{U}_{T, \beta, 1}$ and $\widehat{U}_{T, \beta, 2}$ is a deformation of the smooth affine $U_{T, \beta,  0}$ along the PD thickening $T_0\rightarrow T$, and so there exists an isomorphism of $T$-schemes   $\tilde{f}_{\beta}: \widehat{U}_{T, \beta, 2}\rightarrow \widehat{U}_{T, \beta, 1}$. Write $f_{\beta}$ for the composition
    \[
     \widehat{U}_{T, \beta, 2}\xrightarrow{\tilde{f}_{\beta}} \widehat{U}_{T, \beta, 1} \rightarrow U_{T, \beta,  0}
    \]

    Now given a flat bundle $(\mscr{E}_T, \nabla_T)$ on $X_{T,1}/T$, we construct a flat bundle on $X_{T,2}/T$ as follows. First consider the  flat bundles $\tilde{f}_{\beta}^*((\mscr{E}, \nabla_T)|_{\widehat{U}_{T, \beta, 1}})$ on $\widehat{U}_{T, \beta, 2}/T$; for brevity we denote this flat bundle simply by $\mscr{F}_{\beta}$. We now glue these on the intersections $\widehat{U}_{T, \beta\gamma, 2}:= \widehat{U}_{T, \beta, 2}\cap \widehat{U}_{T, \gamma, 2}$, for each $\beta, \gamma\in I$: that is, we need to specify isomorphisms
\[
\alpha_{\beta\gamma}: \mscr{F}_{\beta}|_{\widehat{U}_{T, \beta\gamma, 2}} \simeq \mscr{F}_{\gamma}|_{\widehat{U}_{T, \beta\gamma, 2}}
\]
   such that $\alpha_{\gamma\delta}\circ \alpha_{\beta\gamma}= \alpha_{\beta\delta}$ for each choice of indices $\beta, \gamma, \delta \in I$, and where these maps are taken on the triple intersections $\widehat{U}_{T, \beta, 2}\cap \widehat{U}_{T, \gamma, 2} \cap \widehat{U}_{T, \delta, 2}$. In \'etale local coordinates $x_1, \cdots, x_d$ on $X/S$, and therefore on $\widehat{U}_{T, \beta, 2}/T$, we define $\alpha_{\beta \gamma}$ as 
   \begin{equation}\label{eqn: taylor-formula-cocycle}
   s\otimes 1\mapsto \sum_{\underline{\ell}} \nabla_{\partial_{x_{\underline{\ell}}}}(s)\otimes \prod_{m=1}^d \gamma_{\ell_m}(f_{\beta}^*(x_m)-f_{\gamma}^*(x_m)).
   \end{equation}
   Here the sum ranges over tuples $\underline{\ell}=(\ell_1, \cdots, \ell_d)$  of non-negative integers, $\gamma_i$ denote the divided power operations, and   $\nabla_{\partial_{x_{\underline{\ell}}}}$ denotes the operator $\nabla_{\partial_{x_1}}^{\ell_1}\circ \cdots \circ \nabla_{\partial_{x_d}}^{\ell_d}$ acting on sections of $\mscr{F}_{\beta}$.
\end{construction}

\begin{proposition}[Simpson]\label{prop: mdr-crystal}
    \autoref{construction: dR-crystal} gives  a crystal structure on $\mscr{M}_{\dR}(X/S)$ over $S$. Concretely, it gives  an equivalence  
    \[
    \epsilon: \pi_1^*\mscr{M}_{\dR}(X/S) \simeq \pi_2^*\mscr{M}_{\dR}(X/S),
    \]
    which moreover satisfies the cocycle condition. Furthermore, $\epsilon$ is functorial, i.e. for flat bundles $(\mscr{E}_1, \nabla_1), (\mscr{E}_2, \nabla_2)\in \pi_1^*\mscr{M}_{\dR}(X/S)$, viewed as flat bundles on $X_{T,1}$, and a horizontal morphism $h: (\mscr{E}_1, \nabla_1)\rightarrow (\mscr{E}_2, \nabla_2)$, there is a natural morphism $\epsilon(h): \epsilon(\mscr{E}_1, \nabla_1) \rightarrow \epsilon(\mscr{E}_2, \nabla_2)$. 
\end{proposition}
\begin{proof}
    This is due to Simpson \cite[\S8]{simpson-moduli-representations-2}, but see for example \cite[Appendix B.2]{p-painleve}, especially for the explication in characteristic $p$; the last claim follows  from the recipe in \autoref{construction: dR-crystal}, since the maps $\alpha_{
\beta\gamma
    }$ are maps of flat connections. %\cite[Construction B.2.2]{p-painleve}.
\end{proof}

\begin{remark}
    There are several equivalent ways to formulate \autoref{prop: mdr-crystal}: one could say that $\mscr{M}_{\dR}(X/S)$ has a canonical descent to $S^{\dR}$, the de Rham stack of $S$ (appropriately defined, especially in characteristic $p$), or that there is a canonical  action of the groupoid $S(1)$ on $\mscr{M}_{\dR}(X/S)$.
\end{remark}
We now assume $k$ is of characteristic $p$. The map  $\epsilon$ in \autoref{prop: mdr-crystal}  is easy to describe  on those flat bundles with vanishing $p$-curvature, as we now spell out; this will be needed in \autoref{subsec:non-abelian-katz-formula} when we formulate the non-abelian version of Katz's formula.

As in the setup in \autoref{construction: dR-crystal}, suppose that $T_0$ is a scheme, $t_0: T_0\rightarrow S$ a map,  $T$ a nilpotent PD thickening of $T_0$ and  $t: T \rightarrow S(1)$ a map extending $\Delta \circ t_0$. Note that, since $T_0\to  T$ is a PD thickening, absolute Frobenius on $T$ factors as $T\rightarrow T_0 \to  T$, and for each $i=1, 2$ we have a diagram where each square is a pullback square:
\[
\begin{tikzcd}
X_{T, i}
\arrow[drrr, bend left, "F_{abs}"]
%\arrow[ddr, bend right, "y"]
\arrow[dr, "F_{X_{T, i}/T}"] & & \\
& X_{T,i}^{(1)} \arrow[r, ""] \arrow[d, ""] & X_{T, 0} \arrow[r] \arrow[d] & X_{T,i} \arrow[d, "f"] \\
& T \arrow[r, "\phi"]  & T_0 \arrow[r, ""] & T.
\end{tikzcd}
\]
Moreover, the map $\phi$ in the above diagram  is independent of $i$, and hence we have a canonical isomorphism 
\begin{equation}\label{eqn: isom-two-frob-twists}
X_{T, 1}^{(1)} \simeq X_{T, 2}^{(1)}.
\end{equation}
\begin{proposition}\label{prop: isomonodromy-on-p-curv-0}
 Suppose that a flat bundle $(\mscr{V}, \nabla)$ on $X_{T,1}/T$ has zero $p$-curvature, so that we can write it in the form 
\[
(\mscr{V}, \nabla) = (F^*_{X_{T,1}/T}\mscr{V}', \nabla_{\text{can}}),
\]
where  $\mscr{V}'$ is a bundle on $X_{T, 1}^{(1)}$, and $\nabla_{\text{can}}$ denotes the canonical connection.  Then 
\begin{equation}
\epsilon(\mscr{V}, \nabla) = (F^*_{X_{T,2}/T} \mscr{V}', \nabla_{\text{can}}).
\end{equation}
Here we have used the isomorphism in \eqref{eqn: isom-two-frob-twists} to view $\mscr{V}'$ as a bundle on $X_{T, 2}^{(1)}$. 
\end{proposition}
\begin{proof}
We will use the following slightly abusive piece of notation: for a map of schemes $U\rightarrow V$ and $M$ a vector bundle (or flat bundle)  on $V$, we write $M\otimes U$ for the pullback to $U$.
We compute $\epsilon(\mscr{V},\nabla)$ using Simpson's explicit construction
(\autoref{construction: dR-crystal}).  With notation as in that construction,
for each index $\beta$ recall the notation
\[
\mscr{F}_\beta \;:=\; \tilde f_\beta^*(\mscr{V},\nabla)
\;=\; (\mscr{V},\nabla)\otimes_{\widehat U_{T,\beta,1}} \widehat U_{T,\beta,2},
\]
which is a flat bundle on $\widehat U_{T,\beta,2}$.  Then $\epsilon(\mscr{V},\nabla)$ is
obtained by gluing the $\mscr{F}_\beta$ along the isomorphisms
\[
\alpha_{\beta\gamma}:\mscr{F}_\beta|_{\widehat U_{T,\beta\gamma,2}}
\stackrel{\sim}{\longrightarrow}
\mscr{F}_\gamma|_{\widehat U_{T,\beta\gamma,2}}
\]
defined by the Taylor formula \eqref{eqn: taylor-formula-cocycle}.

Now assume $(\mscr{V},\nabla)=(F^*_{X_{T,1}/T}\mscr{V}',\nabla_{\text{can}})$.
By definition of the canonical connection, local sections of the form
$F^*_{X_{T,1}/T}(s')$ with $s'$ a local section of $\mscr{V}'$ are horizontal:
\[
\nabla_{\text{can}}\big(F^*_{X_{T,1}/T}(s')\big)=0.
\]
Consequently, for every local vector field $\partial$ on $X_{T,1}/T$ and every
integer $\ell>0$ we have
\[
\nabla_{\partial}^{\ell}\big(F^*_{X_{T,1}/T}(s')\big)=0.
\]

Let $s'$ be a local section of $\mscr{V}'$ on
$\widehat U^{(1)}_{T,\beta\gamma,1}$, and set
$s:=F^*_{X_{T,1}/T}(s')$, viewed (after pullback to $\widehat U_{T,\beta\gamma,2}$)
as a section of $\mscr{F}_\beta|_{\widehat U_{T,\beta\gamma,2}}$.
Plugging $s$ into \eqref{eqn: taylor-formula-cocycle}, all terms with some
$\ell_m>0$ vanish, so only the multi-index $\underline{\ell}=\underline{0}$
contributes.  Hence
\[
\alpha_{\beta\gamma}(s\otimes 1)=s\otimes 1.
\]
Since $\mscr{V}=F^*_{X_{T,1}/T}\mscr{V}'$ is (locally) generated over the
structure sheaf by such Frobenius-pullback sections, the preceding identity
implies that, under the natural identifications
\[
\mscr{F}_\beta|_{\widehat U_{T,\beta\gamma,2}}
\;\simeq\;
\mscr{V}'|_{\widehat U^{(1)}_{T,\beta\gamma,1}}
\otimes_{\widehat U^{(1)}_{T,\beta\gamma,1}}
\widehat U_{T,\beta\gamma,2}
\qquad\text{and similarly for }\gamma,
\]
the map $\alpha_{\beta\gamma}$ acts as the identity. %on the $\mscr{V}'$-factor.
In other words, the descent datum $\{\mscr{F}_\beta,\alpha_{\beta\gamma}\}$ is
exactly the descent datum $(F^*\mscr{V}'|_{\widehat{U}_{T, \beta\gamma, 2}}, \id ) 
$, and therefore we have an identification of the bundle underlying $\epsilon(\mscr{V}, \nabla)$ with $F^*_{X_{T,2}/T}\mscr{V}'$. Here we have made use of the canonical identifications $X_{T,1}^{(1)}\simeq X_{T,2}^{(1)}$, $\widehat{U}^{(1)}_{T, \beta\gamma, 1}\simeq \widehat{U}^{(1)}_{T, \beta\gamma, 2}$.

%obtained by restricting the Frobenius pullback of
%$\mscr{V}'$ to the formal neighborhoods $\widehat U_{T,\beta,2}$.

%Finally, by the canonical identification
%$X_{T,1}^{(1)}\simeq X_{T,2}^{(1)}$ from \eqref{eqn: isom-two-frob-twists}, we
%may view $\mscr{V}'$ as a bundle on $X_{T,2}^{(1)}$, and the glued bundle on
%$X_{T,2}$ is precisely %$F^*_{X_{T,2}/T}\mscr{V}'$.  
Finally, the glued connection is the
canonical one (it is characterized by horizontality of the Frobenius-pullback
sections, which we have just seen are preserved by the gluing).  Therefore
\[
\epsilon(\mscr{V},\nabla)\;=\;(F^*_{X_{T,2}/T}\mscr{V}',\nabla_{\text{can}}),
\]
as claimed.

\end{proof}

\subsection{Liftings of tangent vectors}\label{subsec:lifting-of-tangent}
We now explain how to take the \enquote{associated graded of a foliation on a filtered space} in the language of \autoref{subsec:filtrations}. The construction below is analogous to taking the associated graded of a connection on a vector bundle with respect to a Griffiths-transverse filtration.
\begin{definition}
	Let $\pi: \mathscr{M}\to S$ be a morphism and $$d\pi: T_{\mathscr{M}}\to \pi^*T_S$$ its derivative. Let $\lambda$ be a regular function on $\mathscr{M}$. A \emph{$\lambda$-lifting of tangent vectors} is a map $$p: \pi^*T_S\to T_{\mathscr{M}}$$ satisfying $$d\pi\circ p =\lambda\cdot \on{id}_{\pi^*T_S}.$$  Note that if $\lambda=0$ such a map factors through $\ker d\pi$; in particular if $S$ is smooth a $0$-lifting of tangent vectors amounts to a map $$p: \pi^*T_S\to T_{\mathscr{M}/S}.$$ We will say a $\lambda$-lifting of tangent vectors is \emph{integrable} if $p$ is a morphism of sheaves of Lie algebras.

\end{definition}
\begin{example}
	Let $P\to X$ be a principal $G$-bundle and $p: \pi^*T_X\to T_P$ a $G$-equivariant $\lambda$-lifting of tangent vectors. This is the same as a $\lambda$-connection on $P$; if it is integrable, it is the same data as a flat $\lambda$-connection. In the case when $G=GL_n$ and $P$ is the frame bundle of a vector bundle $\mathscr{E}$, $T_P$ descends to the vector bundle $\on{At}(\mathscr{E})$ on $X$, and this notion coincides with the correspondence between maps $q^\gamma: T_{X}\to \on{At}(\mathscr{E})$ and $\lambda$-connections discussed in \autoref{prop:splitting-atiyah-sequence}.
\end{example}
Note that an integrable $1$-lifting of tangent vectors is precisely a horizontal integrable foliation, i.e.~a splitting of the tangent exact sequence $$0\to T_{\mathscr{M}/S}\to T_{\mathscr{M}}\to \pi^*T_S\to 0$$ respecting the Lie algebra structure.

Now suppose we are given a filtration $\mathscr{M}'$ on $\mathscr{M}$ in the sense of \autoref{subsec:filtrations}, i.e.~a $\mathbb{G}_m$-equivariant map $\lambda: \mathscr{M}'\to \mathbb{A}^1_S$ and an identification of $\mathscr{M}'_1$ with $\mathscr{M}$. If $\mathscr{M}'/S$ is also equipped with a $\lambda$-lifting of tangent vectors (and hence $\mathscr{M}/S$ is equipped with a horizontal foliation $\mathscr{F}$), we will refer to the $0$-lifting of tangent vectors on $\mathscr{M}'_0$ as the \emph{associated graded} of $\mathscr{F}$ with respect to the filtration $\mathscr{M}'$.
\begin{definition}\label{defn:Hodge-lifting}
Suppose we are given a smooth projective morphism of smooth schemes $f: X\to S$.

	There is a natural $\lambda$-lifting of tangent vectors $\Xi_{X/S}$ on $\mathscr{M}_{\Hod}(X/S)\to S$. Namely, given a $S'$-point of $S$ and a $\lambda$-connection $(\mathscr{E},\gamma)$ on $X_{S'}/{S'}$ with trivialized determinant, one obtains as in \autoref{prop:splitting-atiyah-sequence} a map $q^\gamma: T_{X_{S'}/S'}\to \on{At}^0_{X_{S'}/S'}(\mathscr{E})$. The composition $$T_{S'}\to R^1f'_*T_{X_{S'}/S'} {\to} R^1f'_*\on{At}^0_{X_{S'}/S'}(\mathscr{E})^\bullet_\lambda$$ is a $\lambda$-lifting of tangent vectors, under the identifications of \autoref{prop:tangent-spaces}. Here the first arrow is the Kodaira-Spencer map classifying first-order deformations, and the second is induced by the composition $$T_{X_{S'}/S'}\overset{q^\gamma}{\to} \on{At}^0_{X_{S'}/S'}(\mathscr{E})\to \on{At}^0_{X_{S'}/S'}(\mathscr{E})^\bullet_\lambda.$$ Note that the second arrow above is not a map of complexes, though the composition is.
\end{definition}
Restricting to $0\in \mathbb{A}^1_S$, we obtain a $0$-lifting of tangent vectors $\Theta_{X/S}: \pi_{\Dol}^*T_S\to T_{\mathscr{M}_{\Dol}(X/S)}$, which we refer to as the associated graded of the isomonodromy foliation. Indeed, by \cite[Lemma 4.5]{katzarkov1999non}, the corresponding $1$-lifting of tangent vectors is precisely the isomonodromy foliation. The $0$-lifting of tangent vectors $\Theta_{X/S}$ is to the isomonodromy foliation as the Kodaira-Spencer map is to the Gauss-Manin connection.
\subsection{Chen's formula}\label{section: chen's-formula}
Let $\Sigma/\mb{C}$ be a smooth projective curve of genus $g\geq 2$. Write $M_{\Dol}(\Sigma, r)$ for the coarse moduli space of semistable  Higgs bundles on $\Sigma$, of rank $r$ and degree zero (with trivialized determinant),  and $M^s_{\Dol}(\Sigma, r)\subset M_{\Dol}(\Sigma, r)$ the open set parametrizing stable Higgs bundles. Let $\bun(\Sigma, r)$ denote the coarse moduli space of stable rank $r$ bundles on $\Sigma$. Note that the cotangent bundle $T^*\bun(\Sigma, r)$ has a natural holomorphic symplectic structure, and that there is a natural open immersion $T^*\bun(\Sigma, r)\hookrightarrow  M^s_{\Dol}(\Sigma, r).$

\begin{proposition}\label{prop: holo-sympl-higgs}
    There is a natural holomorphic symplectic form on $M^s_{\Dol}(\Sigma, r)$, extending the one on $T^*\bun(\Sigma, r)$. 
\end{proposition}

\begin{proof}
    See for example \cite[Prop. 3.1, Prop. 3.3, Thm. 4.1]{biswas2021symplectic}. Similar statements have also appeared in \cite{bottacin1995symplectic, markman1994spectral}.

    The symplectic form is derived from Grothendieck-Serre duality and the trace pairing on $\ad(\mscr{E})$, and we now give an explicit (pointwise) formula for it for completeness. At $[(\mscr{E}, \theta)]\in M_{\Dol}^s(\Sigma, r)$, let $C^{\bullet}$ denote the complex 
    \[
    [\ad(\mscr{E}) \rightarrow \ad(\mscr{E})\otimes\Omega^1_{\Sigma}],
    \]
with $\ad(\mscr{E})$ placed in degree 0. Identifying the tangent space at $[(\mscr{E}, \theta)]$ with $\mb{H}^1(C^{\bullet})$, the holomorphic symplectic form is given by 
\[
\mb{H}^1(C^{\bullet})\otimes \mb{H}^1(C^{\bullet}) \rightarrow \mb{H}^2(C^{\bullet}\otimes C^{\bullet})\rightarrow H^1(\Omega^1_{\Sigma})\rightarrow \mb{C};
\]
here,  the second map is induced by the map of complexes 
\[
C^{\bullet}\otimes C^{\bullet} \rightarrow [0 \rightarrow \Omega^1_{\Sigma} \rightarrow 0]
\]
where the non-trivial map is 
    \begin{align*}
        \ad(\mscr{E}) \otimes (\ad(\mscr{E})\otimes \Omega^1_{\Sigma}) \oplus (\ad(\mscr{E})\otimes \Omega^1_{\Sigma}) \otimes \ad(\mscr{E}) &\rightarrow \Omega^1_{\Sigma} \\
        (a_1\otimes b_1) \oplus (a_2\otimes b_2) \mapsto \tr(a_1\circ b_1+a_2\circ b_2).
    \end{align*}
    Here, $a_1, b_2$ are sections of $\ad(\mscr{E})$, and $a_2, b_1$ are sections of $\ad(\mscr{E})\otimes \Omega^1_{\Sigma}$.

Note that this construction in fact yields a symplectic form on the tangent sheaf $T_{\mathscr{M}_{\Dol}(X/S)}$ of the moduli stack of Higgs bundles (and even a $0$-shifted symplectic structure).
\end{proof}

We continue to denote by $h$ the Hitchin morphism
\[
h: \mathscr{M}_{\Dol}(\Sigma, r)\rightarrow \bigoplus_{i=2}^r H^0(\Sigma, \Sym^i\Omega^1_{\Sigma}),
\]
as defined in \autoref{section: hitchin-map}, and $h_2: \mathscr{M}_{\Dol}(\Sigma, r)\rightarrow H^0(\Sigma, \Sym^2\Omega^1_{\Sigma})$ the projection onto the space of quadratic differentials; we refer to $h_2$ as the quadratic Hitchin map.

$\mathscr{M}_{\Dol}(\Sigma, r)$ is the fiber over $[\Sigma]$ of the map $\mathscr{M}_{\Dol}(\mathscr{C}_g/\mathscr{M}_g, r)\to \mathscr{M}_g$, where $\mathscr{M}_g$ is the moduli space of curves of genus $g$ and $\mathscr{C}_g$ is the universal curve. We now explain Chen's computation of the map $\Theta_{\mathscr{C}_g/\mathscr{M}_g}$ in terms of the data above.
\begin{theorem}[Chen's formula]\label{thm:chen-formula}
    Fix $v\in H^1(\Sigma, T_{\Sigma})=T_{\mathscr{M}_g, [\Sigma]}$, and let $v^{\vee}: H^0(\Sigma, \Sym^2\Omega^1_{\Sigma}) \rightarrow \mb{C}$ be the linear map obtained by dualizing and applying Serre duality.

     Then 
    \begin{equation}\label{eqn: chen-formula}
    \Theta_{\mathscr{C}_g/\mathscr{M}_g}(v)=\frac{1}{2}H_{h_2^*v^{\vee}},
    \end{equation} where the right hand side denotes the Hamiltonian vector field of the function $h_2^*v^{\vee}$, with respect to the holomorphic symplectic structure given by \autoref{prop: holo-sympl-higgs}. 
\end{theorem}

\begin{proof}
%    On the open dense $T^*\bun(\Sigma, r)\subset M^s_{\Dol}(\Sigma, r)$, this is due to  Chen \cite[Theorem 2.1]{chen2012associated}. Therefore the difference between the two sides of (\eqref{eqn: chen-formula}) is zero on  an open dense subset, and hence zero, as claimed.

Chen \cite[Theorem 2.1]{chen2012associated} proves this for stable Higgs bundles (as she phrases her results in terms of the coarse moduli space) but in fact the proof is purely \v{C}ech-cohomological and works at the level of stacks.
\end{proof}

We will not use the following, but we feel that it has some significance:
\begin{proposition}
	The lifting of tangent vectors $\Theta_{X/S}$ is tangent to the fibers of the Hitchin map.
\end{proposition}
\begin{proof}
	This is immediate from the definition (computing with e.g.~a \v{C}ech hypercover). One may see it geometrically in the case of a family of curves as follows, from Chen's formula: $\Theta_{\mathscr{C}_g/\mathscr{M}_g}$ is the Hamiltonian vector field associated to a function pulled back from the base of the Lagrangian fibration $h$.
\end{proof}

\subsection{$p$-curvature of the isomonodromy foliation}
Suppose $A$ is a ring in characteristic $p>0$, $M/A$ a smooth scheme, and $\mscr{F}\subset T_{M/A}$ is a foliation, namely a subbundle closed under the Lie bracket on $T_{M/A}$. 

In this paper, we are primarily interested in \emph{horizontal foliations} and generalizations thereof, namely crystals of stacks. We first describe the former.
\begin{definition}
    Suppose $\pi: M\rightarrow S$ is a smooth map of smooth schemes over $A$, and that $\mscr{F}\subset T_{M/A}$ is a foliation. We say that $\mscr{F}$ is a \emph{horizontal} foliation if there is a diagram 
    \[
    \begin{tikzcd}[column sep=scriptsize]
\mscr{F} \arrow[dr, dotted, "\exists \simeq "] \arrow[rr, ""]{}
& & T_{M/A}  \\
& \pi^*T_{S/A} \arrow[ur, "d\pi"].
\end{tikzcd}
\]  
\end{definition}
In other words, $\mscr{F}$ is horizontal if it provides a splitting of the derivative $d\pi$. 

Recall the following:
\begin{definition}[Bost, Ekedahl--Shepherd-Barron--Taylor]
    The $p$-curvature of $\mscr{F}$ is the map
    \begin{align*}
        F^*_{abs}\mscr{F}&\rightarrow T_{M/A}/\mscr{F}\\
        v&\mapsto [v^p].
    \end{align*}    
\end{definition}
\begin{remark}
    It is straightforward to check that this is a $\mscr{O}_M$-linear map; note that this is certainly not true for the map $F^*_{abs}\mscr{F}\rightarrow T_{M/A}$ sending a vector field $v$ to $v^p$: that is, further projecting to $T_{M/A}/\mscr{F}$ is necessary for this extra linearity. This notion recovers the usual one of $p$-curvature of a flat bundle $(\mscr{E}, \nabla)$, when one considers the horizontal foliation on the total space $\mscr{E}$ induced by $\nabla$.
\end{remark}
Since we deal with stacks such as $\pi: \mscr{M}_{\dR}\rightarrow S$ and its cousins in this paper, we need to work with  generalizations of the notion of foliations and  $p$-curvature. We simply summarize the results for $\mscr{M}_{\dR}$ here, and  refer the reader to \cite[Appendix A.3]{p-painleve}  for precise definitions in this case and details of the proofs. In \autoref{subsec:non-abelian-katz-formula} we will deal with the case of $\mscr{M}_{\conj}$, which is a more involved analogue of \cite[Appendix B]{p-painleve}. 

Let $A$ be a characteristic $p>0$ ring as above, $f: X\rightarrow S$ a smooth projective map of smooth $A$-schemes, $\pi_{\dR}: \mscr{M}_{\dR}(X/S)\rightarrow S$ the relative de Rham moduli stack. The following is explained (for the associated crystal of functors) in \cite[Appendices A and B]{p-painleve}.
\begin{proposition}\label{prop: p-curv-mdr-crystal}
    The $p$-curvature of the crystal of stacks  $\pi_{\dR}: \mscr{M}_{\dR}(X/S)\rightarrow S$ (recalled in \autoref{prop: mdr-crystal}) is a map 
    \begin{equation}\label{eqn: nab-p-curv-map}
    F^*_{abs}\pi_{\dR}^*T_S \rightarrow T_{\mscr{M}_{\dR}/S},
    \end{equation}
    where $T_{\mscr{M}_{\dR}/S}$ is the tangent sheaf as defined in \autoref{defn: tangent-sheaf}.  

    At an $R$-point of $\mscr{M}_{\dR}$, corresponding to a flat bundle $(\mscr{E}, \nabla)$ on $X_R/R$ with trivialized determinant $\xi: \det(\mscr{E}, \nabla) \simeq (\mscr{O}_{X_R}, d)$, the $p$-curvature is given by
    \begin{equation}\label{eqn: formula-p-curv}
    F_{abs}^*T_R\rightarrow R^1f_{R*} F_{abs}^*T_{X_R/R} \rightarrow R^1f_{R*}(\End(\mscr{E})^0_{\dR});
    \end{equation}
    Here the second map is the one induced  by the $p$-curvature map $F_{abs}^*T_{X_R/R}\rightarrow \End(\mscr{E})$. The first is the map induced by the Kodaira--Spencer map $T_R\rightarrow R^1f_{R*}T_{X_R/R}$ as well as the sequence of  maps 
    \[
    F^*_{\abs}R^1f_{R*}T_{X_R/R} \rightarrow  R^1f_{R*}^{(1)}T_{X_R^{(1)}/R}= R^1f_{R*}F_{X_R/R}^{-1}T_{X_R^{(1)}/R} \rightarrow R^1f_{R*} F_{abs}^*T_{X_R/R},
    \]
    where the first map is the natural basechange map.
\end{proposition}

\begin{proof}
    See \cite[Appendix A.3]{p-painleve} for the definition, and several equivalent formulations, of the $p$-curvature of a crystal of functors, and \cite[Appendix B.4]{p-painleve} for the formula \eqref{eqn: formula-p-curv}. 

    For the reader's convenience, we briefly recall the main idea for the construction of the $p$-curvature of the crystal $\mscr{M}_{\dR}$. We need the following piece of notation:  for an ideal $\bar{I}$ equipped with a divided power structure $\gamma$, for $k\geq 1$ we write $\bar{I}^{[k]}$ for the ideal generated by products of the form $\gamma_{i_1}(x_1)\cdots \gamma_{i_n}(x_n)$ for any $n\geq 1$, $i_1+\cdots +i_n\geq k$, and $x_i\in \bar{I}$. As in \autoref{prop: mdr-crystal}, the crystal structure is the data of an isomorphism $\epsilon: \pi_1^*\mscr{M}_{\dR}\rightarrow \pi_2^*\mscr{M}_{\dR}$ of functors over $S(1)$, satisfying the cocycle condition. 

    Consider the subscheme $V(I, \bar{I}^{[p+1]})\subset S(1)$. Since it is a subscheme contained in $V(I)$, we have a canonical identification $\pi_1|_{V(I+\bar{I}^{[p+1]})}= \pi_2|_{V(I+\bar{I}^{[p+1]})}$, and we denote these maps  by the same symbol $\pi_{\eq}$. Now $\epsilon|_{V(I+\bar{I}^{[p+1]})}$ is an infinitesimal automorphism of $\pi_{\eq}^*\mscr{M}_{\dR}$. Using the identification $\bar{I}/(I+\bar{I}^{[p+1]})\simeq F^*_{\abs}\Omega^1_S$ (see \autoref{prop:frobdifferentials}), we may convert this infinitesimal automorphism into a map of the form \eqref{eqn: nab-p-curv-map}. 
\end{proof}
\begin{remark}
    The exact same  recipe as above also allows us to define the $p$-curvature of any crystal of functors, in particular the $p$-curvature of $\mscr{M}_{\conj}$.  
\end{remark}
\subsection{The non-abelian Katz formula}\label{subsec:non-abelian-katz-formula}
We assume that $S$ is a smooth scheme over a characteristic $p$ field $k$ in this section, and as usual $f: X\rightarrow S$ is a smooth projective morphism. Recall that in \autoref{subsec:M_conj} we defined the conjugate moduli stack $\piconj: \mscr{M}_{\conj}(X/S)\rightarrow S\times \mb{A}^1$, with  canonical isomorphisms $\mscr{M}_{\conj}|_{\lambda=1} = \mscr{M}_{\dR}(X/S)$ and $\mscr{M}_{\conj}|_{\lambda=0} = \mscr{M}_{\Dol}(X^{(1)}/S)$.

\begin{comment}
\begin{definition}[Relative crystalline site]
    Let $T\rightarrow R$ be a map of schemes. The (relative) crystalline site is the site $Cris(T/R)$ whose objects are triples $(U, V, \delta)$ where $U\subset T$ is a Zariski open, $U\subset V$ is a nilpotent thickening of $R$-schemes with ideal $I$, and $\delta$ a divided power structure on $I$. The morphisms are the obvious ones of triples, and covers are taken with respect to $V$.  

    The restricted crystalline site is the full subcategory of $Cris(T/R)$ consisting of objects $(U, V, \delta)$ where $V$ admits a map $r: V\rightarrow T$ making the diagram 
    \begin{equation}
        \begin{tikzcd}[column sep=scriptsize]
U \arrow[dr] \arrow[rr]{}
& & V \arrow[dl, "r"] \\
& T 
\end{tikzcd}
\end{equation}
commute.
\end{definition}

\begin{proposition}[Menzies]
    Suppose $T\rightarrow R$ is smooth. Then the  natural functor from crystals (of sheaves, or schemes, or functors) on  $Cris(T/R)$ to crystals on $Cris^r(T/R)$ is an  equivalence.
\end{proposition}
\begin{proof}
    
\end{proof}
\end{comment}

We  denote  the completed PD envelope of the ideal sheaf $I$ of the  diagonal in $S\times S$ by $(S(1), \bar{I}, \gamma)$, and by $((S\times \mb{A}^1)(1), \overline{I}_{\mb{A}^1, \gamma_{\mb{A}^1}})$ the completed PD envelope of the diagonal in $((S\times \mb{A}^1)\times_{\mb{A}^1} (S\times \mb{A}^1))$; the former was denoted by $(\sxspd, \bar{I}, \gamma)$ in \cite[Appendix A.3]{p-painleve}, but we simplify the notation here since we also have the version over $\mb{A}^1$. It is straightforward to check that the two projection maps $S(1) \rightrightarrows S$ give $S(1)$ the structure of a groupoid, and similarly for $(S\times \mb{A}^1)(1)/\mb{A}^1$. We denote the diagonal maps $S\rightarrow S(1)$ and $S\times \mb{A}^1\rightarrow (S\times \mb{A}^1)(1)$ both by $\Delta$, and hope it will not cause any confusion.

\begin{proposition}\label{prop: crystal-mconj}
   The stack $\mscr{M}_{\conj}(X/S)$ has the structure of a crystal of functors on $Cris(S\times \mb{A}^1/\mb{A}^1)$. Equivalently, there is an action of the groupoid $(S\times \mb{A}^1)(1)$ on $\mscr{M}_{\conj}$. 

   When evaluated at $1\in \mb{A}^1$, this crystal structure recovers that of $\mscr{M}_{\dR}(X/S)$ over $S$, which was recorded in \autoref{prop: mdr-crystal}.
\end{proposition}
\begin{proof}
    For brevity, we write simply $\mscr{M}_{\conj}, \mscr{M}_{\dR}$ for $\mscr{M}_{\conj}(X/S), \mscr{M}_{\dR}(X/S)$.  By definition, we must give an equivalence 
    \[
    \epsilon: \pi_1^*\mscr{M}_{\conj} \simeq \pi_2^*\mscr{M}_{\conj},
    \]
    with $\pi_i$ being the projection maps $(S\times \mb{A}^1)(1)\rightrightarrows S\times \mathbb{A}^1$, which moreover satisfies the cocycle condition. 

    For simplicity of notation, write $Y:= X\times \mb{A}^1$. For a $T$ point $t: T \to (S\times \mb{A}^1)(1)$, and $i=1,2$, we define $Y_{T, i}$, $Y_{(S\times \mb{A}^1)(1), i}$ by the diagram where each square is a pullback square:
    \[
    \begin{tikzcd}
Y_{T, i} \arrow[r, ""] \arrow[d]
& Y_{(S\times \mb{A}^1)(1), i} \arrow[r] \arrow[d] &Y \arrow[d] \\
T \arrow[r, "t"]
& (S\times \mb{A}^1)(1) \arrow[r, "\pi_i"] &  S\times \mb{A}^1
\end{tikzcd}
\]
Then 

\[\pi^*_i\mscr{M}_{\conj}(T)=
\left\{
(\mscr{E}, \nabla, \theta)\Bigg|
\begin{aligned}
 & (\mscr{E}, \nabla)\  \text{ is an integrable connection on } Y_{T, i}/T\\
 &\theta: \mscr{E}\to \mscr{E}\otimes F_{abs}^*\Omega^1_{Y_{T, i}/T} \text{ is a  horizontal F-Higgs field,} \\
&  \text{such that } \psi_p(\mscr{E}, \nabla) = \lambda \theta.  \\
\end{aligned}
\right\}
\]
where $\lambda$ is the composition of $\pi_i\circ t$ with the projection $S\times \mathbb{A}^1\to \mathbb{A}^1$ (which by assumption is independent of $i$, as we are working over $\mathbb{A}^1$).

Then, given $(\mscr{E}, \nabla, \theta)\in \pi_1^*\mscr{M}_{\conj}(T)$, we must specify $\epsilon(\mscr{E}, \nabla, \theta)\in \pi_2^*\mscr{M}_{\conj}(T)$, and check the cocycle conditions. Note that there are obvious forgetful maps $\pi_i^*\mscr{M}_{\conj} \rightarrow \pi_i^*\mscr{M}_{\dR}$. Now from the crystal structure on $\mscr{M}_{\dR}$ recorded in \autoref{prop: mdr-crystal}, for each $(\mscr{E}, \nabla) \in \pi_1^*\mscr{M}_{\dR}(T)$ we obtain $\epsilon(\mscr{E}, \nabla) \in \pi_2^*\mscr{M}_{\dR}(T)$; equivalently, for each flat bundle $(\mscr{E}, \nabla)$ on $Y_{T,1}/T$, we obtain a flat bundle $\epsilon(\mscr{E}, \nabla)$ on $Y_{T,2}/T$. It remains to specify $\epsilon$ on the F-Higgs field $\theta$.

By \autoref{prop: isomonodromy-on-p-curv-0}, $\epsilon(F^*_{abs}\Omega^1_{Y_{T,1}/T}) = F^*_{abs}\Omega^1_{Y_{T,2}/T}$, where both $F^*_{abs}\Omega^1_{Y_{T,1}/T}$ and $F^*_{abs}\Omega^1_{Y_{T,2}/T}$ are equipped with their canonical connections,   and by the functoriality in \autoref{prop: mdr-crystal}, we obtain  a horizontal F-Higgs field $\epsilon(\theta): \epsilon(\mscr{E}, \nabla) \rightarrow \epsilon(\mscr{E}, \nabla)\otimes F^*_{abs}\Omega^1_{Y_{T,2}/T}$.  It is straightforward to check that the resulting map $\epsilon: \pi_1^*\mscr{M}_{\conj} \rightarrow \pi_2^*\mscr{M}_{\conj}$ is an equivalence and satisfies the cocycle condition.  

Finally, by construction, the evaluation of $\epsilon$ at $1\in \mb{A}^1$ is precisely the crystal structure on $\mscr{M}_{\dR}$.

\end{proof}

Let $\Psi_{\conj}$ denote the $p$-curvature of the crystal given in \autoref{prop: crystal-mconj}, where we use the notion of $p$-curvature of a crystal of functors as defined in \cite[\S A.3]{p-painleve}. As explained in \cite[Proposition/Construction A.3.2]{p-painleve}, we may view it as a map
\[
\Psi_{\conj}: \pi_{\conj}^*F_{abs}^*T_{S\times \mb{A}^1/\mb{A}^1} \rightarrow T_{\mscr{M}_{\conj}/S\times \mb{A}^1}.
\]
Recall also, from \autoref{defn:Hodge-lifting}, the lifting of tangent vectors for the Dolbeault moduli space   $\pi_{\Dol}^{(1)}: \mscr{M}_{\Dol}(X^{(1)}/S)\rightarrow S$, which we write as a map of sheaves
\[
\Theta_{X^{(1)}/S}: (\pi_{\Dol}^{(1)})^*T_S\rightarrow T_{\mscr{M}_{\Dol}(X^{(1)}/S)/S}.
\]

We now write down a cocycle representing $\Psi_{\conj}$; this is very similar to the cocycle for the $p$-curvature of $\mscr{M}_{\dR}$ given in  \cite[B.4.2]{p-painleve} and the text immediately before it. 

To give a cocycle representing $\Psi_{\conj}$ means, for each choice of $t: T\rightarrow \mscr{M}_{\conj}$,    to give a cocycle on $X_T:= (X\times \mb{A}^1)\times_{S\times \mb{A}^1} T$, compatibly in the choice of $t$.

So, suppose we are given such a $t$, and denote by $h$ the composition $h: T\xrightarrow{t} \mscr{M}_{\conj}\xrightarrow[]{\pi_{\conj}} S\times \mb{A}^1$; define  $X_T$ as the fiber product
\[
\begin{tikzcd}
    X_T \arrow[r, "\tilde h"] \arrow[d] & X\times \mb{A}^1 \arrow[d] \\
    T \arrow[r, "h"] & S\times \mb{A}^1.
\end{tikzcd}
\]

 By definition, the map $t: T\rightarrow \mscr{M}_{\conj}$ corresponds to a triple $(\mscr{E}_T, \nabla, \theta)$ where $(\mscr{E}_T, \nabla)$ is a flat bundle on $X_T/T$, and $\theta$ a F-Higgs field on $\mscr{E}_T$, satisfying the conditions stated  in \autoref{defn: mconj}.

\begin{definition}
    Let $\mscr{M}_{\conj}^{
    \lambda-\tf
    }$ be the subfunctor of $\mscr{M}_{\conj}$ consisting of maps $t: T\to \mathscr{M}_{\conj}$ such that $t^*T_{\mscr{M}_{\conj}/S\times \mb{A}^1}$ is $\lambda$-torsion free.   
\end{definition}
\begin{remark}
    In \autoref{lemma:equisingularity} we will show that $\mscr{M}_{\conj}^{
    \lambda-\tf
    }$ is sufficiently large for our purposes, namely that, when $p> \dim(X/S)$,  it contains the locus of   $(\mscr{E}, \nabla, 
    \theta)$ where $\theta$ has order of nilpotence $\leq (p-1)/2$. 
\end{remark}
There is a Cartesian square:
$$\xymatrix{
\mathscr{M}_{\Dol}(X^{(1)}/S)\ar[r]^G\ar[d]^{\pi^{(1)}_{\Dol}} & \mathscr{M}_{\Dol}(X/S) \ar[d]^{\pi_{\Dol}} \\
S \ar[r]^{F_{\text{abs}}} & S
}$$
Write $f_{\mb{A}^1}: X\times \mb{A}^1\rightarrow S\times \mb{A}^1$ for the basechange of $f$ along $\mb{A}^1\rightarrow \spec(k)$. 

Our main result of this section, which will be proven in \autoref{subsec:proof-of-NA-Katz}, is:
\begin{theorem}[Non-abelian Katz formula]\label{lemma:katz-formula}
    We have $\Psi_{\conj}|_{\lambda=0} =0$.  Therefore, after restricting to $\mscr{M}_{\conj}^{\lambda-\tf
    }$, it makes sense to define the map of sheaves $\frac{\Psi_{\conj}}{\lambda}$. Then, on $\mscr{M}_{\conj}^{\lambda-\tf
    }|_{\lambda=0}$, there is a commutative diagram 
 \[\begin{tikzcd}
 (F_{abs}^* \pi_{\conj}^*T_{S\times \mb{A}^1/\mb{A}^1})|_{\lambda=0} \arrow[r] \arrow[rr, bend left=15, "\frac{\Psi_{\conj}}{\lambda}\big |_{\lambda=0}"] & F_{abs}^*(\pi_{\conj}^*R^1f_{\mb{A}^1*}T_{X\times \mb{A}^1/S\times \mb{A}^1})|_{\lambda=0} \arrow[r] 
& T_{\mscr{M}_{\conj}/S\times \mb{A}^1}|_{\lambda=0}  \\
(\pi^{(1)}_{\Dol})^*F_{abs}^*T_{S} \arrow[r, ""] \arrow[u, "\simeq"] \arrow[rr, bend right=15, "G^*\Theta_{X/S}"']
& (\pi^{(1)}_{\Dol})^*F_{abs}^*R^1f_*T_{X/S} \arrow[r] \arrow[u] 
&G^* T_{\mscr{M}_{\Dol}(X/S)/S} \arrow[u, "\simeq"]
\end{tikzcd}
\]
    
\end{theorem}

\subsection{Kodaira--Spencer cocycle}
Let $A$ be a ring, $f: X\rightarrow S$ a smooth map of smooth $A$-schemes. We have the exact sequence of sheaves on $X$ 
\[
0\rightarrow T_{X/S} \rightarrow T_X\rightarrow f^*T_S\rightarrow 0,
\]
which induces a map (using the projection formula and the long exact sequence in cohomology)
\[
\KS: T_S\rightarrow R^1f_*T_{X/S};
\]
we refer to this  as the \emph{Kodaira--Spencer map}. The following is immediate from unwinding the definition of $\KS$. 
\begin{proposition}\label{prop: ks-cocycle}
    Let $S^{(2)}\subset S\times S$ be the first order neighborhood of the diagonal, and $\pi_{1,2}: S^{(2)}\rightarrow S$ the two projection maps. 

    Let $\{U_{\beta}\}_{\beta\in I}$ be a covering of $X$ by open affines. Then the Kodaira--Spencer map, viewed as a section of $\Omega^1_S\otimes R^1f_*T_{X/S}$, can be represented by the cocycle $\kappa_{\beta\gamma} \in (f^*\Omega^1_S \otimes T_{X/S}) (U_{\beta\gamma})$ given, in local coordinates $x_1, \cdots, x_d$ on $X/S$, by 
    \[
    \kappa_{\beta\gamma} = \sum_{k=1}^d  (f_{\beta}^*(x_k)- f_{\gamma}^*(x_k))\otimes \partial_{x_k}.
    \]
\end{proposition}
See \cite[Appendix B.3]{p-painleve} for some further discussion.
\subsection{Proof of the non-abelian Katz formula}\label{subsec:proof-of-NA-Katz}
In this section we keep the notation as in \autoref{subsec:non-abelian-katz-formula}. In particular, $k$ is a characteristic $p>0$ field, $S$ a smooth scheme over $k$;   the completed PD envelope of the ideal sheaf $I$ of the  diagonal in $S\times S$ is denoted by $(S(1), \bar{I}, \gamma)$.
%Let $I$ be the ideal sheaf on $S(1)$ defining the diagonal $\Delta: S \rightarrow S(1)$, and let $\bar{I}$ denote the PD ideal of $I$. 

We first record the following description of the Frobenius pullback of K\"ahler differentials. Recall that for $k\geq 1$, we write $\bar{I}^{[k]}$ for the ideal generated by products of the form $\gamma_{i_1}(x_1)\cdots \gamma_{i_n}(x_n)$ where $x_j\in \bar{I}$ and   $i_1+\cdots +i_n\geq k$.
\begin{proposition}\label{prop:frobdifferentials}
%, or that it arises as a Frobenius pullback, i.e. it is of the form $Z^{[p]}$ for some closed subscheme $Z\xhookrightarrow{} X$, with $X/\on{Spec}(R)$ smooth, as in \autoref{definition:p-power-leaves}.
There is a natural isomorphism 
\[F_{\text{abs}}^*\Omega^1_S\simeq \bar{I}/(\bar{I}^{[p+1]}+I),\]
given by 
\[
F_{\text{abs}}^*da \mapsto \gamma_p(a\otimes 1-1\otimes a)
\]
for local sections $a$ of $\mscr{O}_S.$
\end{proposition}
\begin{proof}
    See \cite[Proposition~1.6]{ogus-vologodsky}.
    \end{proof}
\begin{proof}[Proof of \autoref{lemma:katz-formula}]
%We will give a formula for the map $\frac{\Psi_{\conj}}{\lambda}|_{\lambda=0}$. 

We pick open affines $U_{\beta}\subset X$ covering $X$, where $\beta$ ranges over some indexing set $I$,  and write 
\[
(\widehat{U_{\beta}\times \mb{A}^1})_1:= (U_{\beta
}\times \mb{A}^1)\times_{S\times \mb{A}^1, \pi_1} (S\times \mb{A}^1)(1),
\]
and similarly $(\widehat{U_{\beta}\times \mb{A}^1})_2.$ Now each of  $(\widehat{U_{\beta}\times \mb{A}^1})_1$, $(\widehat{U_{\beta}\times \mb{A}^1})_2$ is a deformation of $U_{\beta}\times \mb{A}^1$ over $(S\times \mb{A}^1)(1)$, and since the latter is affine, there exists an isomorphism
\[
\tilde{f}_{\mb{A}^1, \beta}: (\widehat{U_{\beta}\times \mb{A}^1})_2 \simeq (\widehat{U_{\beta}\times \mb{A}^1})_1
\]
over $(S\times \mb{A}^1)(1)$. We write $f_{\mb{A}^1, \beta}$ for the composition 
\[
f_{\mb{A}^1, \beta}: (\widehat{U_{\beta}\times \mb{A}^1})_2 \xrightarrow[]{\tilde{f}_{\mb{A}^1, \beta}} (\widehat{U_{\beta}\times \mb{A}^1})_1\rightarrow U_{\beta}\times \mb{A}^1.
\]

We also write $U_{\beta\gamma}:= U_{\beta}\cap U_{\gamma}$, and $(\widehat{U_{\beta\gamma}\times \mb{A}^1})_{1}, (\widehat{U_{\beta\gamma}\times \mb{A}^1})_{2}$ as above, $U_{\beta\gamma, T} \subset X_T$ for the base change of $U_{\beta\gamma}\times \mb{A}^1$ along $T\rightarrow S\times \mb{A}^1$.

In local coordinates $\{x_1, \cdots, x_d\}$ on $X/S$, and therefore on $X\times \mb{A}^1/S\times \mb{A}^1$,  consider the elements $c_{\beta\gamma, \mscr{E}_T} \in \tilde{h}^*f_{\mb{A}^1}^*\bar{I} \otimes \End(\mscr{E})|_{U_{\beta\gamma, T}\times \mb{A}^1}$ given by 
\[
s\otimes 1 \mapsto \sum_{\underline{\ell}} \nabla_{\partial_{x_{\underline{\ell}}}}(s) \otimes \prod_{m=1}^{d} \tilde{h}^* \gamma_{\ell_m} (f_{\mb{A}^1, \beta}^*(x_m)-f_{\mb{A}^1, \gamma}^*(x_m)),
\]
where the sum is over tuples $\underline{\ell}=(\ell_1, \cdots, \ell_d)$ with $\ell_m\geq 0$, and  $\nabla_{\partial_{x_{\underline{\ell}}}}$ denotes the operator $\prod_{m=1}^d \nabla_{\partial_{x_m}}^{\ell_m}$, and $s$ is a local section to $\mathscr{E}$. 

It is straightforward to check that the $c_{\beta\gamma, \mscr{E}_T}$ form a cocycle representing a class in $\tilde{h}^*f_{\mb{A}^1}^*\bar{I}\otimes R^1f_{T, *}\on{End}(\mscr{E}_{T, dR})$. In fact, the collection of such cocycles, as $T$ varies, is equivalent data to the crystal structure on $\mscr{M}_{\conj}$.

Now, using the description of the crystal structure on $\mscr{M}_{\conj}$ given in \autoref{prop: crystal-mconj} and the description of its $p$-curvature recalled in the proof of \autoref{prop: p-curv-mdr-crystal}, we will obtain a cocycle representative of $\Psi_{\conj}$--see  \cite[Lemma B.4.2]{p-painleve} for the analogous computation for $\mscr{M}_{\dR}$. That is,  after projecting the $c_{\beta\gamma, \mscr{E}_T}$ along the natural map 
\[
\tilde{h}^*f_{\mb{A}^1}^*\bar{I} \rightarrow  \tilde{h}^*f_{\mb{A}^1}^* (\bar{I}/(\bar{I}^{[p+1]}+I)),
\]
and using the isomorphism $\bar{I}/(\bar{I}^{[p+1]}+I)\simeq F^*_{\abs}\Omega^1_S$, we obtain a cocycle representing $\Psi_{\conj}$; this latter cocycle is given by $d_{\beta\gamma, \mscr{E}_T} \in \tilde{h}^*f_{\mb{A}^1}^*(\bar{I}/ \bar{I}^{[p+1]}+I) \otimes \End(\mscr{E}_T)|_{U_{\beta\gamma, T}}$: 
\begin{equation}\label{eqn: p-curv-conj-cocycle}
    s\otimes 1 \mapsto \sum_{k=1}^d \nabla_{\partial_{x_k}}^p(s) \otimes \gamma_p(f^*_{\mb{A}^1, \beta}(x_k)- f^*_{\mb{A}^1, \gamma}(x_k)).
\end{equation}

Examining \eqref{eqn: p-curv-conj-cocycle}, we see  that $\Psi_{\conj}|_{\lambda=0}$ vanishes, since $(\mscr{E}_T, \nabla)$ has vanishing $p$-curvature  upon restriction to $\lambda=0$. Indeed, by this assumption, $$\nabla_{\partial_{x_k}}^p(s)=\nabla_{\partial_{x_k}^p}(s)=0.$$ 

Next, we compute $\frac{\Psi_{\conj}}{\lambda}\big|_{\lambda=0}$.
For a general $t:T\to \mscr{M}_{\conj}$, the defining
condition $\psi_p(\mscr{E}_T,\nabla)=\lambda\cdot\theta$ implies, for each
coordinate vector field $\partial_{x_k}$,
that
\[
\nabla(\partial_{x_k})^p=\psi_p(\mscr{E}_T,\nabla)(\partial_{x_k})
=\lambda\cdot \theta(\partial_{x_k}).
\]
Plugging this into \eqref{eqn: p-curv-conj-cocycle} yields
\[
d_{\beta\gamma,\mscr{E}_T}(s\otimes 1)
=\lambda\cdot \sum_{k=1}^d \theta(\partial_{x_k})(s)\ \otimes\
\gamma_p\!\big(f^*_{\mb{A}^1,\beta}(x_k)-f^*_{\mb{A}^1,\gamma}(x_k)\big).
\]
In particular, $\Psi_{\conj}$ is divisible by $\lambda$ at the level of cocycles.
On the locus where $T_{\mscr{M}_{\conj}/S\times \mb{A}^1}$ is $\lambda$-torsion free, namely $\mscr{M}_{\conj}^{\lambda-\tf}$, 
this cocycle therefore determines $\frac{\Psi_{\conj}}{\lambda}$, and
$\frac{\Psi_{\conj}}{\lambda}\big|_{\lambda=0}$ is represented by the cocycle
\begin{equation}\label{eqn:psi-over-lambda-cocycle}
s\otimes 1\ \longmapsto\
\sum_{k=1}^d \theta(\partial_{x_k})(s)\ \otimes\
\gamma_p\!\big(f^*_{\mb{A}^1,\beta}(x_k)-f^*_{\mb{A}^1,\gamma}(x_k)\big)
\qquad\text{on }U_{\beta\gamma,T}.
\end{equation}

Finally, we identify \eqref{eqn:psi-over-lambda-cocycle} with the Frobenius pullback
of the lifting of tangent vectors.
By \autoref{prop: ks-cocycle}, the Kodaira--Spencer map for $f_{\mb{A}^1}$, viewed as
a section of $\Omega^1_{S\times \mb{A}^1/\mb{A}^1}\otimes R^1 f_{\mb{A}^1*}T_{X\times \mb{A}^1/S\times \mb{A}^1}$,
is represented on $U_{\beta\gamma}$ by the cocycle
\[
\kappa_{\beta\gamma}
=\sum_{k=1}^d \big(f^*_{\mb{A}^1,\beta}(x_k)-f^*_{\mb{A}^1,\gamma}(x_k)\big)\otimes \partial_{x_k}.
\]
Pulling back along absolute Frobenius and using \autoref{prop:frobdifferentials}
(which identifies $F_{\abs}^*\Omega^1_{S\times \mb{A}^1/\mb{A}^1}$ with
$\bar{I}/(\bar{I}^{[p+1]}+I)$ via $F_{\abs}^*da\mapsto\gamma_p(a\otimes 1-1\otimes a)$),
we see that $F_{\abs}^*\kappa_{\beta\gamma}$ is represented by
\[
\sum_{k=1}^d \gamma_p\!\big(f^*_{\mb{A}^1,\beta}(x_k)-f^*_{\mb{A}^1,\gamma}(x_k)\big)\otimes \partial_{x_k}
\in F_{\abs}^*\Omega^1_{S\times \mb{A}^1/\mb{A}^1}\otimes T_{X\times \mb{A}^1/S\times \mb{A}^1}.
\]
Composing this with the map $T_{X\times \mb{A}^1/S\times \mb{A}^1}\to \End(\mscr{E}_T)$
induced by the $F$-Higgs field $\theta$ (i.e.~$\partial\mapsto \theta(\partial)$)
produces exactly the cocycle \eqref{eqn:psi-over-lambda-cocycle}.
Under the identification $\mscr{M}_{\conj}|_{\lambda=0}\simeq \mscr{M}_{\Dol}(X^{(1)}/S)$
and the Cartesian square defining $G$, this is precisely the pullback
$G^*\Theta_{X/S}$, and therefore the diagram in the statement of
\autoref{lemma:katz-formula} commutes.

\begin{comment} Let $\mscr{M}_{\FDol}$ denote the moduli stack of F-Higgs bundles on $X/S$. We have an evident map 
    \[
    j: \mscr{M}_{\conj} \rightarrow \mscr{M}_{\dR} \times \mscr{M}_{\FDol}.
    \]
    Since the map on tangent sheaves is injective, it suffices to prove that $j_*\Psi_{\conj}|_{\lambda=0} = 0$, and that 
    \[
     \frac{j_*\Psi_{\conj}}{\lambda}\bigg\rvert_{\lambda=0} = F^*\Theta.
    \]
\end{comment}
\end{proof}

\begin{remark}
	We believe some form of \autoref{lemma:katz-formula} is likely implicit in modern $p$-adic Hodge theory, specifically in the geometry of the \emph{syntomification} of a characteristic $p$ scheme \cite[\S6]{bhatt-f-gauge}, though some work would be required to extract the statement here.
\end{remark}
\section{Vanishing of $p$-curvature}\label{sec: vanish-p-curv}
In this section we recall Ogus--Vologodsky's non-abelian Hodge theory in characteristic $p>0$ \cite{ogus-vologodsky} and use it to construct so-called \enquote{canonical sections} to the conjugate moduli stack. We deduce from the existence of these sections that the vanishing of the $p$-curvature of the isomonodromy foliation for infinitely many primes $p$ implies the vanishing of $\Theta_{X/S}$ in characteristic zero, in \autoref{thm:p-curvature-vanishing-implies-theta-vanishing}.
\subsection{Non-abelian Hodge theory in positive characteristic}\label{subsec:NAHTcharp}
\subsubsection{Recollection of the non-abelian Hodge correspondence and exponential twisting}
Fix an $\mathbb{F}_p$-scheme $S$ and a flat lift $\widetilde{S}$ of $S$ over $\mathbb{Z}/p^2\mathbb{Z}$. Let $X$ be a smooth projective $S$-scheme, and fix a lift $\widetilde{X}^{(1)}$ of $X^{(1)}$ over $\widetilde{S}$. In this section we briefly recall some facts about Ogus--Vologodsky's non-abelian Hodge correspondence, which we will use to understand some of the geometry of $\mathscr{M}_{\conj}(X/S)$.

\begin{definition}
Fix a scheme $Y/S$ and an integer $\ell \geq 0$.
\begin{enumerate}
    \item We write  $\on{MIC}_{\ell}(Y/S)$ for the category of pairs $(\mscr{E}, \nabla)$ consisting of a quasi-coherent $\mscr{O}_Y$-module $\mscr{E}$, equipped with an integrable connection $\nabla: \mscr{E}\rightarrow \mscr{E}\otimes \Omega^1_{Y/S}$, such that the $p$-curvature $\psi=\psi_p(\mathscr{E},\nabla)$ satisfies   $\psi^{\circ \ell}: \mscr{E}\rightarrow \mscr{E}\otimes (F^*_{\abs}\Omega^1_{Y/S})^{\otimes \ell}$ is zero.
    \item We write $\on{HIG}_{\ell}(Y/S)$ for the category of pairs $(\mscr{E}, \theta)$ consisting of a quasi-coherent $\mscr{O}_Y$-module $\mscr{E}$, equipped with an integrable Higgs field $\theta: \mscr{E}\rightarrow \mscr{E}\otimes \Omega^1_{Y/S}$, such that the natural map $\theta^{\circ \ell}: \mscr{E}\rightarrow \mscr{E}\otimes (\Omega^1_{Y/S})^{\otimes \ell}$ is zero.
\end{enumerate}
    We say that such connections and Higgs bundles are of order of nilpotence at most $\ell$. 
\end{definition}

\begin{theorem}[Ogus--Vologodsky \cite{ogus-vologodsky}]\label{thm:ogus-vologodsky}
	There is a functorial equivalence of categories (depending on the choice of lift $\widetilde{X}^{(1)}$) $$C^{-1}: \on{HIG}_{p-1}(X^{(1)}/S)\overset{\sim}{\longrightarrow} \on{MIC}_{p-1}(X/S),$$  referred to as the \emph{inverse Cartier transform}, with the following properties:
	\begin{enumerate}
		\item $C^{-1}$ is compatible with pullback, i.e. given a smooth $S$-scheme $Y$, a map $f: Y\to X$, a lift $\widetilde{Y}^{(1)}/\widetilde{S}$ and a map $\widetilde{f}^{(1)}$ lifting $f^{(1)}$ so  the diagram 
		$$\xymatrix{
		Y \ar[r]^f \ar[rd]^{F_{Y/S}} & X \ar[rd]^{F_{X/S}} & \\
		& Y^{(1)} \ar[r]^{f^{(1)}} \ar@{^(->}[d] & X^{(1)} \ar@{^(->}[d] \\
		& \widetilde{Y}^{(1)} \ar[r]^{\widetilde{f}^{(1)}} & \widetilde{X}^{(1)}
		}$$
		over $\widetilde{S}$ commutes, we have a natural isomorphism of functors $$\on{HIG}_{p-1}(X^{(1)}/S)\to \on{MIC}_{p-1}(Y/S)$$ $$C^{-1} \circ f^{(1)*}\simeq f^*\circ C^{-1}.$$
        \item Suppose $(\mscr{E}', \theta')$ is an object of $\on{HIG}_{p-1}(X^{(1)}/S)$, and write $(\mscr{E}, \nabla)$  for $C^{-1}(\mscr{E}', \theta')$.

        We say that a morphism of $\widetilde{S}$-schemes
        \[\widetilde{F}: \widetilde{X}\rightarrow \widetilde{X}^{(1)}\] 
        is a lift of relative Frobenius if it reduces to $F_{X/S}$ modulo $p$. Note that such a lift includes a choice of a lift $\widetilde{X}$ of $X$. Each  lift $\widetilde{F}$ of relative Frobenius induces an isomorphism
        \[
       \eta_{\widetilde{F}}:  (\mscr{E}, \psi) \overset{\sim}{\to} F_{X/S}^*(\mscr{E}', \theta'),
        \]
        where $\psi$ denotes the $p$-curvature map of $(\mscr{E}, \nabla)$. Colloquially, the $p$-curvature of the inverse Cartier transform is simply the Frobenius twist of the Higgs field, in the presence of a lift of Frobenius.
        \item \cite[Remark 2.10]{ogus-vologodsky} Given two lifts of Frobenius $\widetilde{F}_1, \widetilde{F}_2$ as above, differing by a section $\xi$ of $F_{X/S}^*T_{X^{(1)}/S}$, we have $$\eta_{\widetilde{F}_2}=e^\xi\circ \eta_{\widetilde{F}_1}.$$
	\end{enumerate}
\end{theorem}

\begin{construction}\label{construction:exp-twist}
	Combining \autoref{thm:ogus-vologodsky}(1), (2), and (3) allows us to give an explicit description of $C^{-1}$ (following \cite{lan-sheng-zuo}). Starting from an object $(\mscr{E}', \theta')$ of $\on{HIG}_{p-1}(X^{(1)}/S)$, we construct a flat bundle on $X/S$ as follows.
    
    Cover $X/S$ by open affines $U_{\alpha}/S$, such that each $U_{\alpha}/S$  is equipped with a lift of relative Frobenius
    \[
\widetilde{F}_{\alpha}: \widetilde{U}_{\alpha} \rightarrow 
\widetilde{U}_{\alpha}^{(1)}
    \]over $\widetilde{S}$. By Deligne--Illusie \cite[p. 251]{deligne-illusie}, each lift of relative Frobenius provides us with a map 
    \[
    \zeta_{\alpha}: F_{U_{\alpha}/S }^*\Omega^1_{U_{\alpha}^{(1)}/S} \rightarrow \Omega^1_{U_{\alpha}/S},
    \]
    which is the \enquote{divided Frobenius pullback} map, defined as the unique map making the following commute: 
    \[
    \begin{tikzcd} \widetilde{F}_{\alpha}^*\Omega^1_{\widetilde{U_{\alpha}}^{(1)}/\widetilde{S}} \arrow[r, "\widetilde{F}_{\alpha}^*"] \arrow[d]
    & p\Omega^1_{\widetilde{U}_{\alpha}/\widetilde{S}}  \\
  F_{U_{\alpha}/S}^*\Omega^1_{U_{\alpha}^{(1)}/S} \arrow[r, "\zeta_{\alpha}"]
&  \Omega^1_{U_{\alpha}/S}. \arrow[u, "\simeq", "p"']\end{tikzcd}
\]

It is straightforward to check that, for a given lift $\widetilde{U}_{\alpha}$ of $U_{\alpha}$, the set of lifts of Frobenius is a torsor under $H^0(U_{\alpha}, F_{U_{\alpha}/S}^*T_{U_{\alpha}^{(1)}/S})$. Write $U_{\alpha\beta}:=U_{\alpha}\cap U_{\beta}$. Suppose that, on $U_{\alpha\beta}$, the Frobenius lifts $\widetilde{F}_{\alpha}, \widetilde{F}_{\beta}$ differ by the section $\xi_{\alpha \beta}$ of  $F^*_{U_{\alpha\beta}/S}T_{U_{\alpha\beta}^{(1)}/S}$; write 
    \[
    h_{\alpha\beta}: F_{U_{\alpha\beta}/S}^*\Omega^1_{U^{(1)}_{\alpha\beta}/S}\rightarrow \mscr{O}_{U_{\alpha\beta}}
    \]
    for the map given by pairing with the section $\xi_{\alpha \beta}$.

    On $U_{\alpha}$ consider the  bundle $H_{\alpha}:=F^*_{U_{\alpha}/S}\mscr{E}'|_{U_{\alpha}^{(1)}}$, equipped with the connection 
    \[
    \nabla_{\alpha}:= \nabla_{\text{can}}-\zeta_{\alpha}(F_{U_{\alpha}/S}^*\theta');
    \]
    note that we have chosen to write the formula with the opposite sign to that of \cite[\S 2.2]{lan-sheng-zuo}, since they show that their construction agrees with that of Ogus--Vologodsky's after a sign-change. 
    
    We then glue the bundles $H_{\alpha}|_{U_{\alpha\beta}}, H_{\beta}|_{U_{\alpha\beta}}$ using $G_{\alpha\beta}\in \aut_{\mscr{O}_{U_{\alpha\beta}}}(F_{U_{\alpha\beta}/S}^*\mscr{E}'|_{U_{\alpha\beta}})$, defined as the truncated exponential
    \[
    G_{\alpha \beta}:= \sum_{i=0}^{p-1} \frac{(h_{\alpha\beta}(F_{U_{\alpha\beta}/S}^*\theta'))^i}{i!}.
    \]
    Here, by $h_{\alpha\beta}(F_{U_{\alpha\beta}/S}^*\theta')$, we mean the composition
    \[    F^*_{U_{\alpha\beta}/S}\mscr{E}'|_{U_{\alpha\beta}}\xrightarrow[]{F^*_{U_{\alpha\beta}/S}\theta'} 
F^*_{U_{\alpha\beta}/S}\mscr{E}'|_{U_{\alpha\beta}}\otimes F^*_{U_{\alpha\beta}/S}\Omega^1_{U_{\alpha\beta}^{(1)}/S} \xrightarrow[]{\id \otimes h_{\alpha\beta}}F^*_{U_{\alpha\beta}/S}\mscr{E}'|_{U_{\alpha\beta}}.
    \]
\end{construction}
\begin{proposition}\label{prop: ov=exp-twist}
    The recipe in \autoref{construction:exp-twist} gives a flat bundle on $X/S$, which is precisely  $C^{-1}(\mscr{E}', \theta')$ of 
\autoref{thm:ogus-vologodsky}. 

Moreover, on each $U_{\alpha}$, the isomorphism (of the underlying bundles) obtained via \autoref{construction:exp-twist}
\[C^{-1}(\mscr{E}', \theta')|_{U_{\alpha}}\simeq F^*_{U_{\alpha}/S}\mscr{E}'\]
is precisely $\eta_{\widetilde{F}_{\alpha}}$ in (2) of \autoref{thm:ogus-vologodsky}.
\end{proposition}
\begin{proof}
    We have  to check that the $G_{\alpha\beta}$'s satisfy the cocycle condition, and that the connections $\nabla_{\alpha}$ glue to give  a nilpotent flat connection: see \cite[\S 2.2]{lan-sheng-zuo} for these claims. For the claim that this construction agrees with $C^{-1}(\mscr{E}', \theta')$, see \cite[\S 3]{lan-sheng-zuo}.
\end{proof}
\begin{remark}\label{rem:w2-lifting-X}
The assumption that $S$ admits a flat lift $\widetilde{S}$ to $\mathbb{Z}/p^2\mathbb{Z}$, and that $X^{(1)}$ admits a flat lift over $\widetilde{S}$, is slightly too strong for some applications. It may be replaced with the weaker assumption that $X$ lifts to $W_2(S)$. This variant assumption has the following convenient property: if $T\to S$ is a morphism, then $X_T$ lifts over $W_2(T)$ if $X$ lifts over $W_2(S)$ (by functoriality of $W_2$). See \cite[Remark 3.3]{petrov-maryland} and \cite{terentiuk}. Unfortunately a complete proof of this observation has not yet appeared in the published literature.

We will use the following weak form of the above, which follows from our formulas in \autoref{construction:exp-twist}. Suppose $X^{(1)}$ lifts to $\widetilde S$ as in \autoref{thm:ogus-vologodsky}. Then for any $T\to S$, one may pull back the morphisms $\zeta_\alpha, h_{\alpha\beta}$ of \autoref{construction:exp-twist} to $T$, whence one obtains an Ogus--Vologodsky correspondence on $T$, functorial in $T/S$ (and depending on the lift of $X^{(1)}$).
\end{remark}
\begin{remark}
Note that the inverse Cartier transform preserves semistability, i.e.~it sends a semistable Higgs bundle of slope zero to a semistable flat bundle of slope zero.	
\end{remark}

\subsubsection{Canonical sections to the conjugate moduli stack}
Fix a map $T\to S$ and 
 let $(\mathscr{E}, \theta)$ be a Higgs bundle on $X^{(1)}_T/T$. We will construct, if the order of nilpotence of $\theta$ is not too large in terms of $p$, a canonical section to $\mathscr{M}_{\conj}(X_T/T)\to \mathbb{A}^1_T$ passing through $(\mathscr{E}, \theta)$ (viewed as a point of $\mathscr{M}_{\conj}(X_T/T)_0$).

\begin{definition}\label{defn:canonical-section}
	Let $(\mathscr{E}, \theta)$ be a Higgs bundle on $X^{(1)}_T/T$ of order of nilpotence at most $p-1$. Set $(\mathscr{E}_\lambda, \nabla_\lambda)$ on $X\times_S \mathbb{A}^1_T/\mathbb{A}^1_T$ to be $$(\mathscr{E}_\lambda, \nabla_\lambda):=C^{-1}(\mathscr{E}, \lambda \theta),$$ where $\lambda$ is the coordinate on $\mathbb{A}^1$. Note that $(\mathscr{E}_0, \nabla_0)$ has vanishing $p$-curvature, i.e. $\psi_p(\mathscr{E}_\lambda, \nabla_\lambda)$ is in the ideal generated by $\lambda$. Thus we may set $$\widetilde\theta:=\lambda^{-1}\psi_p(\mathscr{E}_\lambda, \nabla_\lambda).$$
	
	The triple $(\mathscr{E}_\lambda, \nabla_\lambda, \widetilde\theta)$ evidently gives rise to a section to the natural map $\lambda: \mathscr{M}_{\conj}(X_T/T)\to \mathbb{A}^1_T$. We refer to it as the \emph{canonical section} associated to $(\mathscr{E}, \theta)$.
\end{definition}
\begin{remark} 
Note that in general (i.e.~if $(\mathscr{E}, \theta)$ is not a fixed point of the natural $\mathbb{G}_m$-action on $\mathscr{M}_{\Dol}(X^{(1)}/S)$), the canonical section associated to $(\mathscr{E}, \theta)$ is not $\mathbb{G}_m$-equivariant.
\end{remark}

\begin{proposition}\label{prop:central-fiber-canonical-section}
	Let $(\mathscr{E},\theta)$ be as above and let $(\mathscr{E}_\lambda, \nabla_\lambda, \widetilde \theta)$ be the canonical section associated to it. There is a natural isomorphism $$(F_{X/S}^*(\mathscr{E}), \nabla_{\on{can}})\to (\mathscr{E}_0, \nabla_0)$$ sending $F_{X/S}^*\theta$ to $\widetilde\theta_0$. That is,  the canonical section associated to $(\mathscr{E}, \theta)$ passes through $(\mathscr{E},\theta)$, viewed as a point of $\mathscr{M}_{\conj}(X_T/T)_0$, under the identification of  $\mathscr{M}_{\conj}(X_T/T)_0$ with $\mathscr{M}_{\Dol}(X^{(1)}_T/T)$ of \eqref{eqn:non-abelian-f-zip}.
\end{proposition}
\begin{proof}
	We get an explicit construction of $C^{-1}(\mscr{E}, \lambda \theta)$ from \autoref{construction:exp-twist} and \autoref{rem:w2-lifting-X} and  we use freely the  setup and notation therein. In particular, we have an open cover $(U_{\alpha}\times \mb{A}^1_T)_{\alpha}$ of $X\times \mb{A}^1_T$, identifications of underlying bundles 
    \[
    \eta_{\widetilde{F}_{\alpha}}: C^{-1}(\mscr{E}, \lambda \theta)|_{U_{\alpha}\times \mb{A}^1_T} \simeq F^*_{U_{\alpha}\times \mb{A}^1_T/\mb{A}^1_T}\mscr{E},
    \]
    which are in turn glued along
    gluing maps 
    \[
    G_{\alpha\beta} \in \aut_{\mscr{O}_{U_{\alpha\beta}\times \mb{A}^1_T}}(F^*_{U_{\alpha\beta}\times \mb{A}^1_T/\mb{A}^1_T} \mscr{E}|_{(U_{\alpha\beta}\times \mb{A}^1_T)^{(1)}/\mb{A}^1_T}).
    \]
    Inspecting the formula  for $G_{\alpha\beta}$ in \autoref{construction:exp-twist}, we see  that $G_{\alpha\beta}|_{\lambda=0}=\id$, and so we get  an isomorphism $(F_{X/S}^*(\mathscr{E}), \nabla_{\on{can}})\to (\mathscr{E}_0, \nabla_0)$, which proves the first claim. 

    Moreover, again using the fact that the gluing map $G_{\alpha\beta}|_{\lambda=0}=\id$, it suffices to check the second claim on each $U_{\alpha}$.
    
    From   (2) of \autoref{thm:ogus-vologodsky}, we have that, under the identification $\eta_{\widetilde{F}_{\alpha}}$,
    \[
    \psi_p(\mscr{E}_{\lambda}, \nabla_{\lambda})|_{U_{\alpha}\times \mb{A}^1_S} = 
    F^*_{U_{\alpha}\times \mb{A}^1_S/\mb{A}^1_S}\lambda \theta,
    \]
    from which the claim on $U_{\alpha}$ follows immediately.
\end{proof}
\begin{remark}
The construction of the canonical sections above is a non-abelian analogue of Deligne--Illusie's splitting of the conjugate filtration on de Rham cohomology \cite{deligne-illusie}.	
\end{remark}
\begin{remark}
Our construction of the canonical section above relies on the Ogus--Vologodsky correspondence and thus only makes sense for bundles of order of nilpotence at most $p-1$. However, the formulas make sense on the $PD$-completion $0^{\widehat{\sharp}}$ of $0\in \mathbb{A}^1_S$. That is, if one considers the fiber product
$$\xymatrix{
\mathscr{M}_{\conj}(X/S)^{\widehat{\sharp}} \ar[r] \ar[d] & \mathscr{M}_{\conj}(X/S)\ar[d]\\
0^{\widehat{\sharp}} \ar[r] & \mathbb{A}^1_S
}$$
there exists a canonical section to $\mathscr{M}_{\conj}(X/S)^{\widehat{\sharp}}\to 0^{\widehat{\sharp}}$ passing through every point of the central fiber of $\mathscr{M}_{\conj}(X/S)^{\widehat{\sharp}}$.
\end{remark}
\begin{lemma}[Equisingularity]\label{lemma:equisingularity}
	Suppose $p>\dim(X/S)$. Let $(\mathscr{E}, \theta)$ be a Higgs bundle on $X^{(1)}/S$ of order of nilpotence at most $(p-1)/2$. Let $t: \mathbb{A}^1_S\to \mathscr{M}_{\conj}(X/S)$ be the canonical section to $\mathscr{M}_{\conj}(X/S)\to \mathbb{A}^1_S$ classifying $(\mathscr{E}_\lambda, \nabla_\lambda)$, defined as in \autoref{defn:canonical-section}. Then the pullback of the tangent complex of $\mathscr{M}_{\conj}(X/S)$ along $t$ has $i$-th cohomology sheaves which are $\lambda$-torsion-free for $i=0,-1$.
\end{lemma}
\begin{proof}
We wish to show that the cohomology sheaves of $t^*\mathscr{T}_{\mathscr{M}_{\conj}(X/S)/S\times \mathbb{A}^1}$ (in degrees $0$ and $-1$) are $\lambda$-torsion-free. 

Without loss of generality we may assume $S$ is the spectrum of a finite $k$-algebra, in which case it suffices to show that these cohomology sheaves, viewed as a sheaf on $\mathbb{A}^1_k$ (by pushforward along the finite map $\mathbb{A}^1_S\to \mathbb{A}^1_k$), are vector bundles. 

By \autoref{prop:tangent-spaces}, for $i=0, -1$,  the fiber of $\mathscr{H}^i(t^*\mathscr{T}_{\mathscr{M}_{\conj}(X/S)})$ at $\lambda=a\in \overline{k}$ is isomorphic to $H^{i+1}(\on{End}(\mathscr{E}_a, \nabla_a)_{\dR})$ for $a\neq 0$ and equal to $H^{i+1}(\on{End}(\mathscr{E}, \theta)_{\on{Higgs}})$ for $a=0$. It is enough to prove that the ranks of these vector spaces are independent of $a\in \mathbb{A}^1(\overline{k})$, as a coherent sheaf on $\mathbb{A}^1_k$ of constant rank is a vector bundle.

We first show that this holds for $a\neq 0$. Indeed in this case $(\mathscr{E}_a, \nabla_a)=C^{-1}(\mathscr{E}, a\theta)$ by definition, and hence $H^{i+1}(\on{End}(\mathscr{E}_a, \nabla_a)_{\dR})=H^{i+1}(\on{End}(\mathscr{E}, a\theta)_{\on{Higgs}})$ by \cite[Theorem 2.26]{ogus-vologodsky}. These latter vector spaces have dimension independent of $a$; indeed, multiplication by $a^j$ in degree $j$ induces an isomorphism between $\on{End}(\mathscr{E}, \theta)_{\on{Higgs}}$ and $\on{End}(\mathscr{E}, a\theta)_{\on{Higgs}}$. 

Finally, to show the rank at $a=0$ agrees with the rank at $a=1$, it suffices to compare $H^{i+1}(\on{End}(\mathscr{E}, \theta)_{\on{Higgs}})$ and $H^{i+1}(\on{End}(\mathscr{E}_1, \nabla_1)_{\dR})$. But again these two objects are related by the inverse Cartier transform, completing the proof.
\end{proof}
\subsection{Vanishing of $p$-curvature and liftings of tangent vectors}
Recall that, as in \autoref{section: semistable-deform}, we use  $\mscr{M}_{\dR}^{\on{ss}}$ to denote the semistable locus with respect to a fixed polarization, and similarly for the other stacks. 

\begin{theorem} \label{thm:p-curvature-lifting-vanishes} Let $k$ be a field of characteristic $p>0$, $S$ a smooth $k$-scheme, and $f: X\to S$ a smooth projective morphism. Suppose $p>\dim(X/S)$.

Suppose that for some $r\geq 2$ the $p$-curvature of the isomonodromy foliation on $\mathscr{M}_{\dR}^{\on{ss}}(X/S, r)$ vanishes identically. Then the lifting of tangent vectors $\Theta_{X/S}$ on $\mathscr{M}_{\Dol}^{\on{ss}}(X/S, r)$ vanishes on the locus of Higgs bundles of order of nilpotence at most $(p-1)/2$.

The same statement holds if we replace $\mscr{M}_{\dR}^{\on{ss}}(X/S, r)$ and $\mscr{M}_{\Dol}^{\on{ss}}(X/S, r)$ by the stable loci $\mscr{M}_{\dR}^{\on{s}}(X/S, r)$ and $\mscr{M}_{\Dol}^{\on{s}}(X/S, r)$ respectively.
\end{theorem}
\begin{proof}
	Without loss of generality $k=\overline{k}$. Let $R$ be an Artin $k$-algebra with residue field $k$ and $s$ an $R$-point of $S$. Consider an $R$-point of $\mscr{M}^{\on{ss}}_{\Dol}(X/S, r)$ corresponding to a Higgs bundle $(\mathscr{E}, \theta)$  on $X_{s}/R$ of order of nilpotence at most $(p-1)/2$, such that $(\mathscr{E},\theta)_k$ is semistable. We wish to show that the composition $$T_{R}\to  R^1f_{s*}T_{X_{s}/R}\overset{\theta}{\to} R^1f_{s*}\on{End}(\mathscr{E})_{\on{Higgs}}$$ vanishes.
	
	Let $s': \on{Spec}(R^{(1)})\to S$ be the pullback of $s$ along the absolute Frobenius map on $S$, so that we have a diagram
	$$\xymatrix{
	 X_{s'} \ar[r]^{F_{X/S, s'}} \ar[rd]^{f'} & (X^{(1)})_{s'} \ar[r]^G\ar[d] & X_s\ar[d]\\
	  & \on{Spec}(R^{(1)}) \ar[r] \ar[d]^{s'} & \on{Spec}(R) \ar[d]^s \\
	  & S \ar[r]^{F_{abs}} & S
	}$$
	where the two squares are Cartesian.
	Note that $R\to R^{(1)}$ is flat by the smoothness of $S$.
	
	Pulling back $(\mathscr{E}, \theta)$ to $(X^{(1)})_{s'}$ along $G$, we obtain an $R^{(1)}$ point of $\mathscr{M}_{\conj}(X_{s}/R)_0=\mathscr{M}_{\Dol}(X_{s}^{(1)}/R)$, where we use the isomorphism \eqref{eqn:non-abelian-f-zip}, or equivalently an $R^{(1)}$-point of $\mathscr{M}_{\conj}(X_{s'}/R^{(1)})_0.$

	Let $t: \mathbb{A}^1_{R^{(1)}}\to \mathscr{M}_{\conj}(X_{s'}/R^{(1)})$ be the canonical section to $\lambda: \mathscr{M}_{\conj}(X_{s'}/R^{(1)})\to \mathbb{A}^1_{R^{(1)}}$, classifying the family  $(\mathscr{E}_\lambda, \nabla_\lambda, \widetilde{\theta})$ passing through this point, as defined in \autoref{defn:canonical-section}; note that, by construction,  $t$ lies in the semistable locus as semistability is preserved by inverse Cartier.  We first study the map $$\Psi_{\conj}:F_{abs}^*T_{S\times \mb{A}^1/\mb{A}^1}|_{R^{(1)}} \rightarrow H^0(t^*\mathscr{T}_{\mathscr{M}_{\conj}(X_{s'}/R^{(1)})/R^{(1)}\times \mathbb{A}^1}),$$ For $\lambda\not= 0$, this map is zero by the assumption on the vanishing of $p$-curvature and  the $\mathbb{G}_m$-equivariance of $\mathscr{M}_{\conj}(X/S)\to \mathbb{A}^1_S$. By \autoref{lemma:equisingularity}, the target of this map is $\lambda$-torsion-free hence the map is identically zero, whence $\Psi_{\conj}/\lambda$ is also zero.

Now \autoref{lemma:katz-formula} tells us that the same is true for $\Theta_{X/S}|_{(\mathscr{E},\theta)}$, using the flatness of $R\to R^{(1)}$. As $(\mathscr{E},\theta)$ was arbitrary, this completes the proof in the semistable case. The stable case is identical.
\end{proof}
\begin{remark}
	Some version of \autoref{thm:p-curvature-lifting-vanishes}, under the assumption that the relative Frobenius of $X/S$ admits a lift over $\mathbb{Z}/p^2$, is sketched  in  \cite[Lemma 8.4 and Theorem 8.5]{menzies2019p}. While that assumption is too strong for our purposes the argument here was inspired by \cite{menzies2019p}.
\end{remark}
We now prove the main characteristic-zero consequence of the theory developed above.
\begin{theorem}\label{thm:p-curvature-vanishing-implies-theta-vanishing}
Let $A$ be a domain which is a finite-type $\mathbb{Z}$-algebra of characteristic zero, and let $K$ be its field of fractions. Let $f: X\to S$ a smooth projective morphism of smooth $A$-schemes. Suppose that, for some $r\geq 2$,  the $p$-curvature of the isomonodromy foliation on $\mathscr{M}_{\dR}^{\on{ss}}(X/S,r)\bmod p$ vanishes identically for infinitely many $p$. Then the lifting of tangent vectors $\Theta_{X_K/S_K}$ on $\mathscr{M}_{\Dol}^{\on{ss}}(X_K/S_K, r)$ vanishes identically. The same statement holds if we replace the semistable loci by the stable ones.
\end{theorem}
\begin{proof}
	Fix a point $s$ of $S_K$, with residue field $K'$ and let $A'\supset A$ be a domain which is a finite-type $\mathbb{Z}$-algebra with fraction field $K'$, so that $s$ extends to an $A'$-point of $S$. We wish to show that the map $$\Theta_{X_K/S_K, s}: T_{S,s}\otimes \mathscr{O}_{\mathscr{M}_{\Dol}^{\on{ss}}(X_s, r)}\to {T}_{\mathscr{M}_{\Dol}^{\on{ss}}(X_s, r)}$$ vanishes. We will first show that this map vanishes on the formal neighborhood of the nilpotent cone $h^{-1}(0)$, where $$h: \mathscr{M}^{\on{ss}}_{\Dol}(X_s, r)\to \bigoplus_{i=2}^r H^0(X_s, \on{Sym}^i\Omega^1_{X_s})$$ is the Hitchin map. We will conclude by arguing that this formal neighborhood intersects the closure of every associated point of ${T}_{\mathscr{M}_{\Dol}^{\on{ss}}(X_s, r)}$.
	
	\emph{Step 1. Vanishing on the formal neighborhood of the nilpotent cone.} For each integer $n\geq 0$, let $$\mathscr{M}_{\Dol}^{\on{ss}}(X_s, r)_n\subset \mathscr{M}_{\Dol}^{\on{ss}}(X_s, r)$$ be the closed substack consisting of those Higgs bundles $(\mscr E, \theta)$ such that   $\theta^n=0$. For $n=r$, this locus contains $h^{-1}(0)$, and $$\bigcup_n\mathscr{M}_{\Dol}^{\on{ss}}(X_s, r)_n$$ is the formal substack of $\mathscr{M}_{\Dol}^{\on{ss}}(X_s, r)$ obtained by completing it at $h^{-1}(0)$, i.e.~it is the formal neighborhood of the nilpotent cone. To show that $\Theta_{X_K/S_K, s}$ vanishes on this formal neighborhood, it suffices to show that it vanishes on each $\mathscr{M}_{\Dol}^{\on{ss}}(X_s, r)_n$. But we know that this is true mod $p$ for infinitely many $p$ by \autoref{thm:p-curvature-lifting-vanishes}, hence it is true in characteristic zero.
	
	\emph{Step 2. Associated points of ${T}_{\mathscr{M}_{\Dol}^{\on{ss}}(X_s, r)}$.} We now observe that, to conclude, it suffices that any associated point of ${T}_{\mathscr{M}_{\Dol}^{\on{ss}}(X_s, r)}$ (see \cite[\href{https://stacks.math.columbia.edu/tag/02OJ}{Tag 02OJ}]{stacks-project} for the definition of associated points) has closure which intersects $h^{-1}(0)$. But this follows from the fact that the closure of  an associated point is stable under the natural $\mathbb{G}_m$-action on $\mathscr{M}_{\Dol}^{\on{ss}}(X_s, r)$, as well as \autoref{thm: simpson-langton}.

    The proof of the statement in the case of  the stable loci is identical. 
\end{proof}

\begin{remark}\label{rem:theta-is-KS}
Suppose we are in a situation as in the conclusion of \autoref{thm:p-curvature-vanishing-implies-theta-vanishing}, i.e. $f: X\to S$ is a smooth projective morphism of smooth schemes over a field of characteristic zero, and $\Theta_{X/S}$ vanishes identically. In this case we say that the isomonodromy foliation \emph{preserves the Hodge filtration}, for reasons we now explain.

Recall that $\mathscr{M}_{\Hod}(X/S)$ is equipped with a $\lambda$-lifting of tangent vectors, as in \autoref{defn:Hodge-lifting}, whose fiber at $1$ is precisely the isomonodromy foliation and whose fiber at zero is $\Theta_{X/S}$. Thus the vanishing of $\Theta_{X/S}$ implies that, as in \autoref{lemma:katz-formula}, $\frac{\Xi_{X/S}}{\lambda}$ makes sense where $T_{\mathscr{M}_{\Hod}(X/S)/S\times \mathbb{A}^1}$ is $\lambda$-torsion free, and is a $1$-lifting of tangent vectors.

In other words, in this case the isomonodromy foliation on $\mathscr{M}_{\dR}(X/S)$ extends over $\mathscr{M}_{\Dol}(X/S)$, i.e.~it extends to the associated graded of the Hodge filtration. This is the non-abelian analogue of the statement that if the Kodaira-Spencer map of a geometric variation of Hodge structure vanishes, then the Gauss-Manin connection preserves the Hodge filtration.

We make this intuition precise in a different (though related) way in \autoref{thm:triholomorphic}.
\end{remark}

\section{Relative compactness of leaves---a non-abelian Hodge index theorem}\label{sec:NAHI}
In this section we will recall some complex non-abelian Hodge theory and its interaction with isomonodromy. We will also prove our non-abelian analogues of the Hodge index theorem, \autoref{thm:compactness-thm} and \autoref{thm:triholomorphic}. Much of the work in this section will take place on the coarse spaces of the moduli stacks studied in previous sections, or on their reduced smooth loci.

\subsection{Background on complex non-abelian Hodge theory}
Let $(X, \omega)$ be a compact K\"ahler manifold of dimension $n$. \begin{definition}
\begin{enumerate}
    \item We write $M_{\Dol}(X, r)$ for  the coarse moduli space of semistable (with respect to the K\"ahler class $\omega$), rank $r$, Higgs bundles on $X$, with vanishing rational Chern classes and trivial determinant. This has the structure of a complex analytic space, and the structure of a scheme in the case when $X$ is a projective variety: see \cite[Proposition 1.3.2]{fujiki2006hyperkahler} and \cite[p.16]{simpson-moduli-representations-2}. Let $M_{\Dol, 0}(X, r)$ be the nonsingular points of the reduced subspace of $M_{\Dol}(X, r)$. 
    
    \item We write $M_{B}(X, r)$ for the GIT quotient $\hom(\pi_1(X), \sl_r)\sslash \sl_r$. Let $M_{B, 0}(X, r)$ be the nonsingular points of the reduced subspace of $M_{B}(X, r)$. 
\end{enumerate}
     
\end{definition}
The fundamental theorem of non-abelian Hodge theory, due to Corlette \cite[Theorem 3.4]{corlette1988flat} and Simpson \cite[p. 870]{simpson1988constructing}, is that, starting from $\rho\in M_{B}(X, r)$, one can construct an equivariant harmonic metric (whose definition we will recall in some form below) on the associated semisimple flat bundle, which in turn gives a Higgs bundle, and vice versa. We can state it in terms of moduli spaces as follows:
\begin{theorem}
    There is a canonical homeomorphism of the underlying topological spaces 
    \[
    M_{B}(X, r)^{\top} \simeq M_{\Dol}(X, r)^{\top};
    \]
    moreover, when restricted to $M_{B, 0}(X, r)$, this gives a  real-analytic isomorphism $M_{B, 0}(X, r)\simeq M_{\Dol, 0}(X, r)$.
\end{theorem}
\begin{proof}
    See \cite[Lemma 9.14]{simpson-moduli-representations-2}.
\end{proof}
It was observed by Goldman \cite{goldman1984symplectic} that $M_{B}(X, r)$ has a  natural holomorphic symplectic structure: 
\begin{proposition}\label{prop: holomorphic-symplectic-betti}
%The choice of a Killing, i.e. $\mf{gl}_r$-invariant, non-degenerate bilinear, form $\mf{gl}_r\otimes \mf{gl}_r\rightarrow \mb{C}$ gives rise to an 
There is an algebraic holomorphic symplectic form $\Omega$ on $M_{B}(X, r)$, which may be defined on $\mb{C}$-points as follows. For $[\rho]\in M_{B}(X, r)(\mb{C})$ corresponding to a rank $r$ local system, let $\ad(\rho):= (\rho \otimes \rho^{\vee})^{\on{tr}=0}$ denote the adjoint local system. 

 Recall that the tangent space $T_{[\rho]}M_{B}(X, r)$ is canonically identified with $H^1(X, \ad(\rho))$; moreover, the trace pairing gives a map of local systems $\ad(\rho)\otimes \ad(\rho)\rightarrow \mb{C}$. Then the holomorphic symplectic  form $\Omega$, restricted to $T_{[\rho]}M_B(X, r)$,  is given by the composition 
\[
H^1(X, \ad(\rho))\otimes H^1(X, \ad(\rho)) \rightarrow H^2(X, \ad(\rho)\otimes \ad(\rho))\xrightarrow[]{\tr} H^2(X, \mb{C})\xrightarrow{\wedge [\omega^{n-1}]} H^{2n}(X, \mb{C})\simeq \mb{C}.
\]

\end{proposition}
\subsection{Energy of harmonic maps}
\subsubsection{Recollections on harmonic maps}
We first review some standard notions in the theory of harmonic maps.
\begin{definition}
	Let $(M, h), (N, g)$ be Riemannian manifolds. The \emph{energy density} of a map $f: M\to N$ is the function $e(f): M\to \mathbb{R}$ given by the formula $$\frac{1}{2}\langle df, df\rangle_{T^*M\otimes f^*TN}.$$ We define the \emph{energy} of $f$ to be $$E(f):= \int_M e(f) d\on{vol}_M$$ if $M$ is compact. We say $f$ is \emph{harmonic} if it is a critical point for $E$. That is, if $\vec v\in C^\infty(M, f^*TN)$, we set $$f_{\vec v, t}(x)=\exp_{f(x)}(t\vec v(x))$$ and $$\partial_{\vec v} E(f):=\frac{\partial E(f_{\vec v, t})}{\partial t}\Big\vert_{t=0},$$ and say $f$ is harmonic if   $\partial_{\vec v} E(f)$ vanishes for all $\vec v$.
\end{definition}
\subsubsection{Equivariant, or twisted, harmonic maps}
We will in fact need an equivariant, sometimes also called twisted, version of the discussion above. The notion of an equivariant harmonic map was introduced by Donaldson \cite{donaldson1987twisted}, and the main existence result is due to Corlette \cite{corlette1988flat}. Let $M, N$ be Riemannian manifolds with $M$ compact, connected,  and $\rho: \pi_1(M)\rightarrow \mathrm{Isom}(N)$ a homomorphism. We form the $N$-bundle over $M$ 
\[
\mscr{N}:= \widetilde{M}\times N/\pi_1(M)
\]
where $\widetilde{M}$ denotes the universal cover of $M$, and the $\pi_1(M)$-action on $N$ is through $\rho$. Then $\mscr{N}$ has the structure of a Riemannian manifold, with a natural map $\pi: \mscr{N}\rightarrow M$ which is locally (on $M$) a product; the  fiber of $\pi$ at each $m\in M$, $\mscr{N}_m$, is isometric to $N$. For $s: M\rightarrow \mscr{N}$ a section of $\pi$, we define at $m\in M$, using the local product structure of $\mscr{N}$,  the vertical part of its derivative by the composition
\[
(Ds)_m: T_m M \rightarrow  T_{s(m)}\mscr{N} \rightarrow T_{s(m)} \mscr{N}_m.
\]
The energy density $e(s): M\rightarrow \mb{R}$ is then given by the formula
\[
\frac{1}{2} \langle Ds, Ds \rangle_{T^*M\otimes s^*T\mscr{N}_m}. 
\]
Finally,  the energy of $s$ is  
\[
E(s)=\int_M e(s) d\on{vol}_M.
\]
\begin{definition}
    We say that the section $s: M \rightarrow \mscr{N}$ is harmonic if it is a critical point of  $E(s)$. A section $s$ gives rise to a map $f_s: \widetilde{M}\rightarrow N$, and in the case where $s$ is a harmonic section, we will refer to $f_s$ as a ($\rho$-)equivariant harmonic map.
\end{definition}

\begin{theorem}\label{thm: corlette-harmonic-map}
    Let $M$ be a compact connected Riemannian manifold, and $\rho: \pi_1(M)\rightarrow \sl_r(\mb{C})$ a semisimple representation. Then there is a $\rho$-equivariant harmonic map $\widetilde{M}\rightarrow \sl_r(\mb{C})/\su_r$. 
    
    %It is in fact a minimum of the energy.
    Moreover, such a harmonic map is unique if $\rho$ is irreducible. More generally, any two $\rho$-equivariant harmonic maps are in the same orbit of $Z_{\rho}\leq \mathrm{SL}_r(\mb{C})$, the centralizer of $\rho$, where the latter acts on the left on $\sl_r(\mb{C})/\su_r$. 
\end{theorem}

\begin{proof}
    For the first claim, see \cite[Theorem 3.3]{corlette1988flat}, as well as the text immediately following it, which explains the relationship between that theorem and  harmonic maps. 

    For the claim that the harmonic map is an  energy minimizer, see \cite[\S2.6]{korevaar1993sobolev}.
\end{proof}

The following is a consequence of work of Uhlenbeck \cite{uhlenbeck1982connections}, and is well known to experts.
\begin{theorem}\label{thm:energy-proper}
    Suppose $X$ is an orientable (topological) surface equipped with a Riemannian metric. The energy gives a  proper function 
\[
E: M_{B}(X, r)\rightarrow \mb{R},
\]
real-analytic on $M_{B,0}(X, r)$, sending a semisimple representation $\rho$ to the energy of the associated $\rho$-equivariant harmonic map to $\sl_r(\mb{C})/\su_r$.
\end{theorem}
\begin{proof}
    See for example \cite[Theorem 4.4]{bradlow2012morse} for a sketch of the argument.
\end{proof}
\subsubsection{Gardiner's formula}
\begin{theorem}\label{thm:gardiner-formula}
	Let $\Sigma$ be a compact topological surface of genus $g\geq 2$, with universal cover $\mathbb{H}$, and $\mscr{T}_g$ the Teichm\"uller space of marked complex structures on $\Sigma$.  Let $$\rho: \pi_1(\Sigma)\to SL_r(\mathbb{C})$$ be a semisimple representation. For each complex structure $[X]\in \mathscr{T}_g$ on $\Sigma$, let $$u_X: \mathbb{H}\to SL_r(\mathbb{C})/SU(r)$$ be a $\rho$-equivariant harmonic map (whose existence is guaranteed by \autoref{thm: corlette-harmonic-map}),  $(\mathscr{E}_X, \theta_X)$ be the corresponding Higgs bundle on $X$, and $\Phi_X:=h_2(\mathscr{E}_X, \theta_X)\in H^0(X, \omega_X^{\otimes 2})$ the quadratic component of the Hitchin image of $(\mathscr{E}_X, \theta_X)$; note that $\Phi_X$ is simply $-\frac{\tr (\theta^2)}{2}$ (see \autoref{section: hitchin-map} for the definition of the Hitchin map). Set $$E_\rho: \mathscr{T}_g\to \mathbb{R}$$ $$X\mapsto E(u_X)$$ to be the function associating to a complex structure on $\Sigma$ the energy of the map $u$. Fix $v\in H^1(X, T_X)$ a first-order deformation of $X$. Then $$\partial_v E_\rho=-8r \Re(\tr(v\cup \Phi_X)).$$ \end{theorem}
\begin{proof}
    We explain how this follows from the version of Gardiner's formula stated in \cite[Theorem 3.11]{daskalopoulos2007harmonic}; note that this theorem discusses energy minimizers, rather than harmonic maps, but there is no difference in our case since the target $\sl_r(\mb{C})/\su(r)$ is non-positively curved (NPC). Now,  \cite[Theorem 3.11]{daskalopoulos2007harmonic},  states that 
    \[
    \partial_vE_{\rho} = 2\Re(\tr(v \cup \mathrm{Hopf}(u_X))).
    \]
    Here, $\mathrm{Hopf}(u_X)$ is the \emph{Hopf differential} of $u_X$ (descended to $X$): see \cite[\S 2.2.3]{daskalopoulos2007harmonic} for the definition. Note that the Hopf differential, in this case,   is identified with $2r\tr(\theta^2)$--see e.g. \cite[\S 5.4]{li2019introduction}; the  latter is, in turn, $-4rh_2(\mscr{E}_X, \theta_X)$.
\end{proof}
\subsection{The compactness theorem}
We now prepare to prove \autoref{thm:intro-hodge-1} (see \autoref{thm:compactness-thm} below). The proof will proceed by reduction to families of curves; the main input is the following lemma.
\begin{lemma}\label{lemma:energy-constant}
Let $f:X\to S$ be a relative curve of genus $g$, with $S$ smooth. Fix $s\in S(\mathbb{C})$ and set $\Sigma=X_s$. Let $\rho: \pi_1(\Sigma)\to \sl_r(\mathbb{C})$ be a semisimple representation. 

Suppose that $\Theta_{X/S}$ vanishes on the leaf of the isomonodromy foliation  corresponding to $\rho$ (viewed as a real-analytic subset of $\mathscr{M}_{\Dol}(X/S)$). Let $\mathscr{U}\subset \mathscr{T}_g$ be the preimage of the image of the map $S\to \mathscr{M}_g$ classifying $f$.

Then $E_\rho|_\mathscr{U}$ is locally constant.
\end{lemma}
\begin{proof}
	Let $v\in H^1(\Sigma, T_\Sigma)$ be a vector
	in the image of the Kodaira--Spencer map $T_{S,s}\to H^1(\Sigma, T_{\Sigma})$.
	We will show that $\partial_vE_\rho=0$.

	Let $(\mscr{E}_\Sigma,\theta_\Sigma)$ be the Higgs bundle on $\Sigma$ corresponding to $\rho$
	(by non-abelian Hodge theory), and set
	\[
	\Phi_\Sigma:=h_2(\mscr{E}_\Sigma,\theta_\Sigma)\in H^0(\Sigma,\omega_\Sigma^{\otimes 2}).
	\]
	By Gardiner's formula, it is enough to show that $\tr(v\cup \Phi_\Sigma)=0$.

	Let $v^\vee:H^0(\Sigma,\omega_\Sigma^{\otimes 2})\to \mb{C}$ be the linear functional corresponding
	to $v$ by Serre duality, so that $v^\vee(\Phi_\Sigma)=\tr(v\cup \Phi_\Sigma)$.
	Consider the holomorphic function $f:=h_2^*v^\vee$ on $\mathscr{M}_{\Dol}(\Sigma,r)$.
	By Chen's formula \eqref{eqn: chen-formula} and the naturality of the lifting $\Theta_{X/S}$ under pullback
	along the classifying map $S\to \mathscr{M}_g$, the hypothesis that $\Theta_{X/S}$ vanishes on the isomonodromy
	leaf corresponding to $\rho$ implies that the Hamiltonian vector field $H_f$ vanishes at the point
	$(\mscr{E}_\Sigma,\theta_\Sigma)$.

	Let $\Omega$ be the holomorphic symplectic form of \autoref{prop: holo-sympl-higgs}. On the smooth locus,
	$H_f$ is characterized by $\iota_{H_f}\Omega=df$, hence $H_f=0$ implies $df=0$ at $(\mscr{E}_\Sigma,\theta_\Sigma)$.
	Let $\xi$ be the vector field generated by the $\mb{C}^{\times}$--action scaling the Higgs field,
	$t\cdot (\mscr{E},\theta)=(\mscr{E},t\theta)$. Since $h_2$ is quadratic in $\theta$, we have
	$f(t\cdot(\mscr{E},\theta))=t^2f(\mscr{E},\theta)$, and therefore $df(\xi)=\xi(f)=2f$.
	Evaluating at $(\mscr{E}_\Sigma,\theta_\Sigma)$ and using $df=0$ gives $f(\mscr{E}_\Sigma,\theta_\Sigma)=0$, i.e.
	\[
	\tr(v\cup \Phi_\Sigma)=v^\vee(\Phi_\Sigma)=f(\mscr{E}_\Sigma,\theta_\Sigma)=0.
	\]
	Thus $\partial_vE_\rho=-8r\,\Re(\tr(v\cup \Phi_\Sigma))=0$, as required. Since $s$ and $v$ were arbitrary,
	$E_\rho|_\mathscr{U}$ is locally constant.
\end{proof}
\begin{theorem}\label{thm:compactness-thm}
Let $f: X\to S$ be a smooth projective morphism of smooth complex varieties. Fix $s\in S(\mathbb{C})$ and let $\mathbb{V}$ be a rank $r$ local system on $X_s$, with corresponding monodromy representation $[\rho]$. 

Suppose the lifting of tangent vectors $$\Theta_{X/S}: \pi_{\Dol}^*T_S\to {T}_{\mathscr{M}_{\Dol}(X/S)/S}$$ vanishes identically on the leaf of the isomonodromy foliation corresponding to $[\rho]$. Then the orbit $\pi_1(S(\mathbb{C})^{\an},s)\cdot[\rho]\subset M_B(X_s, r)^{\an}$ has compact closure (in the analytic topology).
\end{theorem}

\begin{proof}[Proof of \autoref{thm:compactness-thm}]
	Without loss of generality we may assume $f$ has relative dimension $1$, by replacing $X$ with a family $X'$ of sufficiently ample smooth complete intersection curves in $X$ of genus $g\geq 2$, at the cost of replacing $S$ with a dense open $S'\subset S$ containing $s$. Indeed, by the Lefschetz hyperplane theorem, we have that for each $s\in S'$, $X'_s\to X_s$ induces a surjection on $\pi_1$. Moreover $\pi_1(S')\to\pi_1(S)$ is surjective. Thus $\pi_1(S', s)\cdot [\rho|_{X'_s}]$ has compact closure in $M_B(X'_s, r)(\mathbb{C})^{\an}$ iff the same is true for $\pi_1(S(\mathbb{C})^{\an},s)\cdot[\rho]\subset M_B(X_s, r)(\mathbb{C})^{\an}$. Thus we may and do replace $X/S$ with $X'/S'$, which we rename $X/S$.
	
	Now this orbit is precisely the intersection of the leaf of the isomonodromy foliation in $M_B(X/S, r)$ with $M_B(X_s)$, whence it suffices to show that the leaf of the isomonodromy foliation is relatively compact over $S(\mathbb{C})^{\an}$. But by \autoref{lemma:energy-constant}, the energy of $\rho$ is constant along this leaf, whence the leaf is contained in a level set of the energy function on $M_B(X/S)$. This latter is relatively compact over $S^{\an}$ by \autoref{thm:energy-proper}.
\end{proof}
The following is immediate:
\begin{corollary}
Let $f: X\to S$ be a smooth projective morphism of smooth complex varieties. Fix $s\in S$.
Suppose the lifting of tangent vectors $$\Theta_{X/S}: \pi_{\Dol}^*T_S\to {T}_{\mathscr{M}_{\Dol}(X/S)/S}$$ vanishes identically.  Then every orbit of $\pi_1(S(\mathbb{C})^{\an},s) \lefttorightarrow M_B(X_s)$ has compact closure.
\end{corollary}

\subsection{The hyperk\"ahler metric}

Recall that $M_{\Dol, 0}(X, r)$ denotes the nonsingular points of the reduced subspace of $M_{\Dol}(X, r)$, and similarly $M_{B, 0}(X, r)$ the nonsingular points of $M_B(X, r)$. 
\begin{definition}
 A hyperk\"ahler manifold is a tuple $(M, g, I, J, K)$ consisting of a Riemannian manifold $(M, g)$ and almost complex structures $I, J, K$ such that 
    \begin{enumerate}
        \item $I^2=J^2=K^2=IJK=-1$, and 
        \item each of $(M, g, I)$, $(M, g, J)$, $(M, g, K)$ is a K\"ahler manifold.
    \end{enumerate}
\end{definition}
 For $q=(a,b,c)\in \mb{R}^3$ satisfying $a^2+b^2+c^2=1$, $I_q=aI+bJ+cK$ is also an almost complex structure on $M$; in other words, a hyperk\"ahler manifold has a sphere's worth of almost complex structures, each of which is in fact a K\"ahler structure. We write $\omega_q$ for the corresponding K\"ahler forms, which are defined by the formula 
 \[
 \omega_q(v, w)=g(I_qv, w).
 \]
 \begin{definition}
     For a hyperk\"ahler manifold $(M, g, I, J, K)$, we denote by  $\omega_I, \omega_J, \omega_K$  the K\"ahler forms corresponding to $I, J, K$, respectively. We also define the following 2-forms  
     \begin{align*}          \Omega_I&=\omega_J+i\omega_K,\\
    \Omega_J&=\omega_K+i\omega_I,\\
    \Omega_K&=\omega_I+i\omega_J,\\         \end{align*}     
    and it is straightforward to check that they are holomorphic in the complex structures $I, J, K$, respectively, and are moreover symplectic, i.e. non-degenerate and closed.
 \end{definition}

\begin{theorem}\label{thm: mbetti-hk-metric}
    There is a hyperk\"ahler manifold $(M, g, I, J, K)$ such that we have  identifications of complex manifolds
    \begin{enumerate}
        \item $(M, I)= M_{\Dol, 0}(X, r)$, 
        \item $(M, J) = M_{B, 0}(X, r)$.
        %\item $(M, I_q)\simeq M_{B, 0}(X, r)$ for all $q\neq (\pm 1, 0, 0)$.
    \item Via the identification in (1), we obtain a $\mb{C}^{\times}$-action on $M$. For $t\in \mb{C}^{\times}$, acting by $t$ induces a biholomorphic map $(M, I_q)\simeq (M, I_{t(q)})$. Moreover, for $t\in S^1\subset \mb{C}^{\times}$, this is a K\"ahler isometry.
    \item The holomorphic symplectic form $\Omega_J$ is the natural one on $M_J=M_{B, 0}(X, r)$, as given by \autoref{prop: holomorphic-symplectic-betti}. 
\end{enumerate}   
    \begin{proof}
    	This is \cite[(1.5.2)]{fujiki2006hyperkahler}.
    \end{proof}
  
\end{theorem}
We can package the above data via the twistor construction as follows: there is a unique complex structure $\mscr{J}$ on $M\times \mb{P}^1_{\mb{C}}$ such that
\begin{itemize}
    \item the restriction of $\mscr{J}$ to $M\times q$  is precisely $(M, I_q)$, and 
    \item $\mscr{J}|_{m\times \mb{P}^1_{\mc{C}}}$ is the usual complex structure, for each $m\in M$. 
\end{itemize}
We denote the resulting complex manifold by $(Z, \mscr{J})$, which, by construction,  is equipped with a holomorphic map $f: Z\rightarrow\mb{P}^1_{\mb{C}}$. There is then a holomorphic $\mb{C}^{\times}$-action on $Z$ which covers the usual action on $\mb{P}^1_{\mc{C}}$. 
\begin{lemma}
    The energy functional $E$ on $M$ is a K\"ahler potential for $\omega_J$; that is,
    \[
    \partial_J \overline{\partial}_J E = i \omega_J.
    \]
    Here $d = \partial_J + \overline{\partial}_J$ is the decomposition of $d$ into bi-degree $(1,0)$ and $(0,1)$ components, in the complex structure $J$.
\end{lemma}
 \begin{proof}
     The statement for $r=2$ and $X$ a closed Riemann surface of genus $\geq 2$  is due to Hitchin \cite[Equation (9.10)]{hitchin1987self}; for the general case, see \cite[Theorem 6.6]{spinaci2014deformations}.
 \end{proof}
 The following is now immediate from \autoref{lemma:energy-constant}:
\begin{theorem}\label{thm:kahler-preserved}
	Let $f: X\to S$ be a smooth projective morphism, and suppose the lifting of tangent vectors $\Theta_{X/S}$ vanishes identically. Then the K\"ahler form $\omega_J$ is preserved under isomonodromic deformation.
\end{theorem}
\begin{proof}
As in the proof of \autoref{thm:compactness-thm} we may locally on $S$, choose a family $X'\to X$  of sufficiently ample complete intersection curves of genus $g\geq 2$ such that for each $s\in S$, $X'_s\to X_s$ induces a surjection on $\pi_1$, so that the induced map $M_{\dagger}(X_s)\hookrightarrow M_{\dagger}(X'_s)$ is a closed embedding for $\dagger=B, \Dol,$ etc.; by assumption $\Theta_{X/S}$ vanishes on the image of this embedding (on the Dolbeault moduli space). Moreover the image of this embedding is evidently preserved by the isomonodromy foliation.

By \autoref{lemma:energy-constant}, the energy functional is thus preserved under isomonodromic deformation on this locus, whence the same is true for $\omega_J$, for which it is a K\"ahler potential.
\end{proof}

 \subsection{Triholomorphic action}
 \begin{theorem}\label{thm:triholomorphic}
     Suppose $X\rightarrow S$ is a smooth projective morphism of $\mb{C}$-varieties, with connected fibers. Suppose that the lifting of tangent vectors $\Theta_{X/S}$ on $M_{\Dol}(X/S, r)$ vanishes. Then, for any basepoint $s\in S$,  the action of $\pi_1(S(\mathbb{C})^{\an}, s)$ on $M_{B, 0}(X_s)$ is triholomorphic. In particular, the $\pi_1(S(\mathbb{C})^{\an}, s)$-action preserves the Riemannian metric $g$ from \autoref{thm: mbetti-hk-metric}. 
 \end{theorem}

\begin{proof}
    We have the relative Betti moduli space $M_{B, 0}(X/S, r)\rightarrow S$. Suppose $\epsilon>0$ and let  $\gamma:[-\epsilon,\epsilon]\rightarrow S$ be a path, and consider the fiber product
    \[
    \begin{tikzcd}
  \gamma^*M_{B, 0}(X/S, r) \arrow[r, ""] \arrow[d, "G"]
    & M_{B, 0}(X/S, r) \arrow[d, "" ] \\
  (-\epsilon, \
  \epsilon) \arrow[r, "\gamma"].
&  S \end{tikzcd}
\]
Write $s_0=\gamma(0)$. By isomonodromy, we identify each fiber of $G$, as real manifolds,  with $M=M_{B, 0}(X_{s_0}, r)$. We therefore obtain a family of hyperk\"ahler metrics $(g_{t}, I_t, J_t, K_t)$ on $M$, indexed by $t\in (-\epsilon, \epsilon)$. We must show that $(g_t, I_t, J_t, K_t)$ is a constant family. Recall that the $\pi_1(S(\mathbb{C})^{\an}, s)$ action on $M_{B, 0}(X,r)$  is holomorphic symplectic, and therefore both $J_t$ and $\Omega_{J_t}=\omega_{K_t}+i\omega_{I_t}$ are constant. 

By \autoref{thm:kahler-preserved}, the vanishing of the lifting of tangent vectors implies that $\omega_{J_t}$ is constant as well. We deduce that $g_t$ is constant, since 
\[g_t(v, w) = -\omega_{J_t}(J_t v, w)\]
is determined by $J_t$ and $\omega_{J_t}$.  %\daniel{What is $I_{J_t}$?}

On the other hand, $g_t$ is also determined by $I_t, \omega_{I_t}$ (resp. $K_t, \omega_{K_t}$), and hence $I_t$ (resp. $K_t$) is constant as well, as desired.
\end{proof}
 
\section{Finiteness of non-abelian monodromy}\label{sec:finiteness-na-monodromy}
%\section{The Ekedahl--Shepherd-Barron--Taylor conjecture}
In this section we give the main applications of the theory developed in the previous sections. We first prove a very strong form of our results for relative curves, namely that if the $p$-curvature of the isomonodromy foliation vanishes mod $p$ for infinitely many $p$, then such a family is in fact isotrivial. We then prove \autoref{thm:main-theorem-integral-points} and \autoref{thm:main-theorem-esbt}. Finally, we briefly discuss the relationship between these results and those of \cite{p-painleve}.
\subsection{Relative curves}
We begin with a case in which we can prove a very strong form of the Ekedahl--Shepherd-Barron--Taylor conjecture, namely where $f: X\to S$ is a relative curve. In this case the vanishing of $p$-curvatures of the isomonodromy foliation in fact implies that the family is isotrivial. Recall that for each stack $\mscr{M}_{\dagger}$, we denote by $\mscr{M}^{\on{s}}_{\dagger}$  the substack of stable objects.
\begin{theorem}\label{thm:relative-curves}
Let $R\subset \mathbb{C}$ be a finitely-generated $\mathbb{Z}$-algebra, $S$ a smooth $R$-scheme, and $f: X\to S$ a smooth projective morphism of relative dimension $1$ with connected fibers of genus $g\geq 2$. Suppose that for some $r\geq 2$ the $p$-curvature of the isomonodromy foliation on $\mathscr{M}_{\dR}^{\on{s}}(X/S, r)$ vanishes mod $\mathfrak{p}$ for a Zariski-dense set of points $\mathfrak{p}\in \on{Spec}(R)$. Then $X_{\mathbb{C}}/S_{\mathbb{C}}$ is isotrivial.
\end{theorem}
\begin{proof}
	By \autoref{thm:p-curvature-vanishing-implies-theta-vanishing}, we know that $\Theta_{X/S}$ is identically zero on $\mscr{M}_{\Dol}^{\on{s}}(X_K/S_K, r)$, where $K$ denotes the field of fractions of $R$. Fix $s\in S(\mathbb{C})$. By Chen's formula  (\autoref{thm:chen-formula}) and the surjectivity of the Hitchin map (even on the stable locus) for curves, this implies that for each $v\in T_{S,s}$ with associated Kodaira-Spencer class $\kappa(v)\in H^1(X_s, T_{X_s})$,  the functional $\kappa(v)^{\vee}$ of \autoref{thm:chen-formula} is identically zero. But this functional is precisely given by the (perfect) Serre duality pairing $$H^1(X_s, T_{X_s})\otimes H^0(X_s, \omega_{X_s}^{\otimes 2})\to \mathbb{C},$$ whence $\kappa(v)=0$. As $s, v$ were arbitrary, this implies that the Kodaira-Spencer map for $f: X\to S$ is trivial, and hence that $f$ is isotrivial.
\end{proof}
\begin{remark}
    Note that the semistable analogue of \autoref{thm:relative-curves}, that is assuming the vanishing of $p$-curvature on $\mscr{M}_{\dR}^{\on{ss}}$ instead, which of course follows from  \autoref{thm:relative-curves}, also follows from Katz's theorem \autoref{thm:katz-p-curvature-Gauss-Manin} and the Torelli theorem for curves directly. We can see this by considering the extensions given, tautologically, by $R^1f_*\Omega^{\bullet}_{X/S, \dR}$. Indeed, we can map the total space, denoted as $\mscr{H}(X/S)$, of this vector bundle to $\mscr{M}_{\dR}(X/S, 2)$ by viewing a class in $H^1_{\dR}$ as a rank two flat connection $(\mscr{E}, \nabla)$ which is an extension of the form 
    \begin{equation}\label{eqn: extn-h1-class}
        0\to (\mscr{O}, d) \to (\mscr{E}, \nabla) \to (\mscr{O}, d) \to 0. 
    \end{equation}   
   By taking symmetric powers, we obtain a map   $\mscr{H}(X/S)\to \mscr{M}_{\dR}(X/S, r)$ for any $r\geq 2$. One can check that this map is a map of crystals, and that furthermore it is injective on tangent spaces away from the zero section of $\mscr{H}(X/S)$. The assumption then implies that the usual $p$-curvature of the Gauss-Manin connection on $R^1f_*\Omega^{\bullet}_{X/S, \dR}$ vanishes for infinitely many $p$, and Katz's theorem implies that the Kodaira--Spencer map of $X/S$ vanishes identically. It then follows from the Torelli theorem that $X/S$ is isotrivial. 

   On the other hand, it does not seem possible to make such arguments on the stable locus.   We are grateful to Sasha Petrov for this remark. 
\end{remark}

\subsection{Integral points}
\begin{lemma}\label{lemma:compact-and-discrete}
	Let $X$ be a finite-type affine $\mathbb{Z}$-scheme. Let $K$ be a number field and let $\iota_1, \cdots, \iota_m: K\hookrightarrow \mathbb{C}$ be the distinct embeddings of $K$ into $\mathbb{C}$. For each $i=1, \cdots, m$, fix $S_i\subset X(\mathbb{C})^{\an}$ a compact subset, and define $X(\mathscr{O}_K; S_1, \cdots, S_m)$ to be the set of $\mathscr{O}_K$-points $x$ of $X$ such that, for all $i=1, \cdots, m$, $\iota_i(x)$ lies in $S_i$. Then $X(\mathscr{O}_K; S_1, \cdots, S_m)$ is finite.
\end{lemma}
\begin{proof}
As $X$ is affine we may without loss of generality take it to be affine space $\mathbb{A}^n$. Thus we have compact sets $S_i\subset \mathbb{C}^n$ and wish to show that the set $x\in \mathscr{O}_K^n$ such that $\iota_i(x)$ lies in $S_i$ for all $i$ is finite.

	Consider the inclusion $$\mathscr{O}_K^n\hookrightarrow \prod_{\iota: \mathscr{O}_K\hookrightarrow \mathbb{C}} \mathbb{C}^n$$ given by the product of the embeddings $\iota$. This map has discrete image. On the other hand, $X(\mathscr{O}_K; S_1, \cdots, S_m)$ is the intersection of this image with the compact set $\prod_i S_i$, whence it is both discrete and compact, hence finite.
\end{proof}
We now prove a weak form of \autoref{thm:main-theorem-integral-points}. We will deduce the strong form in \autoref{subsec:esbt}.
\begin{theorem} \label{thm:main-theorem-integral-points-weak}
	Let $R\subset \mathbb{C}$ be a finitely-generated $\mathbb{Z}$-algebra, and $f: X\to S$ a smooth projective morphism of smooth $R$-schemes. Fix a point $s\in S(\mathbb{C})$ and a positive integer $r\in \mathbb{Z}_{>0}$. If the $p$-curvature of the isomonodromy foliation on $\mathscr{M}_{\dR}(X/S, r)$ mod $p$ vanishes for infinitely many primes $p$, then the orbit of any point of $M_B(X_s, r)(\overline{\mathbb{Z}})$ under the action of $\pi_1(S(\mathbb{C})^{\an},s)$ is finite.
\end{theorem}
\begin{proof}
By \autoref{thm:p-curvature-vanishing-implies-theta-vanishing}, we have that $\Theta_{X/S}$ vanishes identically. In fact this is all we will use below.

Let $K$ be a number field and $\mathscr{O}_K$ its ring of integers; let $\iota_1, \cdots, \iota_n: K\hookrightarrow \mathbb{C}$ be its distinct complex embeddings. Let $[\rho]\in M_B(X_s, r)(\mathscr{O}_K)$ be a point. For each embedding $\iota_i:K\hookrightarrow\mathbb{C}$, let $\rho_{\iota_i}$ be the corresponding semisimple complex representation.
	
For each $i$, set $S_i$ to be the closure of $\pi_1(S(\mathbb{C})^{\an},s)\cdot [\rho_{\iota_i}]$ in the analytic topology of $M_B(X_s, r)(\mathbb{C})$. The set $S_i$ is compact by \autoref{thm:compactness-thm} (using that $\Theta_{X/S}$ vanishes identically). On the other hand, the orbit of $[\rho_{\iota_i}]$ consists of $\mathscr{O}_K$-points for each $i$. Hence the orbit is finite by \autoref{lemma:compact-and-discrete}, as desired.
\end{proof}
\subsection{The Ekedahl--Shepherd-Barron--Taylor conjecture}\label{subsec:esbt}
\begin{lemma}\label{lemma:smooth-integral-point}
	Let $X/\mathbb{C}$ be a smooth projective variety and let $W\subset M_B(X, r)_{\overline{\mathbb{Z}}}$ be the closure of an irreducible component of the geometric generic fiber. Suppose $W(\overline{\mathbb{Z}})$ is non-empty. Then $W(\overline{\mathbb{Z}})$ is dense in $W(\mathbb{C})$.
\end{lemma}
\begin{proof}[Proof of \autoref{lemma:smooth-integral-point}]
Immediate from \cite[Theorem 1]{rumely1986arithmetic}.
 \end{proof} 
We learned of this application of Rumely's theorem \cite{rumely1986arithmetic} from Esnault's Park City lectures \cite{esnaultparkcity}, and it is closely related to \cite{de2024integrality}.

\begin{lemma}\label{lemma:riemannian-tangent}
Let $(M,g)$ be a connected Riemannian real-analytic manifold, and $G$ a group acting on $M$ by analytic isometries. If $G$ fixes a point $x\in M$, and acts trivially on $T_{M, x}$, then $G$ acts trivially on $M$.
\end{lemma}
\begin{proof}
	$G$ acts trivially on a small neighborhood of $x$ (as the exponential map is equivariant for the action of $G$). But now it must act trivially on all of $M$, by real-analyticity of the action.
\end{proof}
We are now ready to prove the main results of this article.
\begin{proof}[Proof of \autoref{thm:main-theorem-integral-points} and \autoref{thm:main-theorem-esbt}]
We prove that, if the $p$-curvature of the isomonodromy foliation on $\mathscr{M}_{\dR}(X/S, r)$ mod $p$ vanishes for infinitely many $p$, then for each component $W\subset M_B(X_s, r)_{\mathbb{C}}$ containing an $\overline{\mathbb{Z}}$-point, the action of $\pi_1(S(\mathbb{C})^{\an},s)$ on the orbit of $W$ factors through a finite group. This contains the statements of both \autoref{thm:main-theorem-integral-points} and \autoref{thm:main-theorem-esbt}.

	By assumption each such geometric component $W$ of $M_B(X_s, r)_{\mathbb{C}}$ has a $\overline{\mathbb{Z}}$-point, whence by \autoref{lemma:smooth-integral-point}, the  reduced locus of $W$ contains a \emph{smooth} $\overline{\mathbb{Z}}$-point. For each such component $W$, fix a smooth $\overline{\mathbb{Z}}$-point $[\rho_W]$ contained in $W$.
	
	By \autoref{thm:main-theorem-integral-points-weak}, the orbit of $[\rho_W]$ under $\pi_1(S(\mathbb{C})^{\an},s)$ is finite. Thus there exists a finite index subgroup $\Gamma\subset \pi_1(S(\mathbb{C})^{\an},s)$ fixing all the $[\rho_W]$, hence acting on $T_{M_{B,0}(X_s, r), [\rho_W]}$ for each $[\rho_W]$. By \autoref{thm:triholomorphic} this action preserves a Riemannian metric, whence it is unitary. 
	
	Now let $\{[\rho_W]^{\sigma}\}_\sigma$ be the set of Galois conjugates of $[\rho_W]$. The action of $\Gamma$ on $$\bigoplus_\sigma T_{M_{B,0}(X_s, r), [\rho_W]^\sigma}$$ is both unitary and integral, hence factors through a finite group by e.g.~\cite[Theorem 7.2.1]{ll-geometric}. Thus there exists a further finite index subgroup $\Gamma'\subset \Gamma$ acting trivially on $T_{M_{B,0}(X_s, r), [\rho_W]}$ for each $W$.
	
	But now the action of $\Gamma'$ on $W$ is trivial by \autoref{lemma:riemannian-tangent}, as desired.
	
	 The statement about algebraicity of leaves in \autoref{thm:main-theorem-esbt} follows from \cite[Theorem 7.7.3]{p-painleve}.
\end{proof}
In fact we have proven the following statement of pure complex geometry:
\begin{theorem}\label{thm:theta-vanishing}
	Let $f: X\to S$ be a smooth projective morphism of smooth complex varieties. Suppose that the lifting of tangent vectors $\Theta_{X/S}\equiv 0$ on $M_{\Dol}(X/S, r)$. Fix $s\in S$. Then
	\begin{enumerate}
		\item The action of $\pi_1(S(\mathbb{C})^{\an}, s)$ on $M_B(X_s, r)(\overline{\mathbb{Z}})$ factors through a finite group.
		\item If each irreducible component of $M_B(X_s, r)$ contains a $\overline{\mathbb{Z}}$-point, then the action of $\pi_1(S(\mathbb{C})^{\an}, s)$ on $M_B(X_s, r)(\mathbb{C})$ factors through a finite group.
	\end{enumerate}
\end{theorem}

We summarize some cases in which we have proven the Ekedahl--Shepherd-Barron--Taylor conjecture, \autoref{conj:esbt}, for isomonodromy foliations:
\begin{corollary}\label{cor:main-theorem-esbt-cases}
		Let $f: X\to S$ be a smooth projective morphism, and fix a point $s\in S(\mathbb{C})$. Suppose that the $p$-curvature of the isomonodromy foliation on $\mathscr{M}_{\dR}(X/S, r)$ vanishes for infinitely many $p$, and either 
	\begin{enumerate}
		\item $r=2$, or
		\item $M_B(X_s,r)_{\mathbb{C}}$ is irreducible, or
	   \item $f$ has relative dimension $1$.

	\end{enumerate}
	Then the action of $\pi_1(S(\mathbb{C})^{\an},s)$ on $M_B(X_s, r)(\mathbb{C})$ factors through a finite group.
\end{corollary}
\begin{proof}
In cases (1) and (2) we claim the hypotheses of \autoref{thm:main-theorem-esbt} are satisfied. Indeed, in the case $M_B(X_s, r)_\mathbb{C}$ is irreducible the trivial representation gives rise to a $\overline{\mathbb{Z}}$-point. In the case $r=2$, the density of $\overline{\mathbb{Z}}$-points follows for example from \cite{coccia2025density} (where it ultimately comes down to Corlette-Simpson's classification of rank $2$ local systems on quasi-projective varieties \cite{corlette2008classification}).

In case (3), this follows from \autoref{thm:relative-curves}.
\end{proof}
\begin{remark}\label{rem:coccia}
	In \cite[Conjecture 1.1.1]{coccia2025density}, Coccia and the second-named author conjecture that the hypothesis of \autoref{thm:main-theorem-esbt} on existence of integral points is always satisfied, and in fact that integral points are potentially dense in the Betti moduli space.
\end{remark}
\subsection{Motivic points}\label{subsec:motivic-points}
In this section we sketch an alternate proof of a special case of \autoref{thm:main-theorem-esbt}, using the results of this paper and our results from \cite{p-painleve}. We are grateful to Michael Groechenig for asking if an approach along these lines might work.
\begin{theorem}\label{thm:motivic-points}
With notation and assumptions as in \autoref{thm:main-theorem-integral-points}, let $W\subset M_B(X_s, r)$ be an irreducible component. Suppose $[\rho]\in W$ is an irreducible Gauss-Manin connection on $X_s$, i.e.~there exists a smooth projective morphism $g: Y\to X_s$ such that the irreducible representation $\rho_{\mathbb{C}}$ is the monodromy representation associated to $R^ig_*\mathbb{C}$ for some $i$. Finally, suppose that the $p$-curvature of the isomonodromy foliation on $\mathscr{M}_{\dR}(X/S, r)$ mod $p$ vanishes for almost all $p$.

Then the orbit of each point of $W(\mathbb{C})$ under the action of $\pi_1(S(\mathbb{C})^{\an},s)$ is finite.
\end{theorem}
\begin{remark}
\autoref{thm:motivic-points} is weaker than \autoref{thm:main-theorem-esbt}	in a number of ways: it requires $p$-curvature vanishing for almost all $p$ (as opposed to only infinitely many), it requires the $[\rho]$ be irreducible, and most seriously, it requires that $[\rho]$ be a Gauss-Manin connection.
\end{remark}

\begin{proof}[Sketch of proof of \autoref{thm:motivic-points}]
Let $(\mathscr{E},\nabla)$ be the flat bundle associated to $\rho_{\mathbb{C}}$. By the assumption on the vanishing of the $p$-curvature of the isomonodromy foliation, the leaf of the isomonodromy foliation through $(\mathscr{E},\nabla)$ is $\omega(p)$-integral in the sense of \cite{p-painleve}\footnote{This is a general feature of leaves of foliations with vanishing $p$-curvature. It follows for example from the formal power series expansion of the leaf as in \cite[(3.7) and 3.4.2]{bost-foliations}.}; as it is a Gauss-Manin connection, this leaf is thus algebraic by the main theorem of that paper, \cite[Theorem 1.2.4]{p-painleve}. In particular the conjugacy class of  $\rho_{\mathbb{C}}$ has finite $\pi_1(S(\mathbb{C})^{\an},s)$-orbit.

By replacing $S$ with a finite \'etale cover we may thus assume that $[\rho_{\mathbb{C}}]$ is fixed by $\pi_1(S(\mathbb{C})^{\an},s)$ and hence yields an algebraic section $t$ of $M_{\dR}(X/S)$, classified by some flat bundle $(\widetilde{\mathscr{E}},\widetilde{\nabla})$ on $X$. The completion of $M_B(X/S, r)$ along $t$ has structure sheaf $\widehat{\mathscr{O}}$, a pro-flat vector bundle on $S$ with connection; modulo $p$ this connection has vanishing $p$-curvature for almost all $p$ by assumption. Let $\mathscr{I}\subset\widehat{\mathscr{O}}$ be the ideal sheaf cutting out $t$. As $(\mathscr{E},\nabla)$ is irreducible, $\mathscr{I}/\mathscr{I}^2$ is isomorphic to $R^1f_*(\on{End}(\widetilde{\mathscr{E}})_{\dR})$. This last is, as follows from the proof of \cite[Theorem 1.2.4]{p-painleve} (as explained in \cite[\S14]{p-painleve}), a $\mathbb{Z}$-VHS and Higgs-de Rham fixed point, which are precisely the hypotheses needed to run Katz's proof of the $p$-curvature for Gauss-Manin connections \cite{katz-p-curvature}. Hence the flat bundle $\mathscr{I}/\mathscr{I}^2$ has finite monodromy. By passing to a cover we may assume it is in fact trivial as a flat bundle, whence the same is true for $\mathscr{I}^n/\mathscr{I}^{n+1}$ for all $n$.

Thus for each $n$, $\widehat{\mathscr{O}}/\mathscr{I}^n$ is a flat bundle with vanishing $p$-curvatures and unipotent monodromy. But the $p$-curvature conjecture is known for such bundles (by e.g. \cite{chudnovsky2006applications, andre2004conjecture, bost-foliations}); hence it has trivial monodromy, as a unipotent representation with finite image is trivial. Equivalently, the action of $\pi_1(S(\mathbb{C})^{\an},s)$ on the formal local ring of $M_B(X_s, r)$ at $[\rho_{\mathbb{C}}]$ is trivial, whence it is trivial on all of $W$.
\end{proof}
\begin{remark}
	It is natural to try to run the argument above without the assumption that $\rho_{\mathbb{C}}$ is irreducible. This seems (to us) challenging: on the coarse moduli space $M_B$ one cannot identify the tangent space in this case, while at the stacky level we do not see how to use the results by \cite{chudnovsky2006applications, andre2004conjecture, bost-foliations} to deduce finiteness. It would be nice to find a way to execute this argument, however, as it would give a more elementary proof of a weak form of \autoref{cor:main-theorem-esbt-cases}(2) by applying it to the case of the trivial representation.
\end{remark}
\begin{remark}
It is natural to try to replace the argument from  \cite{chudnovsky2006applications, andre2004conjecture, bost-foliations}, or the use of the K\"ahler metric in the proof of \autoref{thm:main-theorem-esbt}, by the use of Goldman-Millson theory to try to deduce finiteness of the $\pi_1(S(\mathbb{C})^{\an},s)$-action on (a component of) $M_B(X_s, r)$ from the analogous statement on tangent spaces, as Goldman-Millson theory identifies a formal neighborhood of $[\rho]$ in $\mathscr{M}_B(X_s, r)_{\mathbb{C}}$ with a quadratic cone in its tangent space. Unfortunately this identification depends on the complex structure on $X_s$ and in general is not $\pi_1(S(\mathbb{C})^{\an},s)$-invariant.

For an example, one may consider the universal curve $\mathscr{C}_g\to \mathscr{M}_g$. The Torelli group acts trivially on the tangent space to the trivial representation in $\mathscr{M}_B(\Sigma_{g,n}, r)$. On the other hand, in general it acts nontrivially on the character variety and hence on the formal local ring at the trivial representation.
\end{remark}

\bibliographystyle{alpha}
\bibliography{bibliography-esbt}

\end{document}